\documentclass[11pt,twoside]{article}
\usepackage{amsmath, amsthm, amsfonts, amssymb}
\usepackage{mathrsfs}
\usepackage[misc]{ifsym}
\usepackage[colorlinks,linkcolor=blue,anchorcolor=blue,citecolor=blue,urlcolor=black]{hyperref}
\usepackage{graphicx}
\usepackage{subfigure}
\usepackage{tikz}
\usepackage{enumerate}
\usepackage{anysize}
\usepackage{setspace}
\usepackage{float}
\usepackage{color}
\usepackage{chngcntr}
\everymath{\displaystyle}
\allowdisplaybreaks

\newfam\msbfam
\font\tenmsb=msbm5    \textfont\msbfam=\tenmsb \font\sevenmsb=msbm5
\scriptfont\msbfam=\sevenmsb \font\fivemsb=msbm5
\scriptscriptfont\msbfam=\fivemsb

\newfam\bigfam
\font\tenbig=msbm5 scaled \magstep2   \textfont\bigfam=\tenbig
\font\sevenbig=msbm7 scaled \magstep2 \scriptfont\bigfam=\sevenbig
\font\fivebig=msbm5 scaled \magstep2
\scriptscriptfont\bigfam=\fivebig

\numberwithin{equation}{section}
\newtheorem{theorem}{Theorem}[section]
\newtheorem{lemma}{Lemma}[section]

\newtheorem{remark}{Remark}[section]

\newtheorem{definition}{Definition}[section]

\begin{document}
\title{\bf Multilinear Commutators of Multilinear Square Operators Associated with  New $BMO$ Functions and   New Weight Functions}
\author{\bf Chunliang Li,  Shuhui Yang$^*$ and Yan Lin}

\renewcommand{\thefootnote}{}
\date{}
\maketitle
\footnotetext{This work was partially supported by the National 
Natural Science Foundation of China No. 12071052.}
\footnotetext{$^*$Corresponding author, E-mail:yangshuhui@student.cumtb.edu.cn}
\begin{minipage}{13.5cm}
	
{\bf Abstract}
\quad
Via the new weight $ A_{\vec p}^{\infty } (\varphi ) $ 
and the new $ BMO $ function, the authors introduce a new class of 
multilinear square operators $ T $ with generalized kernels. The boundedness of 
multilinear commutators and multilinear iterative commutators 
generated by $ T $ and the new $ BMO $ function on weighted Lebesgue spaces 
and weighted Morrey spaces is obtained, respectively. The results of this 
article contain some known conclusions.
\medskip
\end{minipage}

\section{Introduction}

\quad\quad The weighted inequality plays a very significant role in harmonic 
analysis. It is of great consequence in the estimation of nonlinear 
partial differential equations and vector-valued inequalities. Numerous 
scholars have developed a great research interest in the weighted inequality.   
Muckenhoupt \cite{b17} first proposed the classical Muckenhoupt’s $ A_p $ 
weight functions and established weighted norm inequalities for Hardy 
maximal functions. 

For $1<p<\infty$, $ 1/p+1/p'=1 $, $\omega $ is a nonnegative and locally 
integrable function. The weight $\omega $ is said to be in the class $ A_p $ 
if there holds
$$\left(\frac{1}{|Q|}\int_Q\omega(y)dy\right)^{\frac{1}{p}}
\left(\frac{1}{|Q|}\int_Q\omega(y)^{\frac{-1}{p-1}}dy\right)^
{\frac{1}{p'}}\leq C,$$
for all cube $ Q $. When $ p=1 $, we have
$$\left(\frac{1}{|Q|}\int_Q\omega(y)dy\right)\leq C\mathop{\inf}\limits_{x\in 
	{ Q}}\omega(x).$$
%Based on the above classical Muckenhoupt’s $ A_p $ weight functions. 
Garc\'{i}a-Cuerva and Rubio \cite{b7} studied the 
classical Muckenhoupt’s $ A_p $ weight functions and some of the associated 
properties. Tang \cite{b24} introduced a new class of weighted functions $ 
A_{p}^{\infty }(\varphi )$ and obtained weighted norm inequalities for 
pseudo-differential operators with smooth symbols and their commutators. For 
some other studies of $ A_{p}^{\infty }(\varphi )$, one can see  
\cite{b9,b10,b14} and related references.

For $ 1< p< \infty $, $ 1/p+1/p'=1$ and $\theta\geq0$, $\omega $ is a 
nonnegative and locally integrable function. We say that $\omega \in 
A_p^{\theta}(\varphi)$ if there holds
$$\left(\frac{1}{\varphi(Q)^{\theta}|Q|}\int_Q\omega(y)dy\right)^{\frac{1}{p}}
\left(\frac{1}{\varphi(Q)^{\theta}|Q|}\int_Q\omega(y)^{\frac{-1}{p-1}}dy\right)^
{\frac{1}{p'}}\leq C,$$
for all cubes $Q$. When $p=1$, we conclude
$$\left(\frac{1}{\varphi(Q)^{\theta}|Q|}\int_Q\omega(y)dy\right)\leq 
C\mathop{\inf}\limits_{x\in { Q}}\omega(x),$$
where $Q(x,r)$ is the cube centered at $x$ with the sidelength $r$ and  
$\varphi(Q):=1+r$.

\begin{remark}
Let $A_p^{\infty}(\varphi):= {\textstyle \bigcup_{\theta \ge 0}} 
A_p^{\theta}(\varphi)$, and $A_{\infty}^{\infty}(\varphi):= {\textstyle 
\bigcup_{p\ge 1}} A_p^{\infty}(\varphi)$. When $ \theta =0$, the class $ 
A_p^0(\varphi) $ and $ A_p $ are equivalent for all $1\leq p<\infty$. In 
general, the class $A_p$ is strictly smaller than the class  
$A_p^{\theta}(\varphi)$ for all $1\leq p<\infty$. Taking for instance, $ x\in 
\mathbb{R}^n $, $ \omega(x):=1+|x|^\alpha  $ and $ \alpha >n(p-1) $, the 
weight belongs to $ A_{p}^{\infty } (\varphi )$, but it does not belong to $ A_p 
$.
\end{remark}
Lerner et. al. \cite{b12} introduced a new class of weighted 
functions $ A_{\vec p}$. For some other studies of $ A_{\vec p}$, one can refer 
to \cite{b15,b21} and so on.

Let $\vec{p}=(p_1,\ldots,p_m)$, $1/p= {\textstyle \sum_{j=1}^{m}}  1/p_j$, and  
 $1\leq p_1,\ldots,p_m<\infty$. Given 
$\vec{\omega}=(\omega_1,\ldots,\omega_m)$, each $ \omega_j $ being 
nonnegative measurable, set $v_{\vec{\omega}}:=  {\textstyle 
\prod_{j=1}^{m}}\omega_j^{p/p_j}.$ We say that $\vec\omega\in A_{\vec{p}}$, if 
$$\mathop{\sup}\limits_{
	Q}\left(\frac{1}{|Q|}\int_Qv_{\vec{\omega}}(x)dx\right)^{\frac1p}
\prod_{j=1}^{m}\left(\frac{1}{|Q|}\int_Q\omega_j(x)^{1-p'_j}dx\right)^{\frac{1}
	{p'_j}}<\infty,$$
where the supremum is taken over all cubes $Q\subset \mathbb{R}^n$, and the 
term $\left(\frac{1}{|Q|}\displaystyle \int_Q 
\omega_j(x)^{1-p'_j}dx\right)^{\frac{1}{p'_j}}$ is understood by  
$\left(\mathop{\inf}\limits_{x\in { Q}}\omega_j(x)\right)^{-1}$ 
when $p_j=1$, $j=1,\ldots, m$. 

Bui \cite{b2} introduced a new class of multiple weights $ A_{\vec 
p}^{\infty } (\varphi )$. For some other studies of $ A_{\vec p}^{\infty } 
(\varphi ) $, one can see, for instance, \cite{b19,b28}.

Let $\vec{p}=(p_1,\ldots,p_m)$, $\theta\geq0$ and $1/p= {\textstyle 
\sum_{j=1}^{m}}  1/p_j$ with $1\leq p_1,\ldots,p_m<\infty$. Given 
$\vec{\omega}=(\omega_1,\ldots,\omega_m)$, each $ \omega_j $ being 
nonnegative measurable, set $v_{\vec{\omega}}:=  {\textstyle 
\prod_{j=1}^{m}}\omega_j^{p/p_j}.$ We say that $ \vec\omega $ satisfies the $ 
A_{\vec{p}}^{\theta}(\varphi) $ 
condition and write $\vec\omega\in A_{\vec{p}}^{\theta}(\varphi)$, if
$$\mathop{\sup}\limits_{ Q} 
\left(\frac{1}{\varphi(Q)^{\theta}|Q|}\int_Qv_{\vec{\omega}}(x)dx\right)^{\frac1p}
\prod_{j=1}^{m}\left(\frac{1}{\varphi(Q)^{\theta}|Q|}\int_Q\omega_j(x)^{1-p'_j}
dx\right)^{\frac{1}{p'_j}}<\infty,$$
where the supremum is taken over all cubes $Q\subset \mathbb{R}^n$. When 
$p_j=1$, $j=1,\ldots, m$, the term 
$\left(\frac{1}{|Q|}\displaystyle \int_Q  
\omega_j(x)^{1-p'_j}dx\right)^{\frac{1}{p'_j}}$ is understood by  
$\left(\mathop{\inf}\limits_{x\in { Q}}\omega_j(x)\right)^{-1}$. 
\begin{remark}
For $1\leq p_1,\ldots,p_m<\infty$, we write 
$A_{\vec{p}}^{\infty}(\varphi):= {\textstyle \bigcup_{\theta \ge 0}} 
A_{\vec{p}}^{\theta}(\varphi)$. When $\theta=0$, $A_{\vec{p}}^{0}(\varphi)$ and 
$A_{\vec{p}}$ are equivalent, where the $A_{\vec{p}}$ weighted function has 
been studied in {\rm \cite{b12}}.
\end{remark}

Based on the weight function described above, a large amount of scholars have 
studied a class of multilinear Calder\'{o}n-Zygmund operators. Pan and Tang 
\cite{b19} obtained the boundedness of certain classes of multilinear  
Calder\'{o}n-Zygmund operators with classical kernels and their iterated 
commutators with new $ BMO $ functions (see Definition \ref{definition1.1}) on 
weighted Lebesgue spaces with   
$A_{\vec{p}}^{\infty}(\varphi)$ weights. Zhao and Zhou \cite{b29} introduced a class 
of multilinear Calder\'{o}n-Zygmund operators with kernels of Dini's type. They 
obtained strong and weak boundedness of multilinear Calder\'{o}n-Zygmund 
operators with kernels of Dini's type on weighted Lebesgue spaces, as well as 
weighted norm inequalities for iterative commutators generated by multilinear 
Calder\'{o}n-Zygmund operators and $ BMO $ functions with 
$ A_{\vec{p}}^{\infty }(\varphi) $ weights.
\begin{definition}{\rm(\cite{b1})}\label{definition1.1}
The new $ BMO $ space $  BMO_{\theta  }(\varphi ) $ with $ \theta \ge 0 $  is 
defined as a set of all locally integrable functions $ b $ satisfying,
\begin{align*}
\Vert 
b\Vert_{BMO_{\theta}(\varphi)}:=\mathop{\sup}\limits_Q\frac{1}{\varphi(Q)^{\theta}
|Q|}\int_Q|b(y)-b_Q|dy<\infty.
\nonumber
\end{align*}
\end{definition}
When $ \theta=0 $, 
$BMO_\theta(\varphi)=BMO(\mathbb{R}^n)$. $BMO(\mathbb{R}^n) 
\subset BMO_{\theta}(\varphi)$ and 
$BMO_{{\theta}_1}(\varphi) \subset BMO_{{\theta}_2}(\varphi)$ for 
${\theta}_1\leq {\theta}_2$. Note $BMO_{\infty}(\varphi):=  {\textstyle 
\bigcup_{\theta \ge 0}} BMO_{\theta}(\varphi)$.

Let $ \vec \theta =(\theta _1,\dots ,\theta _m) $ with $ \theta _1,\dots 
,\theta _m\ge 0 $, and $ \vec b=(b_1,\dots ,b_m) $. If $ b_j\in BMO_{\theta 
_j}(\varphi ) $ for $ 1\le j\le m $, then $ \vec b\in BMO_{\vec \theta }^{m} 
(\varphi ) $.

%As an important direction of harmonic analysis, the theory of 
%multilinear square operators has attracted more and more attention. 
%The multilinear theory has been studied by many scholars, including Chaffee, 
%Hart and Oliveira \cite{b4}, Grafakos, Mohanty and Shrivastava [16], Cao and 
%Hormozi et al [2].
%In \cite{b27}, Sato and Yabuta showed that $ T_{g} $ is bounded 
%from $ L^{p_{1} } \times \cdots \times L^{p_{m}  } $ to $ L^{p} $ for 
%multilinearized Littlewood--Paley operators,
%\begin{align*}
%T_{g}(\vec{f})(x)=\int_{0}^{\infty} 
%\prod_{i=1}^{m}\left(\left(\phi_{i}\right)_{t} * f_i\right)(x) \frac{dt}{t}.
%\end {align*}
%In \cite{b34}, Xue, Peng and Yabuta established the strong and weak estimates 
%of the multilinear Littlewood–Paley g-function with convolution type and its 
%commutator in weighted Lebesgue space,
%$$g(\vec f)(x)=\left ( \int_{0}^{\infty  }\left | \frac{1}{t^{mn}} 
%\int_{\left(\mathbb{R}^{n}\right)^{m}} \psi \left ( \frac{y_1}{t},\cdots 
%,\frac{y_m}{t}  \right )\prod_{j=1}^{m}f_j(x-y_j) dy_j  \right | 
%^2\frac{dt}{t}  \right ) ^{\frac{1}{2} }.$$ 
%In \cite{b19}, Li and Song studied weighted estimates for vector-valued
%multilinear square function $ T $(see(1.1)) and also established multiple 
%weighted inequalities for their iterated commutator generated by the 
%vector-valued multilinear operator $ T $ and $ BMO $ function.

Via $A_{\vec p}^{\theta }(\varphi )$ weight functions and the standard 
multilinear square operator, we first introduced the 
new multilinear square operator. The new multilinear square operator with 
classical kernel is defined as follows. For any $ t\in (0,\infty ) $, let $ 
K_t(y_0,y_1,\ldots,y_m) $ be a locally integrable function defined away from 
the diagonal $y_0=y_1=\ldots=y_m$. For some 
positive constants $ C $ and $\gamma  $, any $ N\ge 0 $,  it satisfies the 
following size condition,
\begin{equation}\label{1.1}	
\left (  \int_{0}^{\infty}|K_{t}(y_0, y_{1}, \ldots, y_{m})|^{2} \frac{d 
t}{t}\right ) ^{1 / 2} \leq 
\frac{C}{(\sum_{j=1}^m|y_0-y_j|)^{mn}(1+\sum_{j=1}^m|y_0-y_j|)^N},
\end{equation}
and the smoothness condition
\begin{align*}
&\left(\int_{0}^{\infty}\left|K_{t}\left(y_0, y_{1},\ldots, 
y_{m}\right)-K_{t}\left(y_0', y_{1},\ldots, y_{m}\right)\right|^{2} 
\frac{d t}{t}\right)^{1 / 2} \\
& \quad \leq \frac{C|y_0-y_0'|^{\gamma }}{(\sum_{j=1}^m|y_0-y_j|)^{mn+\gamma 
}(1+\sum_{j=1}^m|y_0-y_j|)^N},
\end{align*}
whenever $|y_0-y_0'|\leq\frac{1}{2}\max_{1\leq j \leq m}|y_0-y_j|$.

Suppose that $\omega:[0,\infty) \mapsto  [0,\infty)$ is a nondecreasing 
function with $0<\omega(t)<\infty$ for all $t\in [0,\infty )$. For $a>0$, one says that $\omega\in {\rm 
	Dini}(a)$, if
\begin{align*}
[\omega]_{{\rm Dini}[a]}:=\int^{1}_{0}\frac{\omega^a(t)}{t}dt<\infty.  
\end{align*}

The new multilinear square operator with kernel of Dini's type is defined as 
follows, where the size condition \eqref{1.1} of kernel functions  remains 
unchanged and the smoothness condition is changed by
\begin{equation}\label{1.2}	
\begin{aligned}
&\left(\int_{0}^{\infty}\left|K_{t}\left(y_0, y_{1},\ldots, 
y_{m}\right)-K_{t}\left(y_0', y_{1}, \ldots, y_{m}\right)\right|^{2} 
\frac{d t}{t}\right)^{1 / 2} \\
&  \quad \leq 
\frac{C}{(\sum_{j=1}^m|y_0-y_j|)^{mn}(1+\sum_{j=1}^m|y_0-y_j|)^N}\omega 
\left( \frac{|y_0-y_0'|}{\sum_{j=1}^m |y_0-y_j|}\right),
\end{aligned}
\end{equation}
whenever $|y_0-y_0'|\leq\frac{1}{2}\max_{1\leq j \leq m}|y_0-y_j|$.
\begin{remark}
	When  $ \omega (t):=t^\gamma   $ for some $ \gamma  > 0$ in the condition 
	\eqref{1.2},  the new multilinear 
	square operator with kernel of 
	Dini’s type is the new multilinear square operator with classical kernel in 
	{\rm \cite{b27}}.
\end{remark}
The size condition of kernel functions remains unchanged and the  smoothness 
condition of kernel functions can be further weakened.  
\begin{definition}\label{definition1.3}
Let $ 
K_t(x,y_1,\ldots ,y_m):=(1/t^{mn})K(x/t,y_1/t,\ldots,y_m/t)$, for any $ 
t\in (0,\infty )$ and $K(y_0,y_1,\ldots,y_m)$ be a locally integrable function defined away from the diagonal $y_0=y_1=\ldots=y_m$ in $(\mathbb{R}^n)^{m+1}$. The new multilinear square function is defined by
\begin{equation}\label{1.3}	
\begin{aligned}
T(\vec{f})(x):={\left( {\int_0^\infty  {{{\left|{\int_{{{\left({\mathbb{R}^n} 
\right)}^m}} {{K_t}\left( {x,{y_1},\ldots ,{y_m}} \right) \cdot 
\prod_{j=1}^{m}f_{j}\left(y_{j}\right)d{\vec y}} }
\right|}^2}\frac{{dt}}{t}} } \right)^{\frac{1}{2}}},
\end{aligned}
\end{equation}

\noindent whenever $x\notin  {\textstyle \bigcap_{j=1}^{m}}  {\rm supp } f_j$ 
and each $ f_{j} \in C_{c}^{\infty}\left(\mathbb{R}^{n}\right) $.

If the following conditions are satisfied, $T$  is called the new multilinear 
square operator with generalized kernel.

\begin{itemize}
\item [\rm (i)]There is a positive constant $ A>0 $ such that for any $N\geq 
0$, the kernel function satisfies the size condition, 
\begin{align}\label{1.4}
&\left (  \int_{0}^{\infty}|K_{t}(y_0, y_{1}, \ldots, y_{m})|^{2} \frac{d 
t}{t}\right ) ^{1 / 2} \leq 
\frac{A}{(\sum_{j=1}^m|y_0-y_j|)^{mn}(1+\sum_{j=1}^m|y_0-y_j|)^N}.
\end{align}

\item [\rm (ii)]There are positive constants $C_k>0$, $k\in \mathbb{N}_+$, $ 1< q\le 2 $, 
$1/q+1/q'=1$ such that for any $N\geq 0$, the kernel function satisfying the 
smoothness condition,
\begin{align}\label{1.5}	
&\left(\int_{0}^{\infty}\left(\int_{\left(\Delta_{k+2}\right)^{m} 
\backslash\left(\Delta_{k+1}\right)^{m}}\left|K_{t}\left(y_{0}, 
y_{1}, \ldots, y_{m}\right)-K_{t}\left(y_{0}^{\prime}, y_{1}, \ldots, 
y_{m}\right)\right|^{q} d \vec{y}\right)^{\frac{2}{q}} \frac{d 
t}{t}\right)^{\frac{1}{2}}\nonumber\\
& \quad\leq CC_k|y_0-y_o'|^{-\frac{mn}{q'}}\left(1+2^k |y_0-y_0'| 
\right)^{-N}2^{-\frac{kmn}{q'}},
\end{align} 
where $\Delta_k:=Q(y_0, 2^{k}\sqrt{mn}|y_0-y_0'|)$.
\item [\rm (iii)]$T$ can be extended to be a 
bounded operator from $L^{s_1}\times\cdots\times L^{s_m}$ to $L^{s,\infty}$ for 
some $1\le s_j\le q' ,j=1,\ldots ,m $ with	$1/s=1/s_1+\ldots +1/s_m$.
\end{itemize}
\end{definition}

\begin{remark}
The class of generalized kernels includes the kernels of Dini’s type, when 
$C_k:=\omega(2^{-k})$. One can see {\rm \cite{b13}} for details. 
\end{remark}
%\begin{remark}
%When $ N=0 $ in the condition \eqref{1.4}-\eqref{1.5}, the new multilinear 
%square operator is equivalent to the multilinear square operator in {\rm 
%\cite{b5}} 
%\end{remark}
%The classical Muckenhoupt’s class weight functions has attracted the great 
%interest of many scholars. In 2004, Cruz-Uribe and Martell et al \cite{b7} 
%proved weak endpoint inequalities starting from strong-type inequalities, 
%where 
%the weights are $ A_p $. In 2009, Lerner and Ombrosi \cite{b17} established 
%multiple 
%weighted norm inequalities of multilinear Calder\'{o}n-Zygmund operators, 
%where 
%the weights are $ A_{\vec p} $, we refer to \cite{b8,b18,b26}. In 2015, Bui 
%\cite{b3} studied a new class of multiple weights $ A_{\vec p}^{\infty } 
%(\varphi ) $. In 2015, 
%Pan and Tang \cite{b24} obtained the pointwise estimates, strong type and weak 
%end-point estmates for certain classes of multilinear operators and their 
%iterated commutators with new $ BMO $ functions, where the weights are $ 
%A_{\vec p}^{\infty } (\varphi ) $. 
%In 2021, applying the new class of multiple weights $ A_{\vec p}^{\theta } 
%(\varphi ) $, Zhao and Zhou \cite{b36} proved some weighted norm inequalities 
%for certain classes of multilinear operators in the Morrey-type spaces.  
%For some other studies of $A_p^{\infty}(\varphi)$ and $A_{\vec 
%p}^{\infty}(\varphi)$, see \cite{b20,b30,b35}.

Xue and Yan \cite{b27} obtained the strong $ (L^{p_1}(\omega _1)\times 
\cdots\times L^{p_m}(\omega _m),L^{p}(\nu  _{\vec \omega }) ) $ and weak $ 
(L^{p_1}(\omega _1)\times \cdots\times L^{p_m}(\omega _m),L^{p,\infty }(\nu  
_{\vec \omega }) ) $ type boundedness of multilinear square operators (see 
\eqref{1.1}) with classical kernels on weighted Lebesgue spaces. 
Si and Xue \cite{b22} proved weighted norm inequalities for multilinear 
square functions with kernels of Dini’s type with $ 
A_{\vec p} $ weights.
%In [0], authors further attenuated the conditions for the Dini’s kernel and 
%the 
%multilinear square operator $ T $ in \cite{b28}, authors introduced a new 
%class 
%of multilinear square operators $ T $ with generalized kernel, 
%they established strong and weak type boundedness in weighted Lebesgue spaces 
%for the new multilinear square operator $ T $ with generalized kernels, where 
%the weights are $ A_{\vec p}^{\infty }(\varphi ) $. 
For related research on  the theory of multilinear square functions, one can 
see \cite{b3,b4,b5,b8} and related references.

Inspired by the above new weight functions $A_{\vec p}^{\theta 
}(\varphi )$, new $ BMO $ functions and multilinear square operators. 
We introduce the new class of multilinear square operators (see Definition 
\ref{definition1.3}) with generalized kernels, and establish the boundedness of 
multilinear commutators and  multilinear iterative commutators generated by new 
multilinear square operators and new $ BMO $ functions on the weighted 
Lebesgue spaces and new Morrey spaces, respectively.

\begin{definition}
Let $\vec{b}=(b_1,\ldots,b_m)$ be a family of locally integrable functions.  
The multilinear commutator $ T_{\sum \vec{b}} $  
generated by new multilinear square operators are defined as follows.
\begin{equation}
T_{\sum \vec{b}}(\vec{f})(x)=\sum_{j=1}^mT_{b_j}^j(\vec{f})(x):=\sum_{j=1}^mT\left ( f_1,\dots,\left ( b_j(x)-b_j \right )f_j,\dots ,f_m   \right ) (x).
\nonumber
\end{equation}
	
The multilinear iterative commutator $ T_{\prod\vec{b}} $ generated by new 
multilinear square operators are defined as follows.
\begin{equation}
T_{\prod\vec{b}}(\vec f)(x):=T\left ( \left ( b_1(x)-b_1 \right )f_1,\dots,\left ( b_m(x)-b_m \right )f_m  \right ) (x).
\nonumber
\end{equation}
%where $ \vec f=(f_1,\ldots,f_m)$. If $ T $ is associated in the usual way with 
%a kernel $K_t$ satisfying the conditions \eqref{1.4} and \eqref{1.5}, we 
%can write
%\begin{equation}
%T_{\prod\vec{b}}(\vec f):=\left ( \int_{0}^{\infty}\left | 
%\int_{(\mathbb{R}^n)^m}\prod_{j=1}^m\left(b_j(x)-b_j(y_j)\right)K_t(x,y_1,\ldots,y_m)\prod_{j=1}^{m}f_j(y_j)
%d\vec{y} \right | ^2\frac{dt}{t}   \right )^{\frac{1}{2} },
%\nonumber
%\end{equation}
%and
%\begin{equation}
%T_{\sum \vec{b}}(\vec f):=\left ( \int_{0}^{\infty}\left | 
%\int_{(\mathbb{R}^n)^m}\sum_{j=1}^{m} 
%\left(b_j(x)-b_j(y_j)\right)K_t(x,y_1,\ldots,y_m)\prod_{j=1}^{m}f_j(y_j) 
%d\vec{y} \right | ^2\frac{dt}{t}   \right )^{\frac{1}{2} }.
%\nonumber
%\end{equation}
\end{definition}

Morrey \cite{b16} first introduced the classical Morrey space to research the 
local behavior of solutions to second-order elliptic partial differential 
equations. Subsequently, the problem of studying the boundedness of operators 
on Morrey spaces has attracted great interest from many authors. Tang and Dong 
\cite{b25} obtained the boundedness of some Schr\"{o}dinger type operators on 
Morrey spaces related to certain nonnegative 
potentials. Komori and Shirai \cite{b11} studied boundedness of the 
singular integral operator on weighted Morrey spaces. Pan and Tang \cite{b18} 
obtained boundedness of some Schr\"{o}dinger type operators on weighted 
Morrey spaces. Trong and Truong \cite{b26} established the weighted norm 
inequalities for the Riesz transforms associated to Schr\"{o}dinger operators.  
Zhao and Zhou \cite{b28} studied new weighted norm inequalities for certain 
classes of multilinear operators on Morrey spaces. Next, the definition of 
weighted Morrey spaces and the new weighted Morrey space are given, 
respectively.

\begin{definition}{\rm(\cite{b26})} 
Suppose that $ u $, $ w $ are two weights, $ \lambda \in [0,1) $, $ 1\le p< 
l\le \infty  $, $ \alpha \in (-\infty ,\infty ) $, and $ Q:=Q(z,r)$. The 
strong Morrey space $ M_{\alpha ,\lambda }^{p,l} (u,w) $ consists of  
 any measurable function $ f $ satisfying $ \left \| f \right \| 
_{M_{\alpha ,\lambda }^{p,l} (u,w)} < \infty $, where
\begin{align*}
\|f\|_{M_{\alpha, \lambda}^{p, l}(u, w)}:=\sup 
_{r>0}\left[\int_{\mathbb{R}^{n}}\left(\varphi  
(Q(z,r))^{\alpha}u(Q(z,r))^{-\lambda}\left\|f\chi_{Q(z,r)}\right\|_{L^{p}(w)}\right)^{l} 
d z\right]^{1 / l}.
\end{align*}
The weak Morrey space $ WM_{\alpha ,\lambda }^{p,l} (u,w) $ consists of any measurable function $ f $ satisfying $ \left \| f \right \| 
_{WM_{\alpha ,\lambda }^{p,l} (u,w)} < \infty $, where
\begin{align*}
\|f\|_{WM_{\alpha, \lambda}^{p, l}(u, w)}:=\sup 
_{r>0}\left[\int_{\mathbb{R}^{n}}\left(\varphi  
(Q(z,r))^{\alpha}u(Q(z,r))^{-\lambda}\left\|f \chi_{Q(z,r)}\right\|_{L^{p,\infty 
}(w)}\right)^{l}dz\right]^{1 / l}.
\end{align*}
\end{definition}
In particular, when $ \alpha =0 $ and $ u=w=1 $, the space $ M_{\alpha ,\lambda 
}^{p,l} (u,w) $ recover the space $ (L^p,L^l)^\lambda  $ defined in \cite{b6}. 
When $ \alpha =0 $, $ u=w=1 $, and $ l=\infty $, the space $ M_{\alpha ,\lambda 
}^{p,l} (u,w) $ is the Morrey space $ M_{\lambda /p}^{p} ({\mathbb{R}^n})$ in 
\cite{b20}. 
%When $ \alpha =0 $, $ \lambda =1/l -1/p $, and $ l=\infty $, the 
%space $ M_{\alpha ,\lambda }^{p,l} (u,w) $ could be viewed as an extension of 
%the weighted Morrey space $  M_{l }^{p} (u,w) $ introduced by Komori and Shirai 
%\cite{b11}. 
When $ u=w=1 $ and $ l=\infty  $, the $ M_{\alpha ,\lambda }^{p,l} 
(u,w) $ and the Morrey space $ L_{\alpha p,V }^{p,np\lambda } (\mathbb{R}^n) $ 
defined in \cite{b25} are equivalent by Dong and Tang. When $ l=\infty  $ and $ 
u=w $, the $ M_{\alpha ,\lambda }^{p,l} (u,w) $ corresponds to the $ L_{\alpha 
,V }^{p,\lambda } (w) $ introduced by Pan and Tang in \cite{b18}. When $ u=w $, 
$ M_{\alpha, \lambda}^{ p, l}(u, w) $ and $ WM_{\alpha, \lambda}^{ p, l}(u, w)$  are denoted by   $M_{\alpha, \lambda}^{ p, l}(w) $ and $WM_{\alpha, \lambda}^{ p, l}(w) $, 
respectively.

\begin{definition}{\rm(\cite{b28})} 
Suppose that $ u $ is a weight, $ \vec w=(w_1,\ldots ,w_m) $ is a  multiple  
weight, $ \lambda \in [0,1) $, $ 1< l\le \infty  $, $\alpha \in (-\infty 
,\infty )  $, $ \vec p=(p_1,\ldots ,p_m) $, and $ Q:=Q(z,r)$. The new 
strong Morrey space $ M_{\alpha ,\lambda }^{\vec p,l} (u,\vec w) $ is defined 
as the set of all vector-valued measurable functions $ \vec f=(f_1,\ldots 
,f_m) $ satisfying $\| \vec f \| _{M_{\alpha ,\lambda }^{\vec p,l} (u,\vec w)} 
< \infty $, where
\begin{align*}
\|\vec f\|_{M_{\alpha, \lambda}^{\vec p, l}(u, \vec w)}:=\sup 
_{r>0}\left[\int_{\mathbb{R}^{n}}\left(\varphi  
(Q(z,r))^{\alpha}u(Q(z,r))^{-\lambda}\prod_{j=1}^{m} \left\|f_j 
\chi_{Q(z,r)}\right\|_{L^{p_j}(w_j)}\right)^{l} d z\right]^{1 / l}.
\end{align*}
The new weak Morrey space $ WM_{\alpha ,\lambda }^{\vec p,l} (u,\vec w) $ is 
defined as the set of all vector-valued measurable functions $ \vec 
f=(f_1,\ldots ,f_m) $ satisfying $\| \vec f  \| _{WM_{\alpha ,\lambda }^{\vec 
p,l} (u,\vec w)} < \infty$, where
\begin{align*}
\|\vec f\|_{WM_{\alpha, \lambda}^{\vec p, l}(u, \vec w)}
&:=\sup _{r>0}\left[\int_{\mathbb{R}^{n}}\left(\varphi  
(Q(z,r))^{\alpha}u(Q(z,r))^{-\lambda}\prod_{j=1}^{m} \left\|f_j 
\chi_{Q(z,r)}\right\|_{L^{p_j,\infty }(w_j)}\right)^{l} d z\right]^{1/l 
}
\end{align*}
for $  p_j\ge 1 $, $ j=1,\dots ,m $. 
\end{definition}
When $ m=1 $, the new strong Morrey space 
$ M_{\alpha ,\lambda }^{\vec p,l} (u,\vec w) $ and the new weak Morrey space $ 
WM_{\alpha ,\lambda }^{\vec p,l} (u,\vec w) $ are equivalent to the space $ 
M_{\alpha, \lambda}^{ p, l}(u, w) $ and $ WM_{\alpha, \lambda}^{ p, l}(u, w) $, 
respectively.

\begin{definition}\label{definition1.6}
Let $C^{m}_j$ be a family of the finite subsets $\xi =\left \{ \xi (1),\dots ,\xi (j) \right \} $ consisting of $j$ distinct elements from the set $\left \{ 1,\dots ,m \right \} $ for $j,m\in \mathbb{N}^+$ and $1\le j\le m$. In addition, if $a< b$, then $\xi (a)< \xi (b)$. For any $\xi \in C^{m}_j$, the complementary sequence of $\xi$, denoted as $\xi^{'}$, is defined as $\left \{ 1,\dots ,m \right \}\setminus \xi $. In particular, $C^{m}_0=\phi $. For any m-fold sequence $\vec b$ and $\xi \in C^{m}_j$, the j-fold sequence $\vec b_\xi =\left \{b_{\xi(1)},\dots   b_{\xi(j)}\right \} $ is a finite subset of $\vec b =(b_1,\dots ,b_m)$.
\end{definition}
In Figure 1, the relationship between the new class of multilinear square 
operators and the 
classical multilinear square operators, $BMO$ functions and new $BMO$ 
functions, new multiple weights and multiple weights are clearly explained, and the connection between our work and previous research is also 
introduced.

\begin{figure}[H]
\centering
\includegraphics[width=1\linewidth]{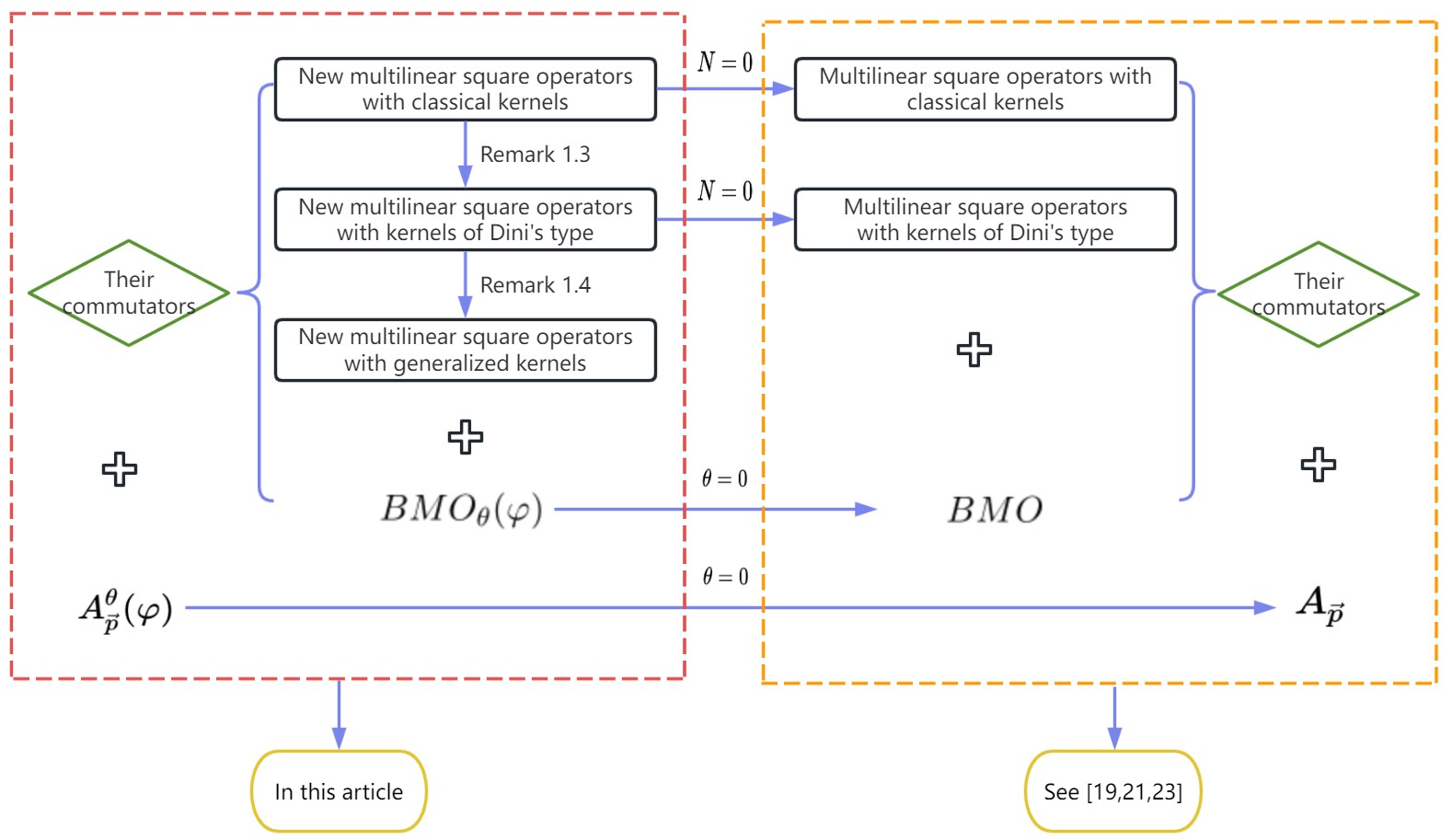} 
\caption{ Relation among of several classes of multilinear square operators, $BMO$ functions, and 
weight functions}
\label{FIG1}
\end{figure}
This article will be organized as follows. In Section \ref{sec2}, we will state 
our main theorems for new weighted norm inequalities of $ T_{\sum \vec{b}} 
$ and $ T_{\prod\vec{b}} $. In Section \ref{sec3}, we will list the necessary 
lemmas and definitions. In Section \ref{sec4}, we will give the proof of all  
theorems.

 In this article all notations are defined as needed. $|E|$ 
denotes the Lebesgue measure of $E$, $\chi_E$ stands for the characteristic 
function of $E$, $Q(x,r)$ is the cube centered at $x$ with the sidelength $r$, the cubes $Q$ must be with side parallel to the axes, $f_Q$ denotes the average $f_Q:=\frac{1}{|Q|} \int_{Q}f(y)dy$, and we define $ 
w(E):=\int_Ew(y)dy$ with the weight $ w $ and the measurable set $ E $.

\section{ Main Theorems}\label{sec2}

\quad\quad  In this section, we give the main theorems of this article. Theorem 
\ref{Theorem2.1} and Theorem \ref{Theorem2.2} prove the 
boundedness of multilinear  commutators and multilinear iterative commutators 
on the weighted Lebesgue spaces, respectively. 

\begin{theorem}\label{Theorem2.1} 
Suppose that $m\geq2$, $T$ is the new multilinear square operator with 
generalized kernel as in Definition \ref{definition1.3}. Let ${\textstyle 
\sum_{k=1}^{\infty }}  k C_k<\infty$, 
$ \vec{\theta}=(\theta_1,\ldots,\theta_m) $, $ \theta_j\geq 0 $, 
$ j=1,\ldots,m $,  $\vec{p}=(p_1,\ldots,p_m)$, $1/p=1/p_1+\ldots+1/p_m$,   
$\vec \omega=(\omega_1,\ldots,\omega_m)\in 
A_{\vec{p}/q'}^\infty(\varphi)$, $v_{\vec\omega}= {\textstyle \prod_{j=1}^{m}} 
\omega_j^{p/p_j}$, and $\vec{b}\in BMO_{\vec\theta}^m(\varphi)$. If 
$q'<p_j<\infty$, $ j=1,\ldots ,m $, there exists a positive constant $ C $ such 
that 
\begin{align*}
\|T_{\sum  \vec b}(\vec f)\|_{L^{p}(v_{\vec\omega})}\le C\sum_{j=1}^{m} \| b_j 
\|_{BMO_{\theta_j}(\varphi)}
\prod_{j=1}^{m}\| f_j \|_{L^{p_j}(\omega_j)}.
\nonumber
\end{align*}
\end {theorem}

\begin{theorem}\label{Theorem2.2} 
	Suppose that $m\geq2$, $T$ is the new multilinear square operator with 
	generalized kernel as in Definition \ref{definition1.3}. Let ${\textstyle 
		\sum_{k=1}^{\infty }}  k^m C_k<\infty$, $ 
	\vec{\theta}=(\theta_1,\ldots,\theta_m) $, $ \theta_j\geq 0 $, 
	$ j=1,\ldots,m $, $\vec{p}=(p_1,\ldots,p_m)$, $1/p=1/p_1+\cdots+1/p_m$, 
	$\vec \omega=(\omega_1,\ldots,\omega_m)\in 
	A_{\vec{p}/q'}^\infty(\varphi)$, $v_{\vec\omega}= {\textstyle \prod_{j=1}^{m}} 
	\omega_j^{p/p_j}$, and $\vec{b}\in BMO_{\vec\theta}^m(\varphi)$. If 
	$q'<p_j<\infty$, $ j=1,\ldots ,m $, there exists a positive constant $ C $ such 
	that 
	\begin{align*}
		\| T_{\prod \vec b}(\vec f)\|_{L^{p}(v_{\vec\omega})}\le C\prod_{j=1}^{m} \| 
		b_j \|_{BMO_{\theta_j}(\varphi)}
		\prod_{j=1}^{m}\| f_j \|_{L^{p_j}(\omega_j)}.
		\nonumber
	\end{align*}
\end {theorem}

Theorem \ref{Theorem2.3} and Theorem \ref{Theorem2.4} establish the 
boundedness of multilinear commutators and multilinear iterative commutators on 
the new Morrey spaces, respectively.

\begin{theorem}\label{Theorem2.3} 
Suppose that $m\geq2$, $T$ is the new multilinear square operator with 
generalized kernel as in Definition \ref{definition1.3}. Let ${\textstyle 
\sum_{k=1}^{\infty }}  k C_k<\infty$, $ 
\vec{\theta}=(\theta_1,\ldots,\theta_m)$, $ \theta_j> 0 $, 
$ j=1,\ldots,m $,  $\alpha \in (-\infty  ,\infty ) $, $ \lambda \in [0,1) $, 
$1\le p< l\le \infty$, $ \vec 
p=(p_1,\ldots,p_m) $, $1/p=1/p_1+\ldots+1/p_m$, $ \vec w=(w_1,\ldots ,w_m)\in 
A_{\vec p /q'}^{\theta }(\varphi ) $, $v_{\vec w}= {\textstyle 
\prod_{j=1}^{m}}  w_{j}^{p/p_j}$, $ \beta =\alpha +mp\lambda \theta/q' $, $ 
\theta > 0 $, and $ \vec{b}\in BMO_{\vec\theta}^m(\varphi)  $. If $q'<p_j<\infty$, $ 
j=1,\ldots ,m $, there exists a positive constant C such that
\begin{align*}
\|T_{\sum \vec b}(\vec{f})\|_{M_{\alpha, \lambda}^{p, 
l}\left(v_{\vec{w}}\right)} \leq C 
\sum_{j=1}^{m} \left \| b_j \right \|_{BMO_{\theta _j}(\varphi )}  
||\vec f||_{M_{\beta , \lambda }^{\vec p, l}(v_{\vec w},\vec w)}.
\end{align*}
\end {theorem}

\begin{theorem}\label{Theorem2.4} 
Suppose that $m\geq2$, $T$ is the new multilinear square operator with 
generalized kernel as in Definition \ref{definition1.3}. Let ${\textstyle 
\sum_{k=1}^{\infty }}  k^m C_k<\infty$, $ 
\vec{\theta}=(\theta_1,\ldots,\theta_m)$, $ \theta_j > 0 $, 
$ j=1,\ldots,m $, $\alpha \in (-\infty  ,\infty )  $, $ \lambda \in [0,1) $, 
$1\le p< l\le \infty$, $ \vec p=(p_1,\ldots,p_m) 
$, $1/p=1/p_1+\ldots+1/p_m$, $ \vec w=(w_1,\ldots ,w_m)\in A_{\vec p 
/q'}^{\theta }(\varphi ) $, $v_{\vec w}= {\textstyle \prod_{j=1}^{m}}  
w_{j}^{p/p_j}$, $ \beta =\alpha +mp\lambda \theta/q' $, $ \theta > 0 $, and $ 
\vec{b}\in BMO_{\vec\theta}^m(\varphi) $. If $q'<p_j<\infty$, $ j=1,\ldots ,m 
$, there exists a positive constant C such that
\begin{align*}
\|T_{\prod \vec b}(\vec{f})\|_{M_{\alpha, \lambda}^{p, 
l}\left(v_{\vec{w}}\right)} \leq C 
\prod_{j=1}^{m} \left \| b_j \right \|_{BMO_{\theta _j}(\varphi )}  
||\vec f||_{M_{\beta , \lambda }^{\vec p, l}(v_{\vec w},\vec w)}.
\end{align*}
\end {theorem}
%\begin{remark}
%When $ N=0 $ in the condition \eqref{1.4}-\eqref{1.5}, the new multilinear
%square operator with generalized kernel is the multilinear square operator 
%with generalized kernel. The class of generalized kernels includes the kernels 
%of Dini’s type and classical kernels. When $ A_{\vec p}^{\theta } (\varphi 
%)=A_{\vec p} $ with $ \theta =0 $, $ \vec \theta =(\theta _1,\dots ,\theta 
%_m)=(0,\dots ,0) $, $ BMO_{\vec \theta }^{m}(\varphi ) =BMO^{m} $, Theorem {\rm 
%\ref{Theorem2.1}-\rm \ref{Theorem2.4}} are the boundedness of 
%multilinear commutators and multilinear iterative commutators generated by 
%multilinear square operators with generalized kernels and $ BMO $ functions on 
%weighted Lebesgue spaces and new Morrey spaces, 
%respectively. As far as I can tell, when we go back to these classical cases, 
%the results obtained are also up to date.
%\end{remark}

\section{ Some Preliminaries and Notations }\label{sec3}
\begin{definition}
Let $0<\eta<\infty$, $ Q $ be a dyadic cube, and $ f $ be a locally integral 
function. The dyadic maximal function $M_{\varphi,\eta}^\triangle$ and the 
dyadic sharp maximal operator $M_{\varphi,\eta}^{\sharp,\triangle}$ are defined 
by 
\begin{align*}
M_{\varphi,\eta}^\triangle f(x):=\mathop{\sup}\limits_{  {
		Q}\ni x}\frac{1}{\varphi(Q)^{\eta}|Q|}\int_Q|f(y)|dy,
\end{align*}
and
\begin{align*}
M_{\varphi,\eta}^{\sharp,\triangle} f(x)
&:=\mathop{\sup}\limits_{{ 
		Q}\ni x,r<1}\frac{1}{|Q|}\int_{Q}|f(y)-f_Q|dy+\mathop{\sup}\limits_{{
		Q}\ni x,r\geq1}\frac{1}{\varphi(Q)^{\eta}|Q|}\int_{Q}|f(y)|dy\\
&\simeq\mathop{\sup}\limits_{ { 
		Q}\ni x,r<1}\mathop{\inf}\limits_{ 
	c}\frac{1}{|Q|}\int_{Q}|f(y)-c|dy+\mathop{\sup}\limits_{{
		Q}\ni x,r\geq1}\frac{1}{\varphi(Q)^{\eta}|Q|}\int_{Q}|f(y)|dy.
\end{align*}
For $0<\delta < \infty$, denote
\begin{align*}
M^\triangle_{\delta,\varphi,\eta}f(x):=[M^\triangle_{\varphi,\eta}
(|f|^\delta)]^{\frac1\delta}(x), 
M^{\sharp, \triangle}_{\delta,\varphi,\eta}f(x):=[M^{\sharp, 
\triangle}_{\varphi,\eta}(|f|^\delta)]^{\frac1\delta}(x).
%\nonumber
\end{align*}
\end{definition}

\begin{definition}
Given $0<\eta<\infty$, $0<\delta < \infty$, $\vec{f}=(f_1,\ldots, f_m)$, the multilinear maximal
operators $\mathcal{M}_{\varphi,\eta}$ and $ 
\mathcal{M}_{\delta,\varphi,\eta} $ are defined by
\begin{align*}
\mathcal{M}_{\varphi,\eta}(\vec{f})(x):=\mathop{\sup}\limits_{ {
		Q}\ni x}\prod_{j=1}^m\frac1{\varphi(Q)^\eta \left | Q \right | } \int_Q 
|f_j(y_j)| dy_j,
\end{align*}
and
\begin{align*}
\mathcal{M}_{\delta,\varphi,\eta}(\vec f)(x)
:=\left [ \mathcal{M}_{\varphi,\eta}(|\vec{f}|^{\delta })(x) \right ] 
^{\frac{1}{\delta } }
=\left(\mathop{\sup}\limits_{ 
	{ Q}\ni x}\prod_{j=1}^m\frac1{\varphi(Q)^\eta \left | Q \right |} \int_Q 
|f_j(y_j)|^\delta dy_j \right)^{\frac1\delta}.
\end{align*}
\end{definition}
\begin{lemma}\label{Lemma3.1}{\rm(See \cite{b2}, Proposition 2.5)}  
Let $\theta\geq0$ and $a\geq1$. If $b\in BMO_{\theta}(\varphi)$, then for cubes 
$Q:=Q(x,r)$,
\begin{itemize}
\item [\rm (i)] $\left(\frac1{|Q|} \displaystyle \int_Q|b(y)-b_Q|^a dy 
\right)^{\frac1a}\leq C\Vert b 
\Vert_{BMO_{\theta}(\varphi)}{\varphi}(Q)^{\theta}$.
\item [\rm (ii)] $\left(\frac1{|{t_0}^k Q|} \displaystyle 
\int_{{t_0}^kQ}|b(y)-b_Q|^a 
dy \right)^{\frac1a}\leq Ck\Vert b 
\Vert_{BMO_{\theta}(\varphi)}{\varphi}({t_0}^kQ)^{\theta}$, for all $k\in 
\mathbb{N}$ and $ t_0=2,3 $.
\end{itemize}
\end {lemma}

\begin{lemma}\label{Lemma3.2}{\rm(See    \cite{b12}  (2.16))} 
Let $0<p<u<\infty$ and $ f $ be a measurable function. There is a 
positive constant $C:=C_{p,u}$ such that the following statement holds,
\begin{equation}
|Q|^{-1/p}\Vert f \Vert_{L^p (Q)}\leq C|Q|^{-1/u}\Vert f 
\Vert_{L^{u,\infty}(Q)}.
\nonumber
\end{equation}
\end {lemma}

\begin{lemma}\label{Lemma3.3}{\rm(See \cite{b2}, Proposition 2.3)}  
Let $1\leq p_1,\ldots,p_m<\infty$ and  
$\vec{\omega}=(\omega_1,\ldots,\omega_m)$. Then the following statements are 
equivalent.
\begin{itemize}
\item [\rm (i)] $\vec\omega \in A_{\vec{p}}^{\infty}(\varphi).$\\
\item [\rm (ii)] $\omega_j^{1-p'_j} \in A_{mp'_j}^{\infty}(\varphi),$ 
$j=1,\ldots,m $, and $v_{\vec{\omega}} \in A_{mp}^{\infty}(\varphi)$.
\end{itemize}
\end {lemma}

\begin{lemma}\label{Lemma3.4}{\rm(See \cite{b2}, Proposition 2.4 or \rm See 
\cite{b28}, Theorem 5.1)}  
Let $1\leq p_1,\ldots,p_m<\infty$, and $\vec{\omega}=(\omega_1,\ldots,\omega_m) 
\in A_{\vec{p}}^{\infty}(\varphi) $. Then
\begin{itemize}
\item [\rm (i)] $\vec{\omega}\in A_{r\vec{p}}^{\infty}(\varphi)$, $r\geq 1$. \\
\item [\rm (ii)] If $1<p_1,\ldots,p_m<\infty$, then there exists a  $r>1$ such 
that $\vec{\omega}\in A_{\vec{p}/r}^{\infty}(\varphi)$ and $ p_j/r>1 $, $ 
j=1,\dots ,m $.
\end{itemize}
\end {lemma}

\begin{lemma}\label{Lemma3.5}{\rm(See \cite{b19}, Theorem 3.3)}  
Let $1<p_j<\infty$, $j=1,\ldots, m$, $1/p=1/p_1+\ldots+1/p_m$, 
and $\vec \omega\in A_{\vec{p}}^\infty(\varphi)$, then there exists some 
$\eta_0>0$ depending on $p, m, p_j$, such that\\
\begin{equation}
\Vert \mathcal{M}_{\varphi,\eta_0}(\vec{f}) \Vert_{L^{p}(v_{\vec \omega })}\leq 
C\prod_{j=1}^{m}\Vert f_j \Vert_{L^{p_j}(\omega_j)}.
\nonumber
\end{equation}
\end {lemma}

\begin{lemma}\label{Lemma3.6}{\rm(See \cite{b24}, Proposition 2.2)}  
Let $1<p<\infty$, $\omega\in A_\infty^\infty(\varphi)$, $0<\eta<\infty$, and 
$f\in L^p(\omega)$, then\\
\begin{equation}
\Vert f \Vert_{L^p(\omega)}\leq\Vert M_{\varphi,\eta}^\triangle f 
\Vert_{L^p(\omega)}\leq C \Vert M_{\varphi,\eta}^{\sharp, \triangle} f 
\Vert_{L^p(\omega)}.
\nonumber
\end{equation}
\end {lemma}

\begin{lemma}\label{Lemma3.7}{\rm(See \cite{b28}, Lemma 2.2)}  
Let $ 0< \theta < \infty ,1\le p< \infty  $ and $ E $ be any measurable subset 
of a cube $ Q $. If $ w\in A_{p}^{\theta } (\varphi) $, then there exist 
positive constants $ 0< \delta < 1 $, $ \eta  $ and $ C $ such that 
\begin{align*}
\frac{w(E)}{w(Q)} \leq C 
\varphi(Q)^{\eta}\left(\frac{|E|}{|Q|}\right)^{\delta}.
\end{align*}
\end {lemma}

\begin{lemma}\label{Lemma3.8}{\rm(See \cite{b28}, Lemma 2.3)}  
Let $ 0< \theta < \infty$, and $1\le p< \infty  $. If $ w\in A_{p}^{\theta 
}(\varphi ) $, then there exists a positive constant $ C $, such that for any 
$\rho > 1 $,
\begin{align*}
w(\rho Q) \leq C \varphi(\rho Q)^{p \theta} w(Q).
\end{align*}
\end {lemma}

\begin{lemma}\label{Lemma3.9}{\rm(See \cite{b28}, Lemma 2.9)}  
If $ 0< \theta < \infty  $, $ f\in BMO_\theta(\varphi )  $, $ w\in A_{\infty 
}^{\infty  } (\varphi ) $, and $ 1\le  p< \infty  $, then there exist a positive constant
$ C $ and $ \tilde{s}>0 $ such that for every cube $ Q $, we have 
\begin{align*}
\left ( \frac{1}{w(Q)}\int_{Q} \left | f(x)-f_Q \right |^pw(x) dx \right ) 
^{\frac{1}{p} }\le C\varphi (Q)^{\frac{\tilde{s}}{p} }\left \| f \right 
\|_{BMO_{\theta}(\varphi ) }.
\end{align*}
\end {lemma}

\begin{lemma}\label{Lemma3.10}{\rm(See \cite{b13}, Theorem 2.1)}  
Suppose that $T$ is the new multilinear square operator with generalized 
kernel as in Definition \ref{definition1.3} and $  {\textstyle \sum_{k=1}^{\infty }} C_{k}< \infty  $. If  
$0<\delta<1/m$ and $0<\eta<\infty$, there exists a positive constant $C$,  
such that for all bounded measurable functions $\vec f=(f_1,\ldots,f_m)$ with 
compact support, 
\begin{equation}
M_{\delta,\varphi,\eta}^{\sharp, \triangle}(T(\vec f))(x)\leq 
C\mathcal{M}_{q',\varphi,\eta}(\vec f)(x).
\end{equation}
\end {lemma}
%\begin{lemma}\label{Lemma3.11}{\rm(\cite{b1})}   
%The following statements are established.
%\begin{itemize}
%\item [\rm (i)] $A_p^{\infty}(\varphi) \subset A_q^{\infty}(\varphi) $, $1\leq 
%p<q<\infty$.
%\item [\rm (ii)] If $\omega \in A_p^{\infty}(\varphi )$ with $p>1$, then there 
%exists a $\varepsilon>0$ such that $\omega \in  
%A_{p-\varepsilon}^{\infty}(\varphi )$. Consequently, 
%$A_p^{\infty}(\varphi ):= {\textstyle \bigcup_{q< p}} A_q^{\infty}(\varphi )$.
%\item [\rm (iii)] If $\omega \in A_p^{\infty}(\varphi )$ for $p\geq1$, then 
%there exist positive numbers $l$, $\delta$ and $C$ so that for all cubes 
%$Q$,
%\begin{align*}
%\left(\frac{1}{|Q|}\int_Q\omega^{1+\delta}(x)dx\right)^{\frac{1}{1+\delta}}\leq
%C\left(\frac{1}{|Q|}\int_Q\omega(x)dx\right)(1+r)^l.
%\end{align*}
%\end{itemize}
%\end{lemma}
\begin{lemma}\label{Lemma3.11}
	Suppose that $T$ is the new multilinear square operator with 
	generalized kernel as in Definition \ref{definition1.3}. Suppose $\vec b\in BMO^{m}_{\vec \theta }(\varphi )$, $\vec \theta =(\theta _1,\dots ,\theta _m)$ with $\theta _j\ge 0,j=1,\dots ,m$. Let ${\textstyle 
		\sum_{k=1}^{\infty }}  k C_k<\infty$, $0<\delta<\varepsilon<1/m$, $q'<l<\infty$, and  
		$\eta >(\mathop{\max}\limits_{1\le j\le m}\theta _j)/(1/\delta -1/\varepsilon )$. There exists a positive constant $C$ such that for all $\vec 
		f=(f_1,\ldots,f_m)$ of bounded measurable functions with compact 
		suppport
		\begin{align*}
			M_{\delta,\varphi,\eta}^{\sharp,\triangle}(T_{\sum  \vec b}(\vec 
		f))(x)\leq C\sum_{j=1}^{m}  \| b_j \|_{BMO_{\theta_j}(\varphi)}
		\left(M_{\varepsilon,\varphi,\eta}^\triangle (T(\vec 
		f))(x)+\mathcal{M}_{l,\varphi,\eta}(\vec f)(x)\right).
	\end{align*}
\end {lemma}

\begin{proof}
	We only prove the case 
	$ \theta _1=\cdots =\theta _m=\theta $ for simplicity. 
	
	Fix $x\in \mathbb{R}^n$, for any dyadic cube $Q:=Q(x_0,r)\ni x$,  
	%for any constants $a_1, a_2$ and $\left | \left | a_{1}  \right | ^{t} - | 
	%%%\left | a_{2}  \right | ^{t} \right | \le \left | a_{1}- a_{2}  \right | 
	%%%^{t} 
	%%%%$ for $0<t<1$. 
	we consider the following two cases of the sidelength $r$ : $r<1$ and 
	$r\geq1$.
	
	\noindent{\it \rm \textbf{Case 1}:} $r<1$. Let $Q^*:=14n\sqrt{mn}Q$ 
	and $\lambda_j:=(b_j)_{Q^*}$ be the average on $Q^*$ of $b_j$, $j=1,2,\dots,m$, $z_0\in 4Q\setminus 3Q$. 
	We split each $f_j$ as $f_j=f_j^0+f_j^\infty$, for each $f_j^0=f_j \chi_{Q^*}$, 
	$f_j^\infty=f_j-f_j^0 $. There is
	\begin {align*}
	\prod_{j=1}^mf_j(y_j)&=\sum_{\alpha_1,\ldots,\alpha_m\in\{0,\infty\}}f_1^{\alpha_1}(y_1)\cdots
	f_m^{\alpha_m}(y_m)\\
	&=\prod_{j=1}^m f_j^{0}(y_j)+\sum_{(\alpha_1,\ldots,\alpha_m)\in 
		\mathscr{L}}f_1^{\alpha_1}(y_1)\cdots f_m^{\alpha_m}(y_m),
	\end {align*}
	where  $\mathscr{L}:=\{(\alpha_1,\ldots,\alpha_m)$: there is at least one 
	$\alpha_j=\infty\}$.
	
	Taking
	$A_j= \left(\int_{0}^{\infty} \left | \int_{\left(\mathbb{R}^{n}\right)^{m} 
	} \sum_{(\alpha_1,\ldots,\alpha_m)\in \mathscr{L}}K_{t}\left(z_{0} 
	,\vec{y}\right)(\lambda _j-b_j(y_j))\prod_{i=1}^{m} 
	f_{i}^{\alpha_i}\left(y_{i}\right) d \vec{y} \right | ^{2} \frac{d 
		t}{t}\right)^{\frac{1}{2}},$
	we conclude that
	\begin  {align*}
	& |T_{b_j}^{j}(\vec f)(z)-A_j| \\
	&\leq  \Bigg(\int_{0}^{\infty} \bigg | 
	\int_{\left(\mathbb{R}^{n}\right)^{m} }\Big (K_{t}(z, \vec{y})\big ( b_j(z)-b_{j}(y_j))\prod_{i=1}^m f_i(y_i)\\
	&\ \ \ \ - \sum_{(\alpha_1,\ldots,\alpha_m)\in 
		\mathscr{L}}K_{t}\left(z_{0} 
	,\vec{y}\right)(\lambda _j- b_j(y_j))\prod_{i=1}^m f_i^{\alpha _i}(y_i) \Big )
	d \vec{y} \bigg | ^{2} \frac{d 
		t}{t}\Bigg)^{\frac{1}{2}} \\
%	&=\Bigg(\int_{0}^{\infty} \bigg | 
%	\int_{\left(\mathbb{R}^{n}\right)^{m} }K_{t}(z, \vec{y})\big ( b_j(z)-\lambda _j)\prod_{i=1}^m f_i(y_i)+K_{t}(z, \vec{y})\big ( \lambda _j-b_j(y_j))\prod_{i=1}^m f_i(y_i)\\
%	&\ \ \ \ - \sum_{(\alpha_1,\ldots,\alpha_m)\in 
%		\mathscr{L}}K_{t}\left(z_{0} 
%	,\vec{y}\right)(\lambda _j -b_j(y_j))\prod_{i=1}^m f_i^{\alpha _i}(y_i) 
%	d \vec{y} \bigg | ^{2} \frac{d 
%		t}{t}\Bigg)^{\frac{1}{2}}\\
	&=\Bigg(\int_{0}^{\infty} \bigg | 
	\int_{\left(\mathbb{R}^{n}\right)^{m} }\Big(K_{t}(z, \vec{y})\big ( b_j(z)-\lambda _j)\prod_{i=1}^m f_i(y_i)+K_{t}(z, \vec{y})\big (\lambda _j -b_j(y_j))\prod_{i=1}^m f_i^{0}(y_i)\\
	&\ \ \ \ + \sum_{(\alpha_1,\ldots,\alpha_m)\in 
		\mathscr{L}}K_{t}\left(z
	,\vec{y}\right)(\lambda _j -b_j(y_j))\prod_{i=1}^m f_i^{\alpha _i}(y_i) \\
	&\ \ \ \ - \sum_{(\alpha_1,\ldots,\alpha_m)\in 
		\mathscr{L}}K_{t}\left(z_{0} 
	,\vec{y}\right)(\lambda _j -b_j(y_j))\prod_{i=1}^m f_i^{\alpha _i}(y_i) 
	\Big)d \vec{y} \bigg | ^{2} \frac{d 
		t}{t}\Bigg)^{\frac{1}{2}}\\
	&\leq \Bigg(\int_{0}^{\infty} \bigg | 
	\int_{\left(\mathbb{R}^{n}\right)^{m} }K_{t}(z, \vec{y})\big ( b_j(z)-\lambda _j)\prod_{i=1}^m f_i(y_i)
	d \vec{y} \bigg | ^{2} \frac{d 
		t}{t}\Bigg)^{\frac{1}{2}}\\
	&\ \ \ \ + \Bigg(\int_{0}^{\infty} \bigg | 
	\int_{\left(\mathbb{R}^{n}\right)^{m} }K_{t}(z, \vec{y})\big ( b_j(y_j)-\lambda _j)\prod_{i=1}^m f_i^{0}(y_i)
	d \vec{y} \bigg | ^{2} \frac{d 
		t}{t}\Bigg)^{\frac{1}{2}} \\
	&\ \ \ \ + \sum_{(\alpha_1,\ldots,\alpha_m)\in 
		\mathscr{L}}\left(\int_{0}^{\infty} \left | 
	\int_{\left(\mathbb{R}^{n}\right)^{m} 
	}\big(K_{t}(z, \vec{y})-K_{t}\left(z_{0} 
	,\vec{y}\right)\big)(b_j(y_j)-\lambda _j) \prod_{i=1}^{m} 
	f_{i}^{\alpha_i}\left(y_{i}\right) d \vec{y} \right | ^{2} \frac{d 
		t}{t}\right)^{\frac{1}{2}} \\
	&= |b_j(z)-\lambda _j|T(\vec{f})(z)+T(f_1^{0},\dots ,f_{j-1}^{0},(b_j-\lambda _j)f_{j}^{0},f_{j+1}^{0},\dots ,f_m^{0})(z)\\
	&\ \ \ \ +  \sum_{(\alpha_1,\ldots,\alpha_m)\in 
		\mathscr{L}}\left(\int_{0}^{\infty} \left | 
	\int_{\left(\mathbb{R}^{n}\right)^{m} 
	}\big(K_{t}(z, \vec{y})-K_{t}\left(z_{0} 
	,\vec{y}\right)\big)(b_j(y_j)-\lambda _j) \prod_{i=1}^{m} 
	f_{i}^{\alpha_i}\left(y_{i}\right) d \vec{y} \right | ^{2} \frac{d 
		t}{t}\right)^{\frac{1}{2}}.
	\end {align*}
	We have
	\begin  {align*}
	&\left(\frac1{|Q|}\int_Q |T_{b_j}^{j}(\vec f)(z)-A_j|^\delta dz\right)^{\frac1\delta}\\
	& \leq C \left(\frac1{|Q|}\int_Q |(b_j(z)-\lambda _j)T(\vec f)(z)|^\delta 
	dz\right)^{\frac1\delta}\\
	& \ \ \ \ +C \left(\frac1{|Q|}\int_Q |T(f_1^{0},\dots ,f_{j-1}^{0},(b_j-\lambda _j)f_{j}^{0},f_{j+1}^{0},\dots ,f_m^{0})(z)|^\delta 
	dz\right)^{\frac1\delta}\\
	&\ \ \ \ + C \sum_{(\alpha_1,\ldots,\alpha_m)\in \mathscr{L}}
	\frac{1}{|Q|} \int_{Q}\Bigg(\int_{0}^{\infty} \Big| 
	\int_{\left(\mathbb{R}^{n}\right)^{m} 
	}\big(K_{t}(z, \vec{y})-K_{t}\left(z_{0} 
	,\vec{y}\right)\big) (b_j(y_j)-\lambda _j) \\
	&\ \ \ \times \prod_{i=1}^{m} 
	f_{i}^{\alpha _i}\left(y_{i}\right) d \vec{y} \Big | ^{2} \frac{d 
		t}{t}\Bigg)^{\frac{1}{2}} d z\\
	&:=\uppercase\expandafter{\romannumeral+1}+\uppercase\expandafter{\romannumeral+2}+\uppercase\expandafter{\romannumeral+3}\\
	&:=\uppercase\expandafter{\romannumeral+1}+\uppercase\expandafter{\romannumeral+2}+\sum_{(\alpha_1,\ldots,\alpha_m)\in \mathscr{L}}\uppercase\expandafter{\romannumeral+3}_{\alpha _1\dots \alpha _m}.
	\end {align*}
	
	Choosing $1<q_2<\min\left \{ \varepsilon /\delta,1/(1- \delta)  \right \}$, then select $q_1$ such that $1/q_1+1/q_2=1$, we obtain $1< q_1,q_2< \infty $ and $\delta q_1> 1$. Using  
	H\"{o}lder's inequality and Lemma \ref{Lemma3.1}, we deduce that
	\begin{align}\label{3.2}
		\uppercase\expandafter{\romannumeral+1}&\leq \left( \frac1{|Q|}\int_Q 
		|b_j(z)-\lambda_j|^{\delta q_1}dz\right)^{\frac1{\delta q_1}} \left( \frac1{|Q|}\int_Q |T(\vec{f})(z)|^{\delta 
			q_2}dz\right)^{\frac1{\delta q_2}}\\
		&\leq C\|b_j\|_{BMO_\theta(\varphi)} \left( 
		\frac1{\varphi(Q)^\eta|Q|}\int_Q |T(\vec{f})(z)|^\varepsilon 
		dz\right)^{\frac{1}{\varepsilon} }\nonumber\\
		&\leq C\|b_j\|_{BMO_\theta(\varphi)} 
		M_{\varepsilon,\varphi,\eta}^\triangle 
		(T(\vec{f}))(x).\nonumber
		\end {align}
		
		Let $1/h_j+1/h'_j=1$ and $h_j=t/s_j$. For $1\le s_1,\dots , s_m \leq q'<l$, so $h_j>1$ and $s_jh_j^{'}> 1$. By Lemma \ref{Lemma3.2}, according to 
		$T:L^{s_1}\times\cdots\times L^{s_m}\to L^{s,\infty}$, H\"{o}lder's inqualitiy 
		and Lemma \ref{Lemma3.1}, we have
		\begin{align*}
			\uppercase\expandafter{\romannumeral+2} 
			&\leq C\| T(f_1^0,\dots ,(b_j-\lambda_j)f_j^0,\dots ,f_m^0) \|_{L^{s,\infty}(Q,\frac{dx}{|Q|})}\\
			&\leq C\prod_{i=1,i\ne j}^{m} \left(\frac1{|Q^*|}\int_{Q^*} |f_i(y_i)|^{s_i}dy_i 
			\right)^{\frac1{s_i}} 
			\left(\frac1{|Q^*|}\int_{Q^*} |(b_j(y_j)-\lambda_j)f_j(y_j)|^{s_j}dy_j 
			\right)^{\frac1{s_j}}  \\
			&\leq C\prod_{i=1,i\ne j}^{m} \left(\frac1{|Q^*|}\int_{Q^*} |f_i(y_i)|^{l}dy_i 
			\right)^{\frac1{l}} \left(\frac1{|Q^*|}\int_{Q^*} |b_j(y_j)-\lambda_j|^{s_jh'_j}dy_j 
			\right)^{\frac1{s_jh'_j}}  \\
			& \quad \times  \left(\frac1{|Q^*|}\int_{Q^*} |f_j(y_j)|^{l}dy_j 
			\right)^{\frac1{l}}   \\
			&\leq C \| 
			b_j\|_{BMO_\theta(\varphi)}\prod_{i=1}^{m} \left(\frac1{\varphi(Q^*)^\eta|Q^*|}\int_{Q^*} 
			|f_i(y_i)|^{l}dy_i \right)^{\frac1l}  \\
			&\leq C \| b_j\|_{BMO_\theta(\varphi)}\mathcal{M}_{l,\varphi,\eta}(\vec 
			f)(x).
			\end {align*}
			
			To estimate the term $\uppercase\expandafter{\romannumeral+3}$, let $\Delta_k:=Q(z_0, 2^{k}\sqrt{mn}|z-z_0|)$, $ z\in Q $, $k\in \mathbb{N}_+$, $h=l/q'$ and $1/h+1/h'=1$. It follows from 
			$1<q'<l$ that $h>1$. Since $z\in Q$, $z_0\in 4Q\setminus 3Q$ and $\Delta_2 
			\subset Q^*$, we obtain $(\mathbb{R}^n)^m\setminus (Q^*)^m\subset 
			(\mathbb{R}^n)^m\setminus (\Delta_2)^m$, $|z-z_0|\sim r$ and $\Delta_{k+2} 
			\subset 2^kQ^*$. Taking $N>m\eta/l+\theta$, by H\"{o}lder's inquality, Minkowski's inequality, the smoothness condition and Lemma \ref{Lemma3.1},  we conclude that
			\setlength\abovedisplayskip{10pt}
			\setlength\belowdisplayskip{10pt}
			\begin{align*}
				%& \uppercase\expandafter{\romannumeral+3}\\
%				&\frac{1}{|Q|} \int_{Q}\left(\int_{0}^{\infty} \left | 
%				\int_{\left(\mathbb{R}^{n}\right)^{m} 
%				}\left|K_{t}(z, \vec{y})-K_{t}\left(z_{0} 
%				,\vec{y}\right)\right| (b_j(y_j)-\lambda _j) \prod_{i=1}^{m} 
%				\left|f_{i}^{\alpha _i}\left(y_{i}\right)\right| d \vec{y} \right | ^{2} \frac{d 
%					t}{t}\right)^{\frac{1}{2}} d z\\
\uppercase\expandafter{\romannumeral+3}_{\alpha _1\dots \alpha _m}
				&\leq \frac{1}{|Q|} \int_{Q}\Bigg(\int_{0}^{\infty} \Big ( 
				\int_{\left(\mathbb{R}^{n}\right)^{m} \setminus (Q^*)^m
				}\left|K_{t}(z, \vec{y})-K_{t}\left(z_{0} 
				,\vec{y}\right)\right| (b_j(y_j)-\lambda _j)\\
				& \quad \times \prod_{i=1}^{m} 
				\left|f_{i}\left(y_{i}\right)\right| d \vec{y} \Big ) ^{2} \frac{d 
					t}{t}\Bigg)^{\frac{1}{2}} d z \\
%				&\leq \frac{1}{|Q|} \int_{Q}\Bigg(\int_{0}^{\infty} \bigg|  
%				\int_{\left(\mathbb{R}^{n}\right)^{m} 
%					\backslash\left(\Delta_2\right)^{m}}\left|K_{t}(z, 
%				\vec{y})-K_{t}\left(z_{0} ,\vec{y}\right)\right| 
%				|b_j(y_j)-\lambda 
%				_j|\prod_{i=1}^{m}|f_{i}(y_{i})| d 
%				\vec{y}\bigg|^{2} \frac{d t}{t}\Bigg)^{\frac{1}{2}} d z \\
				%	&\leq \frac{1}{|Q|} \int_{Q}\Bigg(\int_{0}^{\infty}\sum_{k=1}^{\infty } 
				%	
				%\bigg(\int_{\left(\Delta_{k+2}\right)^{2}\backslash\left(\Delta_{k+1}\right)^{2}}\left|K_{t}(z,
				%	y_1,y_2)-K_{t}\left(z_{0},y_1,y_2\right)\right|\\
				%	&\ \ \ \ \times |f_{1}\left(y_{1}\right)||b_2(y_2)-\lambda 
				%	_2||f_{2}(y_{2})|dy_1d y_2\bigg)^{2} \frac{d t}{t}\Bigg)^{\frac{1}{2}} d z\\
				&\leq \frac{1}{|Q|} \int_{Q}\sum_{k=1}^{\infty } \Bigg(\int_{0}^{\infty} 
				\bigg(\int_{\left(\Delta_{k+2}\right)^{m} 
					\backslash\left(\Delta_{k+1}\right)^{m}}\left|K_{t}(z,\vec{y})-K_{t}\left
				(z_{0},\vec{y}\right)\right| 
				|b_j(y_j)-\lambda 
				_j|\\
				& \quad \times \prod_{i=1}^{m}|f_{i}(y_{i})| d 
				\vec{y}\bigg)^{2} \frac{d t}{t}\Bigg)^{\frac{1}{2}} d z\\
				&\leq \frac{1}{|Q|} \int_{Q} \sum_{k=1}^{\infty } \Bigg(  
				\int_{0}^{\infty}     \bigg(\int_{\left(\Delta_{k+2}\right)^{m} 
					\backslash\left(\Delta_{k+1}\right)^{m}}\left|K_{t}(z, 
				\vec{y})-K_{t}\left(z_{0} ,\vec{y}\right)\right|^{q}  d \vec{y}  \bigg) 
				^\frac{2}{q} \\
				& \quad \times  \bigg( \int_{\left(\Delta_{k+2}\right)^{m} 
					\backslash\left(\Delta_{k+1}\right)^{m}} 
				(|b_j(y_j)-\lambda 
				_j|\prod_{i=1}^{m}|f_{i}(y_{i})|)^{q'} d \vec{y} \bigg) 
				^\frac{2}{q'}   \frac{dt}{t}\Bigg) ^{\frac{1}{2}} d z\\
				&\leq \frac{1}{|Q|} \int_{Q} \sum_{k=1}^{\infty } \Bigg( 
				\int_{\left(\Delta_{k+2}\right)^{m}} 
				\bigg(|b_j(y_j)-\lambda 
				_j|\prod_{i=1}^{m}|f_{i}(y_{i})| \bigg)^{q'}d 
				\vec{y} \Bigg) ^\frac{1}{q'} \\
				& \quad \times 
				\Bigg(\int_{0}^{\infty}\bigg(\int_{\left(\Delta_{k+2}\right)^{m} 
					\backslash\left(\Delta_{k+1}\right)^{m}}\left|K_{t}(z, 
				\vec{y})-K_{t}\left(z_{0} ,\vec{y}\right)\right|^{q}  d\vec{y}  \bigg) 
				^\frac{2}{q} \frac{dt}{t}\Bigg) ^{\frac{1}{2}}dz \\
				&\leq \frac{C}{|Q|} \int_Q\sum_{k=1}^{\infty} C_k 
				2^{\frac{-kmn}{q'}}|z-z_0|^{\frac{-mn}{q'}}(1+2^k|z-z_0|)^{-N} \\
				& \quad \times \prod_{i=1,i\ne j}^{m} \left( \int_{2^{k}Q^*}|f_i(y_i)|^{q'} 
				dy_i\right)^{\frac1{q'}}
				\left( \int_{2^{k}Q^*}\left( |b_j(y_j)-\lambda_j| 
				|f_j(y_j)|\right)^{q'} dy_j\right)^{\frac1{q'}} dz\\
				%	&\leq \frac{C}{|Q|}\int_Q\sum_{k=1}^{\infty} \frac{C_k 
					%	|\Delta_{k+2}|^{\frac{2}{q'}}}{(2^kr)^{\frac{2n}{q'}}(1+2^kr)^{N}} \\
				%	&\ \ \ \ \times \left(\frac1{|\Delta_{k+2}|} 
				%	\int_{\Delta_{k+2}}|f_1(y_1)|^{t} 
				%	dy_1\right)^{\frac1{t}}
				%	\left(\frac1{|\Delta_{k+2}|} \int_{\Delta_{k+2}}|f_2(y_2)|^{t} 
				%	dy_2\right)^{\frac1{t}}\\
				%	&\ \ \ \ \times \left( \frac1{|\Delta_{k+2}|} 
				%	\int_{\Delta_{k+2}}|b_2(y_2)-\lambda_2|^{q'h'} dy_2\right)^{\frac1{q'h'}} 
				%	dz\\
				%	&\leq \frac{C}{|Q|}\int_Q \sum_{k=1}^{\infty} \frac{C_k 
					%		|\Delta_{k+2}|^{\frac{m}{q'}} 
					%		\varphi(\Delta_{k+2})^{\frac{m\eta}{t}}}{(2^kr)^{\frac{mn}{q'}}(1+2^kr)^{N}}
				%	\left(\frac1{\varphi(\Delta_{k+2})^{\eta}|\Delta_{k+2}|} 
				%	\int_{\Delta_{k+2}}|f_1(y_1)|^{t} dy_1\right)^{\frac1{t}}\\
				%	& \quad \times \cdots \left(\frac1{\varphi(\Delta_{k+2})^{\eta}|\Delta_{k+2}|} 
				%	\int_{\Delta_{k+2}}|f_m(y_2)|^{t} dy_m\right)^{\frac1{t}}
				%	\left( \frac1{|2^kQ^*|} \int_{2^kQ^*}|b_j(y_j)-\lambda_j|^{q'h'} 
				%	dy_j\right)^{\frac1{q'h'}}dz \\
				&\leq  \frac{C}{|Q|} \int_Q\sum_{k=1}^{\infty} \frac{C_k |2^{k}Q^*|^{\frac{m}{q'} }}{(2^{k}r)^{\frac{mn}{q'} }(1+2^{k}r)^N} \prod_{i=1,i\ne j}^{m} \left(\frac{1}{|2^{k}Q^*|}  \int_{2^{k}Q^*}|f_i(y_i)|^{q'} 
				dy_i\right)^{\frac1{q'}}\\
				& \quad \times \left( \frac{1}{|2^{k}Q^*|}\int_{2^{k}Q^*}\left( |b_j(y_j)-\lambda_j| 
				|f_j(y_j)|\right)^{q'} dy_j\right)^{\frac1{q'}}dz\\
				&\leq  \frac{C}{|Q|} \int_Q\sum_{k=1}^{\infty} \frac{C_k |2^{k}Q^*|^{\frac{m}{q'} }}{(2^{k}r)^{\frac{mn}{q'} }(1+2^{k}r)^N} \prod_{i=1}^{m} \left(\frac{1}{|2^{k}Q^*|}  \int_{2^{k}Q^*}|f_i(y_i)|^{l} 
				dy_i\right)^{\frac1{l}}\\
				& \quad \times \left( \frac{1}{|2^{k}Q^*|}\int_{2^{k}Q^*} |b_j(y_j)-\lambda_j| 
				^{q'h'} dy_j\right)^{\frac1{q'h'}}dz\\
				&\leq  \frac{C}{|Q|} \int_Q\sum_{k=1}^{\infty} \frac{C_k |2^{k}Q^*|^{\frac{m}{q'} }\varphi (2^{k}Q^*)^{\frac{m\eta }{l} }}{(2^{k}r)^{\frac{mn}{q'} }(1+2^{k}r)^N} \prod_{i=1}^{m} \left(\frac{1}{\varphi (2^{k}Q^*)^\eta |2^{k}Q^*|}  \int_{2^{k}Q^*}|f_i(y_i)|^{l} 
				dy_i\right)^{\frac1{l}}\\
				& \quad \times \left( \frac{1}{|2^{k}Q^*|}\int_{2^{k}Q^*} |b_j(y_j)-\lambda_j| 
				^{q'h'} dy_j\right)^{\frac1{q'h'}}dz\\
				&\leq \frac{C}{|Q|}\int_Q\sum_{k=1}^{\infty} \frac{C_k 
					|2^kQ^*|^{\frac{m}{q'}} 
					\varphi(2^kQ^*)^{\frac{m\eta}{l}}}{(2^kr)^{\frac{mn}{q'}}(1+2^kr)^{N}}
				\mathcal{M}_{l,\varphi,\eta}(\vec f)(x) k\| b_j \|_{BMO_\theta (\varphi)} 
				\varphi(2^kQ^*)^\theta  dz\\
				&\leq C\sum_{k=1}^{\infty } kC_k (1+2^kr)^{\frac{m\eta }{l}+\theta -N } \| 
				b_j \|_{BMO_\theta (\varphi)}  \mathcal{M}_{l,\varphi,\eta}(\vec 
				f)(x)\\
				&\leq C\| b_j \|_{BMO_\theta (\varphi)}  \mathcal{M}_{l,\varphi,\eta}(\vec 
				f)(x),
				\end {align*}
				Thus, $\uppercase\expandafter{\romannumeral+3}\leq C\| b_j \|_{BMO_\theta (\varphi)}  \mathcal{M}_{l,\varphi,\eta}(\vec f)(x)$.
				
				\noindent{\it \rm \textbf{Case 2}:} $r\geq1$. Let  
				$\widetilde{Q^*}=8Q$ and $\widetilde{\lambda _j}=(b_j)_{\widetilde{Q^*}}$. We split each $f_j$ as $f_j=f_j^0+f_j^\infty$, for each $f_j^0=f_j \chi_{\widetilde{Q^*}}$, 
				$f_j^\infty=f_j-f_j^0 $. There is
				\begin {align*}
				\prod_{j=1}^mf_j(y_j)
				&=\prod_{j=1}^m f_j^{0}(y_j)+\sum_{(\alpha_1,\ldots,\alpha_m)\in 
					\mathscr{L}}f_1^{\alpha_1}(y_1)\cdots f_m^{\alpha_m}(y_m),
				\end {align*}
				where  $\mathscr{L}:=\{(\alpha_1,\ldots,\alpha_m)$: there is at least one 
				$\alpha_j=\infty\}$.
				
				We deduce that
				\begin{align*}
					T_{b_j}^{j} (\vec f)(z)
%					&=
%					\Bigg(\int_{0}^{\infty} \bigg | 
%					\int_{\left(\mathbb{R}^{n}\right)^{m} }K_{t}(z, \vec{y})\big ( b_j(z)-b_{j}(y_j))\prod_{i=1}^m f_i(y_i)
%					d \vec{y} \bigg | ^{2} \frac{d 
%						t}{t}\Bigg)^{\frac{1}{2}}\\
					&=\Bigg(\int_{0}^{\infty} \bigg | 
					\int_{\left(\mathbb{R}^{n}\right)^{m} }K_{t}(z, \vec{y})\big ( b_j(z)-\widetilde{\lambda }_j)\prod_{i=1}^m f_i(y_i)\\
					& \quad  + K_{t}(z, \vec{y})(\widetilde{\lambda }_j-b_j(y_j))\prod_{i=1}^m f_i(y_i)d \vec{y} \bigg | ^{2} \frac{d t}{t}\Bigg)^{\frac{1}{2}}\\
%					&=\Bigg(\int_{0}^{\infty} \bigg | 
%					\int_{\left(\mathbb{R}^{n}\right)^{m} }K_{t}(z, \vec{y})\big ( b_j(z)-\widetilde{\lambda }_j)\prod_{i=1}^m f_i(y_i)
%					+ K_{t}(z, \vec{y})(\widetilde{\lambda }_j-b_j(y_j))\\
%					& \quad \prod_{i=1}^m f_i^0(y_i)+ \sum_{(\alpha_1,\ldots,\alpha_m)\in \mathscr{L}}K_{t}(z, \vec{y})(\widetilde{\lambda }_j-b_j(y_j))\prod_{i=1}^m f_i^{\alpha_i} (y_i)d \vec{y} \bigg | ^{2} \frac{d 
%						t}{t}\Bigg)^{\frac{1}{2}}\\
					& \le \Bigg(\int_{0}^{\infty} \bigg | 
					\int_{\left(\mathbb{R}^{n}\right)^{m} }K_{t}(z, \vec{y})\big ( b_j(z)-\widetilde{\lambda }_j)\prod_{i=1}^m f_i(y_i)
					d \vec{y} \bigg | ^{2} \frac{d 
						t}{t}\Bigg)^{\frac{1}{2}}\\
					& \quad + \Bigg(\int_{0}^{\infty} \bigg | 
					\int_{\left(\mathbb{R}^{n}\right)^{m} } K_{t}(z, \vec{y})(b_j(y_j)-\widetilde{\lambda }_j)\prod_{i=1}^m f_i^0(y_i)d \vec{y} \bigg | ^{2} \frac{d 
						t}{t}\Bigg)^{\frac{1}{2}}\\
					& \quad +	\sum_{(\alpha_1,\ldots,\alpha_m)\in \mathscr{L}}\Bigg(\int_{0}^{\infty} \bigg | 
					\int_{\left(\mathbb{R}^{n}\right)^{m} }K_{t}(z, \vec{y})(b_j(y_j)-\widetilde{\lambda }_j)\prod_{i=1}^m f_i^{\alpha_i} (y_i)d \vec{y} \bigg | ^{2} \frac{d 
						t}{t}\Bigg)^{\frac{1}{2}}.
					%&\le |T (f_1^0,\dots ,(b_j(x)-b_j)f_j^0,\dots ,f_m^0)(x)| \\
					%& \quad +\sum_{(\alpha_1,\ldots,\alpha_m)\in 
					%	\mathscr{L}}  |T (f_1^{\alpha_1},\dots ,(b_j(x)-b_j)f_j^{\alpha_j},\dots ,f_m^{\alpha _m})(x)|\\
				%&\le |T (f_1^0,\dots ,(b_j(x)-\lambda _j)f_j^0,\dots ,f_m^0)(x)|+|T (f_1^0,\dots ,(\lambda _j-b_j)f_j^0,\dots ,f_m^0)(x)| \\
				%& \quad +\sum_{(\alpha_1,\ldots,\alpha_m)\in 
				%	\mathscr{L}} |T (f_1^{\alpha_1},\dots ,(b_j(x)-\lambda _j)f_j^{\alpha_j},\dots ,f_m^{\alpha _m})(x)|\\
			%	& \quad +\sum_{(\alpha_1,\ldots,\alpha_m)\in 
				%		\mathscr{L}} |T (f_1^{\alpha_1},\dots ,(\lambda _j-b_j)f_j^{\alpha_j},\dots ,f_m^{\alpha _m})(x)|\\
			%&= |b_j(x)-\lambda _j||T (f_1^0,\dots ,f_j^0,\dots ,f_m^0)(x)|+|T (f_1^0,\dots ,(\lambda _j-b_j)f_j^0,\dots ,f_m^0)(x)| \\
			%& \quad +\sum_{(\alpha_1,\ldots,\alpha_m)\in 
			%	\mathscr{L}} |(b_j(x)-\lambda _j)||T (f_1^{\alpha_1},\dots ,f_j^{\alpha_j},\dots ,f_m^{\alpha _m})(x)|\\
		%& \quad +\sum_{(\alpha_1,\ldots,\alpha_m)\in 
		%	\mathscr{L}} |T (f_1^{\alpha_1},\dots ,(\lambda _j-b_j)f_j^{\alpha_j},\dots ,f_m^{\alpha _m})(x)|
	\end {align*}
	We have
	\begin  {align*}
	&\left(\frac1{\varphi (Q)^{\eta }|Q|}\int_Q |T_{b_j}^{j}(\vec f)(z)|^\delta dz\right)^{\frac1\delta}
	\leq \frac{C}{\varphi (Q)^{\eta /\delta }}  \left(\frac1{|Q|}\int_Q |(b_j(z)-\widetilde{\lambda }_j)T(\vec f)(z)|^\delta 
	dz\right)^{\frac1\delta}\\
	& \quad + \frac{C}{\varphi (Q)^{\eta /\delta }}  \left(\frac1{|Q|}\int_Q |T(f_1^{0},\dots ,f_{j-1}^{0},(b_j-\widetilde{\lambda }_j)f_{j}^{0},f_{j+1}^{0},\dots ,f_m^{0})(z)|^\delta 
	dz\right)^{\frac1\delta}\\
	%&\quad + \sum_{(\alpha_1,\ldots,\alpha_m)\in \mathscr{L}}
	%\frac{1}{\varphi (Q)^{\eta /\delta }} \left(\frac1{|Q|}\int_Q |(b_j(z)-\lambda _j)T(f_1^{\alpha _1},\dots  ,f_m^{\alpha _m})(z)|^\delta 
	%dz\right)^{\frac1\delta}\\
	&\quad +C \sum_{(\alpha_1,\ldots,\alpha_m)\in \mathscr{L}}
	\frac{1}{\varphi (Q)^{\eta /\delta }} \left(\frac1{|Q|}\int_Q |T(f_1^{\alpha _1},\dots ,f_{j-1}^{\alpha _{j-1}},(b_j-\widetilde{\lambda }_j)f_{j}^{\alpha _j},f_{j+1}^{\alpha _{j+1}},\dots ,f_m^{\alpha _m})(z)|^\delta 
	dz\right)^{\frac1\delta}\\
	&:=\uppercase\expandafter{\romannumeral+4}+\uppercase\expandafter{\romannumeral+5}+\uppercase\expandafter{\romannumeral+6}\\
	&:=\uppercase\expandafter{\romannumeral+4}+\uppercase\expandafter{\romannumeral+5}+\sum_{(\alpha_1,\ldots,\alpha_m)\in \mathscr{L}}\uppercase\expandafter{\romannumeral+6}_{\alpha_1\ldots\alpha_m}.
	\end {align*}
	
	To estimate $\uppercase\expandafter{\romannumeral+4}$, since $\eta > \theta /\left ( 1/\delta -1/\varepsilon  \right ) $, the following estimate can be obtained similar to \eqref{3.2}.
	\begin  {align*}
	\uppercase\expandafter{\romannumeral+4} \leq C\|b_j\|_{BMO_\theta(\varphi)} 
	M_{\varepsilon,\varphi,\eta}^\triangle 
	(T(\vec{f}))(x).
	\end {align*}
	
	%Let $\widetilde{Q}^*:=8\sqrt n Q$, $\lambda_j:=(b_j)_{\widetilde{Q}^*}$ be the 
	%average of $b_j$, $j=1,2,\dots ,m$. We split each $f_j$ as $f_j=f_j^0+f_j^\infty$, for 
	%each $f_j^0=f_j \chi_{\widetilde{Q}^*}$, $ f_j^\infty=f_j-f_j^0 $.
	Similar to the estimate of $\uppercase\expandafter{\romannumeral+2}$, since $\eta > \frac{\theta }{\frac{1}{\delta }-\frac{1}{\varepsilon }  } > \frac{\theta }{\frac{1}{\delta } -\frac{m}{l} } $, we have
	\begin  {align*}
	\uppercase\expandafter{\romannumeral+5} \leq C \| b_j\|_{BMO_\theta(\varphi)}\mathcal{M}_{l,\varphi,\eta}(\vec 
	f)(x).
	\end {align*}
	%\begin{align*}
	%	&\frac{1}{\varphi (Q)^{\eta /\delta }} \left(\frac1{|Q|}\int_Q |T(f_1^{\alpha _1},\dots ,f_{j-1}^{\alpha _j-1},(b_j-\lambda _j)f_{j}^{\alpha _j},f_{j+1}^{\alpha _j+1},\dots ,f_m^{\alpha _m})(z)|^\delta 
	%	dz\right)^{\frac1\delta}\\
	%	&\leq \frac{1}{\varphi (Q)^{\eta /\delta }|Q|} \int_{Q}\left(\int_{0}^{\infty} \left | 
	%	\int_{\left(\mathbb{R}^{n}\right)^{m} 
		%	}K_{t}(z, \vec{y})(b_j(y_j)-\lambda _j) \prod_{i=1}^{m} 
	%	\left|f_{i}^{\alpha _i}\left(y_{i}\right)\right| d \vec{y} \right | ^{2} \frac{d 
		%		t}{t}\right)^{\frac{1}{2}} d z\\
	%\end {align*}
	
	For the term $\uppercase\expandafter{\romannumeral+6}$, we obtain $ 
{\textstyle \sum_{j=1}^{m}} |z-y_j|\sim 2^k r$ for $z\in Q$, $k\in \mathbb{N}_+$ and any $\vec y\in (2^{k+3}Q)^m\setminus (2^{k+2}Q)^m$. Taking 
	$N>m\eta/l+\theta$, using H\"{o}lder's inquality, Minkowski’s inequality and 
	the size condition, we have
	\begin{align*}
%		&\frac{1}{\varphi (Q)^{\eta /\delta }} \left(\frac1{|Q|}\int_Q |T(f_1^{\alpha _1},\dots ,f_{j-1}^{\alpha _{j-1}},(b_j-\widetilde{\lambda }_j)f_{j}^{\alpha _j},f_{j+1}^{\alpha _{j+1}},\dots ,f_m^{\alpha _m})(z)|^\delta 
%		dz\right)^{\frac1\delta}\\
\uppercase\expandafter{\romannumeral+6}_{\alpha_1\ldots\alpha_m}
		&\leq \frac{1}{\varphi (Q)^{\eta /\delta }|Q|}  \int_{Q}\left(\int_{0}^{\infty} \left | 
		\int_{\left(\mathbb{R}^{n}\right)^{m} 
		}K_{t}(z, \vec{y})(b_j(y_j)-\widetilde{\lambda }_j) \prod_{i=1}^{m} 
		f_{i}^{\alpha _i}\left(y_{i}\right) d \vec{y} \right | ^{2} \frac{d 
			t}{t}\right)^{\frac{1}{2}} d z\\
		&\leq  \frac{C}{\varphi(Q)^{\eta/\delta}|Q|} 
		\int_{Q}\int_{\left(\mathbb{R}^{n}\right)^{m}\backslash(\widetilde{Q^*})^m}
		\Bigg(\int_{0}^{\infty}\left|K_{t}(z,\vec{y})\right|^2\frac{dt}{t}\Bigg)^{\frac{1}{2}}\\
		& \quad \times |b_j(y_j)-\widetilde{\lambda }_j|\prod_{i=1}^{m}|f_{i}(y_{i})|d\vec{y} dz\\
		%&\leq \frac{C}{\varphi(Q)^{\eta/\delta-\theta}|Q|} 
		%\int_{Q}\int_{\left(\mathbb{R}^{n}\right)^{2}\backslash(\widetilde{Q}^*)^2}\frac{|f_1(y_1)|
			%|(b_2(y_2)-\lambda_2| 
			%|f_2(y_2)|}{(\sum_{j=1}^2|z-y_j|)^{2n}(1+\sum_{j=1}^2|z-y_j|)^N}dy_1dy_2dz\\
		&\leq\frac{C}{\varphi(Q)^{\eta/\delta}|Q|} 
		\int_Q \sum_{k=1}^{\infty} \int_{(2^{k+3}Q)^m\setminus (2^{k+2}Q)^m}
		\frac{|b_j(y_j)-\widetilde{\lambda }_j|\prod_{i=1}^{m}|f_{i}(y_{i})|d\vec {y}dz}{(\sum_{j=1}^m|z-y_j|)^{mn}(1+\sum_{j=1}^m|z-y_j|)^N}\\
		%&\leq \frac{C}{\varphi(Q)^{\eta/\delta-\theta}}\sum_{k=1}^{\infty 
		%}\frac{|\Phi_{k+3}|^2}{(2^k\sqrt n r)^{2n}(1+2^k\sqrt n r)^{N}} 
	%\left(\frac1{|\Phi_{k+3}|}\int_{\Phi_{k+3}}|f_1(y_1)|dy_1\right)\\
	%& \ \ \ \ \times 
	%\left(\frac1{|\Phi_{k+3}|}\int_{\Phi_{k+3}}|b_2(y_2)-\lambda_2||f_2(y_2)|dy_2\right)\\
%	&\leq \frac{C}{\varphi(Q)^{\eta/\delta}|Q|}\int_Q\sum_{k=1}^{\infty }\frac{1}{(2^{k}r)^{mn}(1+2^{k}r)^{N}}\prod_{i=1,i\ne j}^{m}  
%	\int_{2^{k+3}Q}|f_i(y_i)| dy_i \\
%	& \quad \times \int_{2^{k+3}Q}|(b_j(y_j)-\widetilde{\lambda }_j)f_j(y_j)| dy_jdz\\
	&\leq \frac{C}{\varphi(Q)^{\eta/\delta}|Q|}\int_Q\sum_{k=1}^{\infty }\frac{|2^{k+3}Q|^m}{(2^{k}r)^{mn}(1+2^{k}r)^{N}}\prod_{i=1,i\ne j}^{m}  
	\left ( \frac{1}{|2^{k+3}Q|} \int_{2^{k+3}Q}|f_i(y_i)| dy_i \right )  \\
	& \quad \times \left ( \frac{1}{|2^{k+3}Q|} \int_{2^{k+3}Q}|(b_j(y_j)-\widetilde{\lambda }_j)f_j(y_j)| dy_j \right ) dz\\
	&\leq  \frac{C}{\varphi(Q)^{\eta/\delta}|Q|} 
	\int_{Q}\sum_{k=1}^{\infty }\frac{|2^{k+3}Q|^m}{(2^{k}r)^{mn}(1+2^{k}r)^{N}}\prod_{i=1}^{m}  
	\left ( \frac{1}{|2^{k+3}Q|} \int_{2^{k+3}Q}|f_i(y_i)|^{l} dy_i \right )^{\frac{1}{l} }\\
	& \quad \times \left (  \frac{1}{|2^{k+3}Q|} \int_{2^{k+3}Q}|b_j(y_j)-\widetilde{\lambda }_j|^{l'}
	dy_j \right )^{\frac{1}{l'} } dz\\
	&\leq\frac{C}{\varphi(Q)^{\eta/\delta}|Q|} 
	\int_{Q}\sum_{k=1}^{\infty }\frac{\varphi (2^{k+3}Q)^{m\eta/l}|2^{k+3}Q|^m}{(2^{k}r)^{mn}(1+2^{k}r)^{N}}\\
	& \quad \times \prod_{i=1}^{m}  
	\left ( \frac{1}{\varphi (2^{k+3}Q)^\eta |2^{k+3}Q|} \int_{2^{k+3}Q}|f_i(y_i)|^{l} dy_i \right )^{\frac{1}{l} }\\
	& \quad \times \left (  \frac{1}{|2^{k+3}Q|} \int_{2^{k+3}Q}|b_j(y_j)-\lambda _j|^{l'} dy_j \right )^{\frac{1}{l'} }dz \\
	%&\leq C\sum_{k=1}^{\infty }\frac{|\Phi_{k+3}|^m 
	%	\varphi(\Phi_{k+3})^{m\eta/t}}{(2^k\sqrt n r)^{mn}(1+2^k\sqrt n 
	%	r)^{N}}\left(\frac1{\varphi(\Phi_{k+3})^\eta|\Phi_{k+3}|} 
%\int_{\Phi_{k+3}}|f_1(y_1)|^{t} dy_1\right)^{\frac1{t}}\\
%& \quad \times  \cdots \left(\frac1{\varphi(\Phi_{k+3})^\eta|\Phi_{k+3}|} 
%\int_{\Phi_{k+3}}|f_m(y_m)|^{t} dy_m\right)^{\frac1{t}}
%\left( \frac1{|\Phi_{k+3}|} \int_{\Phi_{k+3}}|b_j(y_j)-\lambda_j|^{t'} 
%dy_j\right)^{\frac1{t'}}\\
&\leq  C\sum_{k=1}^{\infty }\frac{\varphi (2^{k+3}Q)^{m\eta/l}|2^{k+3}Q|^m}{(2^{k}r)^{mn}(1+2^{k}r)^{N}}\mathcal{M}_{l,\varphi,\eta}(\vec f)(x) k\| b_j \|_{BMO_\theta 
	(\varphi)} \varphi(2^{k+3}Q)^\theta\\
&\leq C\sum_{k=1}^{\infty } k(2^{k} )^{m\eta /l+\theta -N 
} \| b_j \|_{BMO_\theta (\varphi)}  \mathcal{M}_{l,\varphi,\eta}(\vec 
f)(x)\\
&\leq  C\| b_j \|_{BMO_\theta (\varphi)}  \mathcal{M}_{l,\varphi,\eta}(\vec 
f)(x).
\end {align*}

In conclusion, if $r<1$, we have
$$\left(\frac1{|Q|}\int_Q |T_{b_j}^{j}(\vec f)(z)-A_j|^\delta dz\right)^{\frac1\delta}
\le C  \| b_j \|_{BMO_{\theta}(\varphi)}
\left(M_{\varepsilon,\varphi,\eta}^\triangle (T(\vec 
f))(x)+\mathcal{M}_{l,\varphi,\eta}(\vec f)(x)\right),$$
then
\begin{align*}
	&\mathop{\inf}\limits_{ 
		c}\left(\frac1{|Q|}\int_Q \Big||T_{\sum  \vec b}(\vec f)(z)|^\delta -c\Big| dz\right)^{\frac1\delta}
	\le \left(\frac1{|Q|}\int_Q \Big||T_{\sum  \vec b}(\vec f)(z)|^\delta -|\sum_{j=1}^{m} A_j|^{\delta }\Big|  dz\right)^{\frac1\delta}
	\\
	&\le  \left(\frac1{|Q|}\int_Q | \sum_{j=1}^{m} T_{b_j}^{j}
	(\vec f)(z)-\sum_{j=1}^{m} A_j|^\delta  dz\right)^{\frac1\delta}
			\le C\sum_{j=1}^{m} \left(\frac1{|Q|}\int_Q | T_{b_j}^{j}
			(\vec f)(z)-A_j|^\delta  dz\right)^{\frac1\delta} \\
	&\le C \sum_{j=1}^{m}  \| b_j \|_{BMO_{\theta}(\varphi)}
	\left(M_{\varepsilon,\varphi,\eta}^\triangle (T(\vec 
	f))(x)+\mathcal{M}_{l,\varphi,\eta}(\vec f)(x)\right).
\end{align*}
	
If $r\ge 1$, we obtain
$$\left(\frac1{\varphi (Q)^{\eta }|Q|}\int_Q |T_{b_j}^{j}(\vec f)(z)|^\delta dz\right)^{\frac1\delta}
\le C  \| b_j \|_{BMO_{\theta}(\varphi)}
\left(M_{\varepsilon,\varphi,\eta}^\triangle (T(\vec 
f))(x)+\mathcal{M}_{l,\varphi,\eta}(\vec f)(x)\right),$$
then
\begin{align*}
	&\left(\frac1{\varphi (Q)^{\eta }|Q|}\int_Q |T_{\sum  \vec b}(\vec f)(z)|^\delta dz\right)^{\frac1\delta}
	\le  C\sum_{j=1}^{m} \left(\frac1{\varphi (Q)^{\eta }|Q|}\int_Q |
	T_{b_j}^{j}(\vec f)(z)|^\delta dz\right)^{\frac1\delta}\\
	&\le C  \sum_{j=1}^{m}\| b_j \|_{BMO_{\theta}(\varphi)}
	\left(M_{\varepsilon,\varphi,\eta}^\triangle (T(\vec 
	f))(x)+\mathcal{M}_{l,\varphi,\eta}(\vec f)(x)\right).
\end{align*}
Combining the above two cases, so
\begin{align*}
	M_{\delta ,\varphi,\eta}^{\sharp,\triangle} (T_{\sum  \vec b}(\vec f))(x)
	&\simeq\mathop{\sup}\limits_{Q\ni  x,r<1}\mathop{\inf}\limits_{ 
		c}\Bigg(\frac1{|Q|}\int_Q \Big||T_{\sum  \vec b}(\vec f)(z)|^\delta -c\Big| dz \\
	&+
	\mathop{\sup}\limits_{Q\ni  x,r\ge 1}\frac1{\varphi (Q)^{\eta }|Q|}\int_Q |T_{\sum  \vec b}(\vec f)(z)|^\delta  dz\Bigg)^{\frac1\delta}\\
	&\le C  \sum_{j=1}^{m}\| b_j \|_{BMO_{\theta}(\varphi)}
	\left(M_{\varepsilon,\varphi,\eta}^\triangle (T(\vec 
	f))(x)+\mathcal{M}_{l,\varphi,\eta}(\vec f)(x)\right).
\end{align*}

The proof of the Lemma \ref{Lemma3.11} is finished. 
\end{proof}

\begin{lemma}\label{Lemma3.12} 
	Suppose that $T$ is the new multilinear square operator with 
	generalized kernel as in Definition \ref{definition1.3}. Suppose $\vec b\in BMO^{m}_{\vec \theta }(\varphi )$, $\vec \theta =(\theta _1,\dots ,\theta _m)$ with $\theta _j\ge 0,j=1,\dots ,m$. Let ${\textstyle 
		\sum_{k=1}^{\infty }}  k^m C_k<\infty$, $0<\delta<\varepsilon<1/m$, $q'<l<\infty$,  and 
	$\eta>(\sum_{j=1}^{m} \theta _{j})/(1/\delta -1/\varepsilon )$. There exists a positive constant $C$ such that for all $\vec 
	f=(f_1,\ldots,f_m)$ of bounded measurable functions with compact 
	suppport
	\begin{align*}
M_{\delta,\varphi,\eta}^{\sharp,\triangle}(T_{\prod \vec b}(\vec 
f))(x)&\leq C\prod_{j=1}^m \| b_j \|_{BMO_{\theta_j}(\varphi)}
\left(M_{\varepsilon,\varphi,\eta}^\triangle (T(\vec 
f))(x)+\mathcal{M}_{l,\varphi,\eta}(\vec f)(x)\right)\\
&\ \ \ \ +C\sum_{j=1}^{m-1} \sum_{\xi \in C^{m}_j}\prod_{i=1}^{j} ||b_{\xi (i)}||_{BMO_{\theta_{\xi (i)}}(\varphi ) }
M_{\varepsilon,\varphi,\eta}^\triangle (T_{\prod \vec b_{\xi '}}(\vec 
f))(x).
	\end{align*}
	\end {lemma}
 \begin{proof}
	We only prove the case $m=2$ and  $\theta_1=\theta_2=\theta$  for 
	simplicity. In the general case, there are no different computations but only 
	more complicated.
	
	 Fix $x\in \mathbb{R}^n$, for any dyadic cube $Q:=Q(x_0,r)\ni x$,  
	%for any constants $a_1, a_2$ and $\left | \left | a_{1}  \right | ^{t} - | 
	%%%\left | a_{2}  \right | ^{t} \right | \le \left | a_{1}- a_{2}  \right | 
	%%%^{t} 
	%%%%$ for $0<t<1$. 
	we consider the following two cases of the sidelength $r$ : $r<1$ and 
	$r\geq1$.
	
	\noindent{\it \rm \textbf{Case 1}:} $r<1$. Let $Q^*:=14n\sqrt{2n}Q$, $z_{0}\in 4Q \setminus 3Q$ 
	and $\lambda_j:=(b_j)_{Q^*}$ be the average on $Q^*$ of $b_j$, $j=1,2$. 
	We split each $f_j$ as $f_j=f_j^0+f_j^\infty$, for each $f_j^0=f_j \chi_{Q^*}$, 
	$f_j^\infty=f_j-f_j^0 $. There is
	\begin{align*}
		\prod_{j=1}^2 
		f_j(y_j)
		=\prod_{j=1}^2 f_j^{0}(y_j)+\sum_{(\alpha_1,\alpha_2)\in 
			\mathscr{L}}f_1^{\alpha_1}(y_1) f_2^{\alpha_2}(y_2),
		\end {align*}
		where  $\mathscr{L}:=\{(\alpha_1,\alpha_2)$: there is at least one  
		$\alpha_j=\infty\}$.
		
		Take 
		% $$a_1:=\left(\int_{0}^{\infty} \left | \int_{\left(\mathbb{R}^{n}\right)^{2} 
			% }K_{t}\left(z_{0} 
		% ,y_1,y_2\right) \sum_{(\alpha_1,\alpha_2)\in \mathscr{L}}
		% (b_{1}(z)-\lambda _1)(b_{2}(z)-\lambda _2)f_{1}^{\alpha_1}\left(y_{1}\right)  f_{2}^{\alpha_2}\left(y_{2}\right)d y_1dy_2 \right | ^{2} \frac{d 
			% 	t}{t}\right)^{\frac{1}{2}}$$
		\begin{align*}
			A:&=\Bigg(\int_{0}^{\infty} \Big | \int_{\left(\mathbb{R}^{n}\right)^{2} 
			}K_{t}\left(z_{0} 
			,y_1,y_2\right) \sum_{(\alpha_1,\alpha_2)\in \mathscr{L}}
			(b_{1}(y_1)-\lambda _1)(b_{2}(y_2)-\lambda _2)\\
			&\quad  \times  f_{1}^{\alpha_1}\left(y_{1}\right)  f_{2}^{\alpha_2}\left(y_{2}\right)d y_1dy_2 \Big | ^{2} \frac{d t}{t}\Bigg)^{\frac{1}{2}}.
			\end {align*}
			This yields
			\setlength\abovedisplayskip{10pt}
				\setlength\belowdisplayskip{10pt}
\begin{align*}
				&\bigg|T_{\prod \vec b}(\vec 
				f)(z)-A\bigg|\\
				& \leq \Bigg(\int_{0}^{\infty} \Big | 
				\int_{\left(\mathbb{R}^{n}\right)^{2} }\Big(K_{t}(z, y_1,y_2) \prod_{j=1}^2 (b_j(z)-b_j(y_j))f_j(y_j)\\
				& \quad - \sum_{(\alpha_1,\alpha_2)\in \mathscr{L}}
				K_t(z_0,y_1,y_2)(b_{1}(y_1)-\lambda _1)(b_{2}(y_2)-\lambda _2)f_{1}^{\alpha_1}\left(y_{1}\right)  f_{2}^{\alpha_2}\left(y_{2}\right)\Big)d y_1dy_2 \Big | ^{2} \frac{d 
					t}{t}\Bigg)^{\frac{1}{2}}\\
%				&=\Bigg(\int_{0}^{\infty} \Big | 
%				\int_{\left(\mathbb{R}^{n}\right)^{2} }K_{t}(z, y_1,y_2)\Big[  
%				(b_1(z)-\lambda _1)(\lambda _2-b_2(z))+(b_1(z)-\lambda _1)(b _2(z)-b_2(y_2)) \\
%				& \quad +(b_2(z)-\lambda _2)(b_1(z)-b_1(y_1)) 
%				+(b_1(y_1)-\lambda _1)(b_2(y_2)-\lambda _2)\Big ]f_1(y_1)f_2(y_2)\\
%				& \quad - \sum_{(\alpha_1,\alpha_2)\in \mathscr{L}}
%				K_t(z_0,y_1,y_2)(b_{1}(y_1)-\lambda _1)(b_{2}(y_2)-\lambda _2)f_{1}^{\alpha_1}\left(y_{1}\right)  f_{2}^{\alpha_2}\left(y_{2}\right)d y_1dy_2 \Big | ^{2} \frac{d 
%					t}{t}\Bigg)^{\frac{1}{2}}\\
				&=\Bigg(\int_{0}^{\infty} \Big | 
				\int_{\left(\mathbb{R}^{n}\right)^{2} }\big(K_{t}(z, y_1,y_2)  
				(b_1(z)-\lambda _1)(\lambda _2-b_2(z))f_1(y_1)f_2(y_2) \\
				& \quad + K_{t}(z, y_1,y_2) (b_1(z)-\lambda _1)(b _2(z)-b_2(y_2))f_1(y_1)f_2(y_2) \\
				& \quad + K_{t}(z, y_1,y_2) (b_2(z)-\lambda _2)(b_1(z)-b_1(y_1))f_1(y_1)f_2(y_2)  \\
				& \quad + K_{t}(z, y_1,y_2) (b_1(y_1)-\lambda _1)(b_2(y_2)-\lambda _2)f_1^{0}(y_1)f_2^{0}(y_2) \\
				& \quad + \sum_{(\alpha_1,\alpha_2)\in \mathscr{L}}
				K_t(z,y_1,y_2)(b_{1}(y_1)-\lambda _1)(b_{2}(y_2)-\lambda _2)f_{1}^{\alpha_1}\left(y_{1}\right)  f_{2}^{\alpha_2}\left(y_{2}\right)\\
				& \quad - \sum_{(\alpha_1,\alpha_2)\in \mathscr{L}}
				K_t(z_0,y_1,y_2)(b_{1}(y_1)-\lambda _1)(b_{2}(y_2)-\lambda _2)f_{1}^{\alpha_1}\left(y_{1}\right)  f_{2}^{\alpha_2}\left(y_{2}\right)
				\big )d y_1dy_2 \Big | ^{2} \frac{d 
					t}{t}\Bigg)^{\frac{1}{2}}\\
				&\leq \Bigg (\int_{0}^{\infty} \Big | 
				\int_{\left(\mathbb{R}^{n}\right)^{2} 
				}K_{t}(z, y_1,y_2)  \prod_{j=1}^{2} 
				(b_j(z)-\lambda _j)f_j(y_j) d y_j \Big | ^{2} \frac{d 
					t}{t}\Bigg)^{\frac{1}{2}}\\
				& \quad + \Bigg (\int_{0}^{\infty} \Big | 
				\int_{\left(\mathbb{R}^{n}\right)^{2} 
				}K_{t}(z, y_1,y_2)  
				(b_1(z)-\lambda _1)(b _2(z)-b_2(y_2))f_1(y_1)f_2(y_2) d y_1dy_2 \Big | ^{2} \frac{d 
					t}{t}\Bigg)^{\frac{1}{2}}\\
				& \quad + 	\Bigg (\int_{0}^{\infty} \Big | 
				\int_{\left(\mathbb{R}^{n}\right)^{2} 
				}K_{t}(z, y_1,y_2)  
				(b_2(z)-\lambda _2)(b _1(z)-b_1(y_1))f_1(y_1)f_2(y_2)d y_1dy_2 \Big | ^{2} \frac{d 
					t}{t}\Bigg)^{\frac{1}{2}}\\
				& \quad + \Bigg (\int_{0}^{\infty} \Big | 
				\int_{\left(\mathbb{R}^{n}\right)^{2} 
				}K_{t}(z, y_1,y_2) (b_1(y_1)-\lambda _1)(b_2(y_2)-\lambda _2)f_1^{0}(y_1)f_2^{0}(y_2)d y_1dy_2 \Big | ^{2} \frac{d 
					t}{t}\Bigg)^{\frac{1}{2}}\\
				& \quad + \sum_{(\alpha_1,\alpha_2)\in \mathscr{L}}\Bigg (\int_{0}^{\infty} \Big | 
				\int_{\left(\mathbb{R}^{n}\right)^{2} 
				}\left ( K_t(z,y_1,y_2)-K_t(z_0,y_1,y_2) \right ) (b_{1}(y_1)-\lambda _1)(b_{2}(y_2)-\lambda _2)\\
				& \quad \times f_{1}^{\alpha_1}\left(y_{1}\right)  f_{2}^{\alpha_2}\left(y_{2}\right)d y_1dy_2 \Big | ^{2} \frac{d t}{t}\Bigg)^{\frac{1}{2}}\\
				&=|(b_1(z)-\lambda_1)(b_2(z)-\lambda_2)|T(f_1,f_2)(z)+|b_1(z)-\lambda_1|
				T_{b_2}^{2}(f_1,f_2)(z)\\
				&\quad + |b_2(z)-\lambda_2|T_{b_1}^{1}(f_1,f_2)(z)+
				T((b_1-\lambda _1)f_1^0,(b_2-\lambda _2)f_2^0)(z)\\
				& \quad + \sum_{(\alpha_1,\alpha_2)\in \mathscr{L}}\Bigg (\int_{0}^{\infty} \Big | 
				\int_{\left(\mathbb{R}^{n}\right)^{2} 
				}\left ( K_t(z,y_1,y_2)-K_t(z_0,y_1,y_2) \right ) (b_{1}(y_1)-\lambda _1)(b_{2}(y_2)-\lambda _2)\\
				& \quad \times f_{1}^{\alpha_1}\left(y_{1}\right)  f_{2}^{\alpha_2}\left(y_{2}\right)d y_1dy_2 \Big | ^{2} \frac{d t}{t}\Bigg)^{\frac{1}{2}},
				\end {align*}
				and
				\begin{align*}
					&\left(\frac1{|Q|}\int_Q \bigg||T_{\prod \vec b}(\vec 
					f)(z)|^\delta-|A|^\delta \bigg| dz\right)^{\frac1\delta}\\
					&\leq \left(\frac1{|Q|}\int_Q \bigg|T_{\prod \vec b}(\vec 
					f)(z)-A\bigg|^\delta dz\right)^{\frac1\delta}\\
					&\leq C\left(\frac 1{|Q|}\int_Q 
					|(b_1(z)-\lambda_1)(b_2(z)-\lambda_2)T(f_1,f_2)(z)|^\delta 
					dz\right)^{\frac1\delta}\\
					& \quad +C\left(\frac1{|Q|}\int_Q 
					|(b_1(z)-\lambda_1)T_{b_2}^2(f_1,f_2)(z)|^\delta 
					dz\right)^{\frac1\delta}\\%2
					& \quad  + C\left(\frac1{|Q|}\int_Q 
					|(b_2(z)-\lambda_2)T_{b_1}^1(f_1,f_2)(z)|^\delta 
					dz\right)^{\frac1\delta}\\%3
					& \quad +C\left(\frac1{|Q|}\int_Q 
					|T((b_1-\lambda_1)f_1^0,(b_2-\lambda_2)f_2^0)(z)|^\delta 
					dz\right)^{\frac1\delta}\\%4
					& \quad + C\sum_{(\alpha_1,\alpha_2)\in \mathscr{L}}
					\frac{1}{|Q|} \int_{Q}\Bigg (\int_{0}^{\infty} \Big |
					\int_{\left(\mathbb{R}^{n}\right)^{2} 
					}\left(K_{t}(z, y_1,y_2)-K_{t}\left(z_{0} 
					,y_1,y_2\right)\right) \\
					& \quad \times (b_1(y_1)-\lambda _1)(b_2(y_2)-\lambda _2)f_{1}^{\alpha _1}\left(y_{1}\right) 
					f_{2}^{\alpha _2}\left(y_{2}\right)d y_1dy_2 \Big | ^{2} \frac{d t}{t}\Bigg)^{\frac{1}{2}} d z \\
					&:=\sum_{i=1}^{5}  \mathrm{I}_i
					:=\sum_{i=1}^{4}  \mathrm{I}_i+\sum_{(\alpha_1,\alpha_2)\in \mathscr{L}}\mathrm{I}_{5\alpha _1\alpha _2}.
					\end {align*}
					
					Since $0< \delta < \varepsilon < \infty $, choose $1< q_3< \min\left \{ \varepsilon /\delta , 1/(1-\delta) \right \} $ so that $q_3<\varepsilon /\delta $ and $\delta q_3^{'}> 1$. Take $q_1,q_2>1$ such that $1/q_1+1/q_2+1/q_3=1$. Using  
					H\"{o}lder's inequality, we deduce that
					\begin{align*}
						\mathrm{I}_1 &\leq C\left( \frac1{|Q|}\int_Q 
						|b_1(z)-\lambda_1|^{\delta q_1}dz\right)^{\frac1{\delta q_1}}
						\left( \frac1{|Q|}\int_Q |b_2(z)-\lambda_2|^{\delta 
							q_2}dz\right)^{\frac1{\delta q_2}}\\
						&  \quad \times \left( \frac1{|Q|}\int_Q |T(f_1,f_2)(z)|^{\delta 
							q_3}dz\right)^{\frac1{\delta q_3}}\\
						&\leq C\|b_1\|_{BMO_\theta(\varphi)} \|b_2\|_{BMO_\theta(\varphi)}\left( 
						\frac1{\varphi(Q)^\eta|Q|}\int_Q |T(f_1,f_2)(z)|^\varepsilon 
						dz\right)^{\frac{1}{\varepsilon} }\\
						&\leq C\|b_1\|_{BMO_\theta(\varphi)} 
						\|b_2\|_{BMO_\theta(\varphi)}M_{\varepsilon,\varphi,\eta}^\triangle 
						(T(f_1,f_2))(x).
						\end {align*}
						
						Since $ \mathrm{I}_{2} $ and $ \mathrm{I}_{3}$ are symmetrical, we only 
						estimate $ \mathrm{I}_{2} $. Using H\"{o}lder's inequality and 
						Lemma \ref{Lemma3.1}, we obtain
						\begin{align*}
							\mathrm{I}_2
							&\leq C\left(\frac1{|Q|}\int_Q |b_1(z)-\lambda_1|^{\delta q'_3}dz 
							\right)^{\frac1{\delta q'_3}} 
							\left(\frac1{|Q|}\int_Q |T_{b_2}^2(f_1,f_2)(z)|^{\delta q_3}dz 
							\right)^{\frac1{\delta q_3}} \\
							&\leq C\| b_1\|_{BMO_\theta (\varphi)} M_{\varepsilon,\varphi,\eta}^\triangle (T_{  b_{2}}^{2}(
							f_1,f_2))(x).
							\end {align*}
							
							For the term $\mathrm{I}_4$, Let $1/h_1+1/h'_1=1$, $1/h_2+1/h'_2=1$, where $h_1=l/s_1$ and $h_2=l/s_2$, then $h_1,h_2>1$. By Lemma \ref{Lemma3.2}, according to 
							$T:L^{s_1}\times\cdots\times L^{s_m}\to L^{s,\infty}$, H\"{o}lder's inqualitiy 
							and Lemma \ref{Lemma3.1}, we have
							\begin{align*}
%								& \left(\frac1{|Q|}\int_Q 
%								|T((b_1(y_1)-\lambda_1)f_1^0,(b_2(y_2)-\lambda_2)f_2^0)(z)|^{\delta}dz 
%								\right)^{\frac1{\delta}} \\
\mathrm{I}_4
								&\leq C\| 
								T((b_1-\lambda_1)f_1^0,(b_2-\lambda_2)f_2^0) \|_{L^{s,\infty}(Q,\frac{dx}{|Q|})}\\
								&\leq C\left(\frac1{|Q^*|}\int_{Q^*} |(b_1(y_1)-\lambda_1)f_1(y_1)|^{s_1}dy_1 
								\right)^{\frac1{s_1}} 
								\left(\frac1{|Q^*|}\int_{Q^*} |(b_2(y_2)-\lambda_2)f_2(y_2)|^{s_2}dy_2 
								\right)^{\frac1{s_2}}  \\
								&\leq C\prod_{i=1}^{2} \left(\frac1{|Q^*|}\int_{Q^*} |f_i(y_i)|^{l}dy_i \right)^{\frac{1}{l} } 
								\left(\frac1{|Q^*|}\int_{Q^*} |b_1(y_1)-\lambda_1|^{s_1h'_1}dy_1
								\right)^{\frac1{s_1h'_1}}\\
								& \quad \times \left(\frac1{|Q^*|}\int_{Q^*} |b_2(y_2)-\lambda_2|^{s_2h'_2}dy_2 
								\right)^{\frac1{s_2h'_2}}  \\
								&\leq C  \| b_1\|_{BMO_\theta(\varphi)}\| 
								b_2\|_{BMO_\theta(\varphi)}\prod_{i=1}^{2} \left(\frac1{\varphi(Q^*)^\eta|Q^*|}\int_{Q^*} 
								|f_i(y_i)|^{l}dy_i \right)^{\frac{1}{l} } \\
								&\leq C  \| b_1\|_{BMO_\theta(\varphi)}\| b_2\|_{BMO_\theta(\varphi)}\mathcal{M}_{l,\varphi,\eta}(\vec 
								f)(x).
								\end {align*}
								
								For the term $\mathrm{I}_5$,
								let $\Delta_k:=Q(z_0, 2^{k}\sqrt{2n}|z-z_0|)$, $ z\in Q $, $k\in \mathbb{N}_+$, 
								 $h=l/q'$ and $1/h+1/h'=1$. It follows from 
								$1<q'<l$ that $h>1$. Since $z\in Q$, $z_0\in 4Q\setminus 3Q$ and $\Delta_2 
								\subset Q^*$, we obtain $(\mathbb{R}^n)^2\setminus (Q^*)^2\subset 
								(\mathbb{R}^n)^2\setminus (\Delta_2)^2$, $|z-z_0|\sim r$ and $\Delta_{k+2} 
								\subset 2^kQ^*$. Taking $N>2\eta/l+2\theta$, by Minkowski's inequality, 
								H\"{o}lder's inquality, the smoothness condition \eqref{1.5} and Lemma 
								\ref{Lemma3.1},  we conclude that
								\begin{align*}
%									&  \frac{1}{|Q|} \int_{Q}\Bigg(\int_{0}^{\infty} \bigg( 
%									\int_{\left(\mathbb{R}^{n}\right)^{2} 
%										\backslash\left(Q^{*}\right)^{2}}\left|K_{t}(z, y_1,y_2)-K_{t}\left(z_{0} 
%									,y_1,y_2\right)\right| \\
%									& \quad \times |b_1(y_1)-\lambda _1||f_{1}\left(y_{1}\right)||b_2(y_2)-\lambda 
%									_2||f_{2}(y_{2})| d y_1d y_2 \bigg)^{2} \frac{dt}{t}\Bigg)^{\frac{1}{2}} 
%									dz\\
\mathrm{I}_{5\alpha _1\alpha _2}
									&\leq \frac{1}{|Q|} \int_{Q}\Bigg(\int_{0}^{\infty} \bigg(  
									\int_{\left(\mathbb{R}^{n}\right)^{2} 
										\backslash\left(\Delta_2\right)^{2}}\left|K_{t}(z, 
									y_1,y_2)-K_{t}\left(z_{0} ,y_1,y_2\right)\right| \\
									& \quad \times |b_1(y_1)-\lambda _1||f_{1}\left(y_{1}\right)||b_2(y_2)-\lambda 
									_2||f_{2}(y_{2})| d y_1d y_2 \bigg)^{2} \frac{dt}{t}\Bigg)^{\frac{1}{2}} 
									dz\\
									%	&\leq \frac{1}{|Q|} \int_{Q}\Bigg(\int_{0}^{\infty}\sum_{k=1}^{\infty } 
									%	
									%\bigg(\int_{\left(\Delta_{k+2}\right)^{2}\backslash\left(\Delta_{k+1}\right)^{2}}\left|K_{t}(z,
									%	y_1,y_2)-K_{t}\left(z_{0},y_1,y_2\right)\right|\\
									%	&\ \ \ \ \times |f_{1}\left(y_{1}\right)||b_2(y_2)-\lambda 
									%	_2||f_{2}(y_{2})|dy_1d y_2\bigg)^{2} \frac{d t}{t}\Bigg)^{\frac{1}{2}} d z\\
									&\leq \frac{1}{|Q|} \int_{Q}\sum_{k=1}^{\infty } \Bigg(\int_{0}^{\infty} 
									\bigg(\int_{\left(\Delta_{k+2}\right)^{2} 
										\backslash\left(\Delta_{k+1}\right)^{2}}\left|K_{t}(z,y_1,y_2)-K_{t}\left
									(z_{0},y_1,y_2\right)\right| \\
									& \quad \times |b_1(y_1)-\lambda _1||f_{1}\left(y_{1}\right)||b_2(y_2)-\lambda 
									_2||f_{2}(y_{2})|dy_1dy_2\bigg)^{2} \frac{d t}{t}\Bigg)^{\frac{1}{2}} d z\\
									&\leq \frac{1}{|Q|} \int_{Q} \sum_{k=1}^{\infty } \Bigg(  
									\int_{0}^{\infty}     \bigg(\int_{\left(\Delta_{k+2}\right)^{2} 
										\backslash\left(\Delta_{k+1}\right)^{2}}\left|K_{t}(z, 
									y_1,y_2)-K_{t}\left(z_{0} ,y_1,y_2\right)\right|^{q}  d y_1dy_2  \bigg) 
									^\frac{2}{q} \\
									& \quad \times  \bigg( \int_{\left(\Delta_{k+2}\right)^{2} 
										\backslash\left(\Delta_{k+1}\right)^{2}} 
									(\prod_{i=1}^{2} |b_i(y_i)-\lambda _i||f_{i}(y_{i})|)^{q'} d y_1dy_2 \bigg) 
									^\frac{2}{q'}\frac{dt}{t}\Bigg) ^{\frac{1}{2}} d z \\
									&\leq \frac{1}{|Q|} \int_{Q} \sum_{k=1}^{\infty } \Bigg( 
									\int_{\left(\Delta_{k+2}\right)^{2}} 
									\bigg(\prod_{i=1}^{2} |b_i(y_i)-\lambda _i||f_{i}(y_{i})| \bigg)^{q'}d 
									y_1dy_2 \Bigg) ^\frac{1}{q'} \\
									& \quad \times 
									\Bigg(\int_{0}^{\infty}\bigg(\int_{\left(\Delta_{k+2}\right)^{2} 
										\backslash\left(\Delta_{k+1}\right)^{2}}\left|K_{t}(z, 
									y_1,y_2)-K_{t}\left(z_{0} ,y_1,y_2\right)\right|^{q}  d y_1dy_2  \bigg) 
									^\frac{2}{q} \frac{dt}{t}\Bigg) ^{\frac{1}{2}}dz \\
									&\leq \frac{C}{|Q|} \int_Q \sum_{k=1}^{\infty} C_k 
									2^{\frac{-2kn}{q'}}|z-z_0|^{\frac{-2n}{q'}}(1+2^k|z-z_0|)^{-N} \\
									& \quad \times \left( \int_{2^kQ^*}(|b_1(y_1)-\lambda_1||f_1(y_1)|)^{q'} 
									dy_1\right)^{\frac1{q'}}
									\left( \int_{2^kQ^*}\left( |b_2(y_2)-\lambda_2| 
									|f_2(y_2)|\right)^{q'} dy_2\right)^{\frac1{q'}}dz\\
									&\leq \frac{C}{|Q|}\int_Q \sum_{k=1}^{\infty} \frac{C_k 
										|2^kQ^*|^{\frac{2}{q'}} 
									}{(2^kr)^{\frac{2n}{q'}}(1+2^kr)^{N}}
									\left(\frac1{|2^kQ^*|} 
									\int_{2^kQ^*}|(b_1(y_1)-\lambda _1)f_1(y_1)|^{q'} dy_1\right)^{\frac1{q'}}\\
									& \quad \times \left(\frac1{|2^kQ^*|} 
									\int_{2^kQ^*}|(b_2(y_2)-\lambda _2)f_2(y_2)|^{q'} dy_2\right)^{\frac1{q'}}dz
									\\
									&\leq \frac{C}{|Q|}\int_Q \sum_{k=1}^{\infty} \frac{C_k 
										|2^kQ^*|^{\frac{2}{q'}} 
									}{(2^kr)^{\frac{2n}{q'}}(1+2^kr)^{N}}
									\prod_{i=1}^{2} \left(\frac1{|2^kQ^*|} 
									\int_{2^kQ^*}|f_i(y_i)|^{l} dy_i\right)^{\frac1{l}}\\
									& \quad \times \prod_{j=1}^{2} \left(\frac1{|2^kQ^*|} 
									\int_{2^kQ^*}|b_j(y_j)-\lambda _j|^{q'h'} dy_2\right)^{\frac1{q'h'}}dz\\
									&\leq \frac{C}{|Q|}\int_Q \sum_{k=1}^{\infty} \frac{C_k 
										|2^kQ^*|^{\frac{2}{q'} } \varphi (2^kQ^*)^{\frac{2\eta }{l} }
									}{(2^kr)^{\frac{2n}{q'}}(1+2^kr)^{N}}
									\prod_{i=1}^{2} \left(\frac1{\varphi (2^kQ^*)^{\eta }|2^kQ^*|} 
									\int_{2^kQ^*}|f_i(y_i)|^{l} dy_i\right)^{\frac1{l}}\\
									& \quad \times \prod_{j=1}^{2} \left(\frac1{|2^kQ^*|} 
									\int_{2^kQ^*}|b_j(y_j)-\lambda _j|^{q'h'} dy_2\right)^{\frac1{q'h'}}dz
									\\
									&\leq C\sum_{k=1}^{\infty} \frac{C_k 
										|2^kQ^*|^{\frac{2}{q'}} 
										\varphi(2^kQ^*)^{\frac{2\eta}{l}}}{(2^kr)^{\frac{2n}{q'}}(1+2^kr)^{N}}
									\mathcal{M}_{l,\varphi,\eta}(\vec f)(x) k^2\| b_1 \|_{BMO_\theta (\varphi)} \| b_2 \|_{BMO_\theta (\varphi)} 
									\varphi(2^kQ^*)^{2\theta }\\
									&\leq C\sum_{k=1}^{\infty } k^{2}C_k (1+2^kr)^{\frac{2\eta }{l}+2\theta -N } \| 
									b_1\|_{BMO_\theta (\varphi)}\| 
									b_2 \|_{BMO_\theta (\varphi)}  \mathcal{M}_{l,\varphi,\eta}(\vec 
									f)(x)\\
									&\leq C\| 
									b_1 \|_{BMO_\theta (\varphi)}\| b_2 \|_{BMO_\theta (\varphi)}  \mathcal{M}_{l,\varphi,\eta}(\vec 
									f)(x).
									\end {align*}
									
									\noindent{\it \rm \textbf{Case 2}:} $r\geq1$. Let $\widetilde{Q}^*=8 Q$ and $\widetilde{\lambda_j}=(b_j)_{\widetilde{Q}^*}$ be the 
									average on $Q^*$ of $b_j$, $j=1,2$. We split each $f_j$ as $f_j=f_j^0+f_j^\infty$, for 
									each $f_j^0=f_j \chi_{\widetilde{Q}^*}$, $ f_j^\infty=f_j-f_j^0 $. Then
									\begin{align*}
										\prod_{j=1}^{2} f_{j}\left(y_{j}\right) =\prod_{j=1}^2 
										f_j^{0}(y_j)+\sum_{(\alpha_1,\alpha_2)\in 
											\mathscr{L}}f_1^{\alpha_1}(y_1) f_2^{\alpha_2}(y_2),
									\end{align*}
									where  $\mathscr{L}:=\{(\alpha_1,\alpha_2)$: there is at least one 
									$\alpha_j=\infty\}$. 
									
									We deduce that
									\begin{align*}
										&\left(\frac1{\varphi(Q)^{\eta}|Q|}\int_Q |T_{\prod \vec b}(\vec 
										f)(z)|^\delta dz\right)^{\frac1\delta}\\
										&\leq \frac C{\varphi(Q)^{\eta/\delta}}\left(\frac1{|Q|}\int_Q 
										|(b_1(z)-\widetilde{\lambda}_1)(b_2(z)-\widetilde{\lambda}_2)T(f_1,f_2)(z)|^\delta 
										dz\right)^{\frac1\delta}\\%1
										& \quad +\frac C {\varphi(Q)^{\eta/\delta}}\left(\frac1{|Q|}\int_Q 
										|(b_1(z)-\widetilde{\lambda}_1)T_{b_2}^2(f_1,f_2)(z)|^\delta 
										dz\right)^{\frac1\delta}\\%2
										& \quad +\frac C{\varphi(Q)^{\eta/\delta}}\left(\frac1{|Q|}\int_Q 
										|(b_2(z)-\widetilde{\lambda}_2)T_{b_1}^1(f_1,f_2)(z)|^\delta 
										dz\right)^{\frac1\delta}\\%3
										& \quad +\frac C{\varphi(Q)^{\eta/\delta}}\left(\frac1{|Q|}\int_Q 
										|T((b_1-\widetilde{\lambda}_1)f_1^0,(b_2-\widetilde{\lambda}_2)f_2^0)(z)|^\delta 
										dz\right)^{\frac1\delta}\\%4
										& \quad +C\sum_{(\alpha_1,\alpha_2)\in \mathscr{L}}\frac1{\varphi(Q)^{\eta/\delta}}\left(\frac1{|Q|}\int_Q 
										|T((b_1-\widetilde{\lambda}_1)f_1^{\alpha _1},(b_2-\widetilde{\lambda}_2)f_2^{\alpha _2})(z)|^\delta 
										dz\right)^{\frac1\delta}\\%4
										&:=\sum_{i=1}^{5} \widetilde{\mathrm{I}_i}
										:=\sum_{i=1}^{4} \widetilde{\mathrm{I}_i}+\sum_{(\alpha_1,\alpha_2)\in \mathscr{L}}\widetilde{\mathrm{I}_5}_{\alpha _1\alpha _2}.
										\end {align*}
										
										Since $\eta > \frac{2\theta }{\frac{1}{\delta }-\frac{1}{\varepsilon }  } $, $ \widetilde{\mathrm{I}_i}$ can be estimated using the same method of $\mathrm{I}_i$, $i=1,2,3,4$. We only estimate $\widetilde{\mathrm{I}_5}$. 
										Let $1/l+1/l'=1$. For any $(y_1,y_2)\in (2^{k+3}Q)^2\setminus (2^{k+2}Q)^2$, $k\in \mathbb{N}_+$, we obtain $ 
										{\textstyle \sum_{j=1}^{2}} |z-y_j|\sim 2^k r$, where $z\in Q$. Taking 
										$N>2\eta/l+2\theta $, using H\"{o}lder's inquality, Minkowski’s inequality, and 
										the size condition, we obtain
										\begin{align*}
											&\widetilde{\mathrm{I}_5}_{\alpha _1\alpha _2}
											\leq  \frac{C}{\varphi (Q)^{\eta /\delta }|Q|} \int_{Q}\Bigg(\int_{0}^{\infty} \bigg( 
											\int_{\left(\mathbb{R}^{n}\right)^{2} 
												\backslash(\widetilde{Q^*})^{2}}\left|K_{t}(z, y_1,y_2)\right|\prod_{i=1}^{2} |b_i(y_i)-\widetilde{\lambda}_i| \\
											& \ \ \ \ \times |f_{i}\left(y_{i}\right)| d y_1d y_2 \bigg)^{2} \frac{dt}{t}\Bigg)^{\frac{1}{2}} 
											dz \\
											&\leq \frac{C}{\varphi (Q)^{\eta /\delta }|Q|} \int_{Q} \int_{\left(\mathbb{R}^{n}\right)^{2} 
												\backslash(\widetilde{Q^*})^{2}}\left(\int_{0}^{\infty} \left |  
											K_{t}(z, \vec{y})\right|^{2} \frac{d 	t}{t}\right)^{\frac{1}{2} }
											\prod_{i=1}^{2} |b_i(y_i)-\widetilde{\lambda}_i||f_{i}\left(y_{i}\right) |d y_1dy_2 d z\\
											%&\leq \frac{1}{|Q|} \int_{Q} \int_{\left(\mathbb{R}^{n}\right)^{m} 	
											%\backslash\left(\Delta_2\right)^{m}}\frac{C}{\left ( \sum_{j=1}^{m} \left | 
											%z- y_{j}  \right |  \right ) ^{mn}\cdot \left ( 1+  \sum_{j=1}^{m} \left | 
											%z- y_{j}  \right |  \right ) ^{N} } \\
										%&\ \ \ \ \times \prod_{j=1}^{m} \left|f_{j}\left(y_{j}\right)\right| d 
										%\vec{y} d z \\
%										&\leq \frac{C}{\varphi (Q)^{\eta /\delta }|Q|} \int_{Q}  
%										\int_{\left(\mathbb{R}^{n}\right)^{2} 
%											\backslash(\widetilde{Q^*})^{2}}\frac{\prod_{i=1}^{2} (b_i(y_i)-\widetilde{\lambda}_i)f_{i}\left(y_{i}\right)}{\left ( 
%											\sum_{j=1}^{2} \left | z- y_{j}  \right |  \right ) ^{2n}\left ( 1+  
%											\sum_{j=1}^{2} \left | z- y_{j}  \right |  \right ) ^{N} } d y_1dy_2 d z \\
										&\leq \frac{C}{\varphi (Q)^{\eta /\delta }|Q|} \int_{Q} \sum_{k=1}^{\infty } 
										\int_{(2^{k+3}Q)^2\setminus (2^{k+2}Q)^2}\frac{\prod_{i=1}^{2} |b_i(y_i)-\widetilde{\lambda}_i||f_{i}\left(y_{i}\right)|}{\left ( 
											\sum_{j=1}^{2} \left | z- y_{j}  \right |  \right ) ^{2n}\left ( 1+  
											\sum_{j=1}^{2} \left | z- y_{j}  \right |  \right ) ^{N} } d y_1dy_2 d z \\
										&\leq  \frac C{\varphi (Q)^{\eta /\delta }} \sum_{k=1}^\infty 
										\frac{|2^{k+3}Q|^2}{(2^kr)^{2n}(1+2^kr)^{N}}\prod_{j=1}^2\left(\frac{1}
										{|2^{k+3}Q|}\int_{2^{k+3}Q}|b_j(y_j)-\widetilde{\lambda}_j||f_j(y_j)|dy_j\right)\\
										&\leq  C \sum_{k=1}^\infty 
										\frac{|2^{k+3}Q|^{2}\varphi(2^{k+3}Q)^{2\eta/l}}{(2^kr)^{2n}
											(1+2^kr)^N}
										\prod_{j=1}^2\left(\frac{1}{\varphi(2^{k+3}Q)^{\eta}|2^{k+3}Q|}
										\int_{2^{k+3}Q}|f_j(y_j)|^{l}dy_j\right)^{\frac 1{l}} \\
										& \ \ \ \ \times \prod_{i=1}^2\left(\frac{1}{|2^{k+3}Q|}
										\int_{2^{k+3}Q}|b_i(y_i)-\widetilde{\lambda}_i|^{l'}dy_i\right)^{\frac 1{l'}}\\
										&\leq  C \sum_{k=1}^\infty 
										\frac{|2^{k+3}Q|^{2}\varphi(2^{k+3}Q)^{2\eta/l}}{(2^kr)^{2n}
											(1+2^kr)^N} k^2\| 
										b_1 \|_{BMO_\theta (\varphi)}\| b_2 \|_{BMO_\theta (\varphi)}  \mathcal{M}_{l,\varphi,\eta}(\vec 
										f)(x)\varphi (2^{k+3}Q)^{2\theta }
										\\
										&\leq C\sum_{k=1}^{\infty } k^{2} (2^k)^{\frac{2\eta }{l}+2\theta -N } \| 
										b_1\|_{BMO_\theta (\varphi)}\| 
										b_2 \|_{BMO_\theta (\varphi)}  \mathcal{M}_{l,\varphi,\eta}(\vec 
										f)(x)
										\\
										&\leq C \| 
										b_1\|_{BMO_\theta (\varphi)}\| 
										b_2 \|_{BMO_\theta (\varphi)}  \mathcal{M}_{l,\varphi,\eta}(\vec 
										f)(x).
										\end {align*}
										Combining the above estimates, we obtain
										\begin{align*}
											&M_{\delta,\varphi,\eta}^{\sharp,\triangle}(T_{\prod \vec b}(\vec 
											f))(x)\\
											&\leq C\prod_{j=1}^2 \| b_j \|_{BMO_{\theta_j}(\varphi)}
											\left(M_{\varepsilon,\varphi,\eta}^\triangle (T(\vec 
											f))(x)+\mathcal{M}_{l,\varphi,\eta}(\vec f)(x)\right)\\
											&\ \ \ \ +C\| b_1\|_{BMO_\theta (\varphi)}M_{\varepsilon,\varphi,\eta}^\triangle (T_{  b_{2}}^{2}(
											f_1,f_2))(x)+C\| b_2\|_{BMO_\theta (\varphi)}M_{\varepsilon,\varphi,\eta}^\triangle (T_{  b_{1}}^{1}(
											f_1,f_2))(x) .
											\end {align*}
											The proof of Lemma \ref{Lemma3.12} is finished. 
												\end{proof}

\section{ Proof of Main Results }\label{sec4}

\quad\quad In this section, the boundedness of multilinear commutators and 
multilinear iterative commutators generated by new multilinear square operators with  
generalized kernels and $BMO_{\vec\theta}^m(\varphi)$ functions on   
weighted Lebesgue spaces and weighted Morrey spaces is obtained, 
respectively.

\subsection{Proof of Theorem 2.1}
 \begin{proof}
By Lemma \ref{Lemma3.4} for $\vec \omega \in A_{\vec p/q'}^\infty(\varphi)$, 
there exists a $k>1$, such that $\vec \omega \in A_{\vec 
p/(q'k)}^\infty(\varphi)$ and $p_j/(q'k)> 1$, $j=1,\dots ,m$. 
Let $ l=q'k$, then $\vec \omega \in A_{\vec p/l}^\infty(\varphi)$, $l>q'$ and $ 
p_{j}/l>1 $, $ j=1,\dots ,m $.  
Take $\eta=\eta_0$ in Lemma \ref{Lemma3.5} for $\vec \omega \in A_{\vec 
p/l}^\infty(\varphi)$. It derives from Lemma \ref{Lemma3.3} that $v_{\vec 
\omega} \in A_{mp/l}^\infty (\varphi) \subset A_{p/\delta }^\infty (\varphi)$, 
where $\varepsilon$ and $\delta$ are taken by
\begin{align*}
0<\varepsilon<\frac1m, 
0<\delta<\frac 1{\Big(\frac 1 \varepsilon +\frac{\mathop{\max}\limits_{  1\le j\le m}\theta _j}{\eta_0}\Big)}.
\end{align*}
Using Lemma \ref{Lemma3.6}, Lemma \ref{Lemma3.11}, and Lemma \ref{Lemma3.10},  we deduce that
\begin{align*}
\| T_{\sum  \vec b}(\vec f)\|_{L^{p}(v_{\vec\omega})}
&=\left \| |T_{\sum  \vec b}(\vec f)|^\delta \right \| _{L^{p/ 
		\delta}(v_{\vec\omega})}^{1/\delta }
\leq \| M_{\varphi,\eta_0}^{\triangle}(|T_{\sum  \vec b}(\vec 
f)|^\delta) \|_{L^{p/ \delta}(v_{\vec\omega})}^{1/\delta }\\
&\leq C \| M_{\delta,\varphi,\eta_0}^{\sharp,\triangle}(T_{\sum  \vec b}(\vec 
f)) \|_{L^{p}(v_{\vec\omega})}\\
%& \leq C\sum _{j=1}^m \| b_j \|_{BMO_{\theta_j}(\varphi)}
%\| M_{\varepsilon,\varphi,\eta_0}^\triangle (T(\vec 
%f))+\mathcal{M}_{t,\varphi,\eta_0}(\vec f)\|_{L^{p}(v_{\vec\omega})}\\
&\leq  C\sum _{j=1}^m \| b_j \|_{BMO_{\theta_j}(\varphi)}\left ( \| 
M_{\varepsilon,\varphi,\eta_0}^\triangle (T(\vec 
f))\|_{L^{p}(v_{\vec\omega})}+ \|\mathcal{M}_{l,\varphi,\eta_0}(\vec 
f)\|_{L^{p}(v_{\vec\omega})} \right ) \\
&\leq C\sum _{j=1}^m \| b_j \|_{BMO_{\theta_j}(\varphi)}
\|\mathcal{M}_{l,\varphi,\eta_0}(\vec f)\|_{L^{p}(v_{\vec\omega})},
\end{align*}
where
\begin{align*}
\| M_{\varepsilon,\varphi,\eta_0}^\triangle (T(\vec 
f))\|_{L^{p}(v_{\vec\omega})}
&\leq  C \| M_{\varepsilon, \varphi,\eta_0}^{\sharp, \triangle} (T(\vec f)) 
\|_{L^{p}(v_{\vec\omega})}
\leq C \| \mathcal{M}_{q',\varphi,\eta_0}(\vec f) 
\|_{L^{p}(v_{\vec\omega})}\\
&\leq C \| \mathcal{M}_{l,\varphi,\eta_0}(\vec f) \|_{L^{p}(v_{\vec\omega})}.
\end{align*}
Then by Lemma \ref{Lemma3.5},
\begin{align*}
\| T_{\sum  \vec b}(\vec f)\|_{L^{p}(v_{\vec\omega})}
&\leq C \sum _{j=1}^m \| b_j \|_{BMO_{\theta_j}(\varphi)}
\|\mathcal{M}_{\varphi,\eta_0}(|\vec f|^{l})\|^{1/l}_{L^{p/l}(v_{\vec\omega})}\\
&\leq C \sum _{j=1}^m \| b_j \|_{BMO_{\theta_j}(\varphi)}
\prod_{j=1}^m\| |f_j|^{l} \|^{1/l}_{L^{p_j/l}(\omega_j)}\\
&= C \sum _{j=1}^m \| b_j \|_{BMO_{\theta_j}(\varphi)}
\prod_{j=1}^m\| f_j \|_{L^{p_j}(\omega_j)}.
\end{align*}

The proof of Theorem \ref{Theorem2.1} is finished.
\end{proof}

\subsection{Proof of Theorem 2.2}
 \begin{proof}
Take $\eta _0$ and $l$ the same as in the proof of Theorem \ref{Theorem2.1}. Pick $\delta ,\beta _1,\dots ,\beta _m$ such that $0< \beta _m< \frac{1}{m} $, $0< \beta _j< \frac{1}{\frac{\sum_{j=1}^{m} \theta _j}{\eta _0} +\frac{1}{\beta _{j+1}} } $, $j=1,\dots ,m-1$, and $0< \delta < \frac{1}{\frac{\sum_{j=1}^{m} \theta _j}{\eta _0} +\frac{1}{\beta _{1}} }$, then $\eta _0> \frac{\sum_{j=1}^{m} \theta _j}{\frac{1}{\delta } -\frac{1}{\beta _1}  } $ and $\eta _0> \frac{\sum_{j=1}^{m} \theta _j}{\frac{1}{\beta _j } -\frac{1}{\beta _{j+1}}  } $, $j=1,\dots ,m-1$. Applying Lemma \ref{Lemma3.12}, we have
\begin{align*}
&||M_{\delta,\varphi,\eta_0}^{\sharp,\triangle}(T_{\prod \vec b}(\vec 
f))||_{L^{p}(v_{\vec\omega})}\\
&\leq C\prod_{j=1}^m \| b_j \|_{BMO_{\theta_j}(\varphi)}
\left(||M_{\beta _1,\varphi,\eta_0}^\triangle (T(\vec 
f))||_{L^{p}(v_{\vec\omega})}+||\mathcal{M}_{l,\varphi,\eta_0}(\vec f)||_{L^{p}(v_{\vec\omega})}\right)\\
&\ \ \ \ +C\sum_{j=1}^{m-1} \sum_{\xi \in C^{m}_j}\prod_{i=1}^{j} ||b_{\xi (i)}||_{BMO_{\theta _{\xi (i)}}(\varphi )}
||M_{\beta _1,\varphi,\eta_0}^\triangle (T_{\prod \vec b_{\xi '}}(\vec 
f))||_{L^{p}(v_{\vec\omega})}.
\end {align*}
According to Definition \ref{definition1.6}, $\xi =\left \{ \xi (1),\dots , \xi (j)\right \} $, $\xi' =\left \{ \xi (j+1),\dots , \xi (m)\right \} $ and $A_h=\left \{ \xi _1: \xi _1\subset  \xi '\right \} $, where $\xi _1$ is any finite subset of $\xi '$ composed of different elements. Denote $\xi' _1=\xi '-\xi _1$.

Repeating applying Lemma \ref{Lemma3.12} enables us to write
\begin{align*}
	&||M_{\beta _1,\varphi,\eta_0}^{\sharp,\triangle}(T_{\prod \vec b_{\xi '}}(\vec 
	f))||_{L^{p}(v_{\vec\omega})}\\
	&\leq C\prod_{k=j+1}^m \| b_{\xi (k)} \|_{BMO_{\theta_{\xi (k)}}(\varphi)}
	\left(||M_{\beta _2,\varphi,\eta_0}^\triangle (T(\vec 
	f))||_{L^{p}(v_{\vec\omega})}+||\mathcal{M}_{l,\varphi,\eta_0}(\vec f)||_{L^{p}(v_{\vec\omega})}\right)\\
	&\ \ \ \ +C\sum_{h=1}^{m-j-1} \sum_{\xi_1 \in A_h}\prod_{i=1}^{h} ||b_{\xi_1 (i)}||_{BMO_{\theta_{\xi_1 (i)}} (\varphi )}
	||M_{\beta _2,\varphi ,\eta _0}^{\Delta } (T_{\prod \vec b_{\xi '_1}}(\vec 
	f))||_{L^{p}(v_{\vec\omega})}.
\end {align*}
By Lemma \ref{Lemma3.6} and Lemma \ref{Lemma3.11}, we continue in this way to obtain 
\begin{align*}
	&||M_{\delta,\varphi,\eta_0}^{\sharp,\triangle}(T_{\prod \vec b}(\vec 
	f))||_{L^{p}(v_{\vec\omega})}\\
	&\leq C\prod_{j=1}^m \| b_j \|_{BMO_{\theta_j}(\varphi)}
	\Bigg(C_{m+1}(m,n)||\mathcal{M}_{l,\varphi,\eta_0}(\vec f)||_{L^{p}(v_{\vec\omega})}
	+C_{1}(m,n)||{M}_{\beta _1,\varphi ,\eta _0}^{\Delta }(T(\vec f))||_{L^{p}(v_{\vec\omega})}\\
	 &\ \ \ \ + C_{2}(m,n)||{M}_{\beta _2,\varphi ,\eta _0}^{\Delta }(T(\vec f))||_{L^{p}(v_{\vec\omega})}+
	 \cdots +C_{m}(m,n)||{M}_{\beta _m,\varphi ,\eta _0}^{\Delta }(T(\vec f))||_{L^{p}(v_{\vec\omega})}\Bigg),
\end {align*}
where $C_{1}(m,n),C_{2}(m,n),\dots ,C_{m}(m,n),C_{m+1}(m,n)$ are all finite real numbers
associated with $m$ and $n$.

Using Lemma \ref{Lemma3.6}, Lemma \ref{Lemma3.10} and Lemma \ref{Lemma3.5}, we have
\begin{align*}
	&\| T_{\prod \vec b}(\vec f)\|_{L^{p}(v_{\vec\omega})}\le ||M_{\delta,\varphi,\eta_0}^{\triangle}(T_{\prod \vec b}(\vec f))||_{L^{p}(v_{\vec\omega})}\le C
	||M_{\delta,\varphi,\eta_0}^{\sharp,\triangle}(T_{\prod \vec b}(\vec 
	f))||_{L^{p}(v_{\vec\omega})}\\
	&\leq C\prod_{j=1}^m \| b_j \|_{BMO_{\theta_j}(\varphi)}
	\Bigg(C_{m+1}(m,n)||\mathcal{M}_{l,\varphi,\eta_0}(\vec f)||_{L^{p}(v_{\vec\omega})}
	+C_{1}(m,n)||\mathcal{M}_{\beta _1,\varphi,\eta_0}^{\Delta }(T(\vec f))||_{L^{p}(v_{\vec\omega})}\\
	&\ \ \ \ + C_{2}(m,n)||\mathcal{M}_{\beta _2,\varphi ,\eta_0 }^{\Delta }(T(\vec f))||_{L^{p}(v_{\vec\omega})}+
	\cdots +C_{m}(m,n)||\mathcal{M}_{\beta _m,\varphi ,\eta_0 }^{\Delta }(T(\vec f))||_{L^{p}(v_{\vec\omega})}\Bigg)\\
	&\leq C\prod_{j=1}^m \| b_j \|_{BMO_{\theta_j}(\varphi)}
	\Bigg(C_{m+1}(m,n)||\mathcal{M}_{l,\varphi,\eta_0}(\vec f)||_{L^{p}(v_{\vec\omega})}
	+C_{m+2}(m,n)||M_{\beta _m,\varphi,\eta_0}^{\triangle}(T(\vec f))||_{L^{p}(v_{\vec\omega})}\Bigg)\\
	&\leq C\prod_{j=1}^m \| b_j \|_{BMO_{\theta_j}(\varphi)}
	\Bigg(C_{m+1}(m,n)||\mathcal{M}_{l,\varphi,\eta_0}(\vec f)||_{L^{p}(v_{\vec\omega})}
	+C_{m+2}(m,n)||M_{\beta _m,\varphi,\eta_0}^{\sharp,\triangle}(T(\vec f))||_{L^{p}(v_{\vec\omega})}\Bigg)\\
	&\leq C\prod_{j=1}^m \| b_j \|_{BMO_{\theta_j}(\varphi)}
	\Bigg(C_{m+1}(m,n)||\mathcal{M}_{l,\varphi,\eta_0}(\vec f)||_{L^{p}(v_{\vec\omega})}
	+C_{m+2}(m,n)||\mathcal{M}_{q',\varphi,\eta_0}(\vec f)||_{L^{p}(v_{\vec\omega})}\Bigg)\\
	&\leq C\prod_{j=1}^m \| b_j \|_{BMO_{\theta_j}(\varphi)}
	||\mathcal{M}_{l,\varphi,\eta_0}(\vec f)||_{L^{p}(v_{\vec\omega})}
	= C\prod_{j=1}^m \| b_j \|_{BMO_{\theta_j}(\varphi)}
	||\mathcal{M}_{\varphi,\eta_0}(|\vec f|^{l})||_{L^{p/l}(v_{\vec\omega})}^{1/l}
	\\
	&\leq C\prod_{j=1}^m \| b_j \|_{BMO_{\theta_j}(\varphi)}
	\prod_{j=1}^m \| f_j \|_{L^{p_j}(\omega_j)},
\end {align*}

The proof of Theorem \ref{Theorem2.2} is finished.
\end{proof}

\subsection{Proof of Theorem 2.3}
\begin{proof}We only prove the case  $ T_{b_j}^{j}  $ and $ b_j\in 
BMO_{\tilde{\theta}} (\varphi )  $, 
$ \theta _1=\cdots =\theta _m=\tilde{\theta} $ for simplicity. In the general 
case,  the calculation is not different, but only more complicated. Let $ 
Q:=Q(z,r) $  and $ Q^\ast :=3Q $. We split each $f_j$ as 
$f_j=f_j^0+f_j^\infty$, for each $f_j^0=f_j \chi_{Q^*}$, $ f_j^\infty=f_j-f_j^0 
$. We obtain for $x\in Q(z,r)$,
\begin{align*}
|T_{b_j}^{j} (\vec f)(x)|
&\leq |T_{b_j}^{j} (\vec f\chi _{Q^\ast })(x)|+\sum_{(\alpha_1,\ldots,\alpha_m)\in 
	\mathscr{L}}  | 
T(f_{1}^{\alpha_{1}},\ldots,(b_j(x)-b_j)f_{j}^{\alpha_{j}},\ldots 
,f_{m}^{\alpha_{m}})(x)|\\
&\leq|T_{b_j}^{j} (\vec f\chi _{Q^\ast })(x)|+\sum_{(\alpha_1,\ldots,\alpha_m)\in 
	\mathscr{L}}  \Big (| 
T(f_{1}^{\alpha_{1}},\ldots,(b_j(x)-b_{j,Q})f_{j}^{\alpha_{j}},\ldots 
,f_{m}^{\alpha_{m}})(x)|\\
&\quad + | 
T(f_{1}^{\alpha_{1}},\ldots,(b_{j,Q}-b_j)f_{j}^{\alpha_{j}},\ldots 
,f_{m}^{\alpha_{m}})(x)|\Big )\\
&:=\mathrm{I}'+\mathrm{II}',
\end{align*}
where  $\mathscr{L}:=\{(\alpha_1,\ldots,\alpha_m)$: there is at least one 
$\alpha_j=\infty\}, j=1,\ldots,m$.

For the term $ \mathrm{II}' $.  Let us consider $\alpha_j=\infty $ for $ j=1,\cdots ,\xi 
$ and $ \alpha_j=0$ for $ j=\xi+1,\cdots ,m $. There are two cases.
\begin{align*}
|T(f_{1}^{\infty },\ldots ,(b_j(x)-b_{j,Q})f_{j}^{\infty } ,\ldots ,f_{\xi }^{\infty 
},f_{\xi +1}^{0},\ldots ,f_{m}^{0})(x)|\\
+|T(f_{1}^{\infty },\ldots 
,(b_{j,Q}-b_j)f_{j}^{\infty } ,\ldots,f_{\xi }^{\infty },f_{\xi+1}^{0},\ldots 
,f_{m}^{0})(x)|
\end{align*}
or
\begin{align*}
|T(f_{1}^{\infty },\ldots,f_{\xi }^{\infty } ,f_{\xi +1}^{0},\ldots 
,(b_j(x)-b_{j,Q})f_{j}^{0},\ldots,f_{m}^{0})(x)|\\
+|T(f_{1}^{\infty },\ldots,f_{\xi }^{\infty } 
,f_{\xi +1}^{0},\ldots ,(b_{j,Q}-b_j)f_{j}^{0},\ldots ,f_{m}^{0})(x)|.
\end{align*}

It is sufficient to prove the following, and the other case is exactly similar.
\begin{align*}
&|T(f_{1}^{\infty },\ldots ,(b_j(x)-b_{j,Q})f_{j}^{\infty } ,\ldots ,f_{\xi }^{\infty 
},f_{\xi +1}^{0},\ldots ,f_{m}^{0})(x)|\\
&\quad  +|T(f_{1}^{\infty },\ldots 
,(b_{j,Q}-b_j)f_{j}^{\infty } ,\ldots,f_{\xi }^{\infty },f_{\xi+1}^{0},\ldots 
,f_{m}^{0})(x)|\\
&\le |(b_j(x)-b_{j,Q})T(f_{1}^{\infty },\ldots,f_{j}^{\infty } ,\ldots 
,f_{\xi }^{\infty },f_{\xi +1}^{0},\ldots ,f_{m}^{0})(x)|\\
& \quad +|T(f_{1}^{\infty },\ldots,(b_j-b_{j,Q})f_{j}^{\infty } ,\ldots 
,f_{\xi }^{\infty },f_{\xi +1}^{0},\ldots,f_{m}^{0})(x)|.
%&\le C|b_j(x)-b_Q||T(f_{1}^{\infty },\ldots ,f_{j}^{\infty } ,\ldots 
%,f_{\xi }^{\infty },f_{\xi +1}^{0},\ldots ,f_{m}^{0})|\\
%&\ \ \ \ + |T(f_{1}^{\infty },\ldots 
%,(b_j-b_Q+b_{3^{k+1}Q}-b_{3^{k+1}Q})f_{j}^{\infty } ,\ldots ,f_{\xi }^{\infty 
%},f_{\xi +1}^{0},\ldots ,f_{m}^{0})(x)|.
\end{align*}

Using Minkowski's inequality, the size condition \eqref{1.4}, Hölder’s 
inequality and $ \vec w\in A_{\vec p/q'}^{\theta  }(\varphi )$, we conclude that
\begin{align}\label{4.4} 
&|T(f_{1}^{\infty },\ldots ,f_{j}^{\infty } ,\ldots ,f_{\xi }^{\infty 
},f_{\xi +1}^{0},\ldots ,f_{m}^{0})(x)| \nonumber\\
&\le \left(\int_{0}^{\infty} \bigg( \int_{(\mathbb{R}^{n})^{m}\backslash 
	(Q^\ast )^{m}} \left|K_{t}(x, \vec{y})\right| \cdot \prod_{j=1}^{m} 
\left|f_{j}\left(y_{j}\right)\right| d \vec{y} \bigg)^{2} 
\frac{d t}{t}\right)^{\frac{1}{2}} \nonumber\\
&\leq \int_{(\mathbb{R}^{n})^{m}\backslash (Q^\ast)^{m}}\left(\int_{0}^{\infty} 
\left |  K_{t}(x, \vec{y})\right|^{2} \frac{d 	t}{t}\right)^{\frac{1}{2} 
}\cdot \prod_{j=1}^{m} \left|f_{j}\left(y_{j}\right)\right| d \vec{y} 
\nonumber\\
&\leq C \int_{(\mathbb{R}^{n})^{m}\backslash 
(Q^\ast)^{m}}\frac{\left|f_{1}\left(y_{1}\right) \cdots 
f_{m}\left(y_{m}\right)\right|}{\left(\sum_{i=1}^{m}\left|x-y_{i}\right|
\right)^{mn}\left(1+\sum_{i=1}^{m}\left|x-y_{i}\right|\right)^{N}} d \vec{y} 
\nonumber\\
%&=C \sum_{k=1}^{\infty } \int_{\left(3^{k+1}Q\right)^{m} \backslash\left(3^k Q
%\right)^{m}}\frac{\left|f_{1}\left(y_{1}\right) \cdots 
%f_{m}\left(y_{m}\right)\right|}{\left(\sum_{i=1}^{m}\left|x-y_{i}\right|\right)^{m
%n}\left(1+\sum_{i=1}^{m}\left|x-y_{i}\right|\right)^{N}} d \vec{y}\\
& \leq C \sum_{k=1}^{\infty} \frac{1}{\left | 3^{k+1}Q \right |^{m}  
\varphi\left(3^{k+1} Q\right)^{N}} \int_{\left(3^{k+1} Q\right)^{m}} 
\prod_{i=1}^{m}\left|f_{i}\left(y_{i}\right)\right| d \vec{y} \nonumber\\
& \leq C\sum_{k=1}^{\infty} \varphi\left(3^{k+1} Q\right)^{-N}
\prod_{i=1}^{m} \left ( \frac{1}{\left|3^{k+1} Q\right|} \int_{3^{k+1} 
Q}\left|f_{i}\left(y_{i}\right)\right|^{q'} d y_{i} \right )^{\frac{1}{q'} } 
\nonumber\\
%& = C\sum_{k=1}^{\infty} \varphi\left(3^{k+1} Q\right)^{-N}\left|3^{k+1} 
%Q\right|^{-\frac{m}{q'} }\prod_{i=1}^{m} \left ( \int_{3^{k+1} 
%Q}\left|f_{i}\left(y_{i}\right)\right|^{q'} d y_{i} \right )^{\frac{1}{q'} } \\
& \leq C\sum_{k=1}^{\infty} \varphi\left(3^{k+1} Q\right)^{-N}\left|3^{k+1} 
Q\right|^{-\frac{m}{q'} }
\prod_{i=1}^{m} \left ( \int_{3^{k+1} 
Q}\left|f_{i}\left(y_{i}\right)\right|^{p_i}w_i(y_i) d y_{i} \right 
)^{\frac{1}{p_i} } \nonumber\\
&  \quad \times\left ( \int_{3^{k+1} Q}w_i(y_i)^{1-\frac{p_i}{p_i-q'} } d 
y_{i} \right )^{\frac{p_i-q'}{p_i} \cdot \frac{1}{q'} } \nonumber\\
& \leq C\sum_{k=1}^{\infty} \frac{\varphi\left(3^{k+1} Q\right)^{\frac{m\theta 
}{q'} -N}}{{\left(\int_{3^{k+1} Q} v_{\vec{w}}\right)^{1 / p}}} 
\prod_{i=1}^{m}\left\|f_{i} \chi_{3^{k+1} Q}\right\|_{L^{p_i} (w_i)}.
\end{align}

Since $b_j(y_j)-b_{j,Q}=(b_j(y_j)-b_{j,3^{k+1}Q})+(b_{j,3^{k+1}Q}-b_{j,Q})$, the following estimate can be obtained similar to \eqref{4.4}.
\begin{align*}
&|T(f_{1}^{\infty },\ldots ,(b_j-b_{j,Q})f_{j}^{\infty } 
,\ldots ,f_{\xi }^{\infty},f_{\xi +1}^{0},\ldots ,f_{m}^{0})(x)|\\
%&\le |T(f_{1}^{\infty },\ldots ,(b_j-b_{j,3^{k+1}Q})f_{j}^{\infty } ,\ldots 
%,f_{\xi }^{\infty },f_{\xi +1}^{0},\ldots ,f_{m}^{0})(x)|\\
%& \quad + |T(f_{1}^{\infty },\ldots ,(b_{j,3^{k+1}Q}-b_{j,Q})f_{j}^{\infty } 
%,\ldots,f_{\xi}^{\infty },f_{\xi +1}^{0},\ldots ,f_{m}^{0})(x)|\\
&\le C\sum_{k=1}^{\infty} \frac{\varphi\left(3^{k+1} 
Q\right)^{\frac{m\theta }{q'} -N}\cdot 
|b_{j,3^{k+1}Q}-b_{j,Q}|}{{\left(\int_{3^{k+1} Q} v_{\vec{w}}\right)^{1 / p}}} 
\prod_{i=1}^{m}\left\|f_{i} \chi_{3^{k+1}Q}\right\|_{L^{p_i} (w_i)}\\
& \quad +C \sum_{k=1}^{\infty} \frac{1}{\left | 3^{k+1}Q \right |^{m}  
	\varphi\left(3^{k+1} Q\right)^{N}} \int_{\left(3^{k+1} Q\right)^{m}} 
\Big(\prod_{i=1,i\neq 
j}^{m}|f_{i}(y_{i})|\Big)\Big|f_{j}(y_{j})(b_j(y_j)-b_{j,3^{k+1}Q})\Big| d 
\vec{y}.
\end{align*}

Using H\"{o}lder's inequality, we deduce that
\begin{align*}
& \sum_{k=1}^{\infty} \left | 3^{k+1}Q \right |^{-m}  
\varphi\left(3^{k+1} Q\right)^{-N} \int_{\left(3^{k+1} Q\right)^{m}} 
\Big(\prod_{i=1,i\neq 
j}^{m}|f_{i}(y_{i})|\Big)\Big|f_{j}(y_{j})(b_j(y_j)-b_{j,3^{k+1}Q})\Big| d 
\vec{y} 
\\
&\le C \sum_{k=1}^{\infty} \left | 3^{k+1}Q \right |^{-m/q'}  
\varphi\left(3^{k+1} Q\right)^{-N/q'} \left ( \prod_{i=1,i\neq 
j}^{m}\int_{3^{k+1} Q}|f_{i}(y_{i})|^{q'}dy_i \right ) ^{\frac{1}{q'} } \\
& \quad \times \left ( \int_{3^{k+1} Q} 
|f_{j}(y_{j})(b_j(y_j)-b_{j,3^{k+1}Q})|^{q'} d {y_j} \right ) ^{\frac{1}{q'} } 
\\
&\le C \sum_{k=1}^{\infty} \left | 3^{k+1}Q \right |^{-m/q'}  
\varphi\left(3^{k+1} Q\right)^{-N/q'}\prod_{i=1,i\neq 
j}^{m}\Bigg[\left ( \int_{3^{k+1} Q} |f_{i}(y_{i})|^{p_i}w_i(y_i)dy_i \right ) 
^{\frac{1}{p_i} }\\
& \quad \times \left ( \int_{3^{k+1} Q}w_i(y_i)^{\frac{-q'}{p_i-q'} }dy_i 
\right ) ^{\frac{p_i-q'}{p_iq'} }\Bigg]\left ( \int_{3^{k+1} Q} 
|f_{j}(y_{j})|^{p_j}w_j(y_j)dy_j \right )^{\frac{1}{p_j} } \\
& \quad \times \left ( \int_{3^{k+1} Q} 
|b_j(y_j)-b_{j,3^{k+1}Q}|^{\frac{p_jq'}{p_j-q'} } w_j(y_j)^{\frac{-q'}{p_j-q'} 
}d {y_j} \right ) ^{\frac{p_j-q'}{p_jq'} }.
\end{align*}
Since $ \vec w\in A_{\vec p/q'}^{\theta  }(\varphi )$, we have
\begin{align*}
&\prod_{i=1,i\neq j}^{m} \left ( \int_{3^{k+1} Q}w_i(y_i)^{\frac{-q'}{p_i-q'} 
}dy_i \right ) ^{\frac{p_i-q'}{p_iq'} }\\
& \quad =\prod_{i=1}^{m} \left ( \int_{3^{k+1} 
Q}w_i(y_i)^{\frac{-q'}{p_i-q'} }dy_i 
\right ) ^{\frac{p_i-q'}{p_iq'} }\left ( \int_{3^{k+1} 
Q}w_j(y_j)^{\frac{-q'}{p_j-q'} }dy_j 
\right ) ^{-\frac{p_j-q'}{p_jq'} }\\
& \quad \le C\varphi\left(3^{k+1} Q\right)^{\frac{m\theta 
	}{q'} }\left | 3^{k+1}Q \right |^{\frac{m}{q'} }{\left(\int_{3^{k+1} Q} 
	v_{\vec w }\right)^{-\frac{1}{p} }}\left ( \int_{3^{k+1} 
	Q}w_j(y_j)^{\frac{-q'}{p_j-q'} }dy_j 
\right ) ^{-\frac{p_j-q'}{p_jq'} }.
\end{align*}
By Lemma \ref{Lemma3.9},  it follows that
\begin{align*}
&\left ( \int_{3^{k+1} Q} 
|b_j(y_j)-b_{j,3^{k+1}Q}|^{\frac{p_jq'}{p_j-q'} } w_j(y_j)^{\frac{-q'}{p_j-q'} 
}d {y_j} \right ) ^{\frac{p_j-q'}{p_jq'} }\\
& =\left ( \int_{3^{k+1} 
	Q}w_j(y_j)^{\frac{-q'}{p_j-q'} }dy_j 
\right ) ^{\frac{p_j-q'}{p_jq'} }\Bigg (\frac{1}{\Big ( \int_{3^{k+1} 
		Q}w_j(y_j)^{\frac{-q'}{p_j-q'} }dy_j 
	\Big )}  \\
& \quad \times \int_{3^{k+1} Q} 
|b_j(y_j)-b_{j,3^{k+1}Q}|^{\frac{p_jq'}{p_j-q'} } w_j(y_j)^{\frac{-q'}{p_j-q'} 
}d {y_j} \Bigg ) ^{\frac{p_j-q'}{p_jq'} } \\
& \le C \left ( \int_{3^{k+1} 
	Q}w_j(y_j)^{\frac{-q'}{p_j-q'} }dy_j 
\right ) ^{\frac{p_j-q'}{p_jq'} }\varphi \left ( 3^{k+1} 
Q \right ) ^{\frac{\tilde{s}(p_j-q')}{p_jq'} }||b_j||_{BMO_{\tilde{\theta}} 
	(\varphi ) }.
\end{align*}
Combining the above estimates, note that $ \varphi\left(3^{k+1} 
Q\right)^{-N}\le \varphi\left(3^{k+1} Q\right)^{-\frac{N}{q'} } $, then
\begin{align}\label{4.5} 
& \sum_{k=1}^{\infty} \left | 3^{k+1}Q \right |^{-m}  
\varphi\left(3^{k+1} Q\right)^{-N} \int_{\left(3^{k+1} Q\right)^{m}} 
\Big(\prod_{i=1,i\neq 
	j}^{m}|f_{i}(y_{i})|\Big)\Big|f_{j}(y_{j})(b_j(y_j)-b_{j,3^{k+1}Q})\Big| d 
\vec{y} 
\nonumber\\
&\le C ||b_j||_{BMO_{\tilde{\theta}} (\varphi ) }\sum_{k=1}^{\infty} 
\varphi\left(3^{k+1} 
Q\right)^{\frac{m\theta }{q'} -\frac{N}{q'} }\frac{\varphi \left ( 3^{k+1} 
	Q \right ) ^{\frac{\tilde{s}(p_j-q')}{p_jq'} }}{{\left(\int_{3^{k+1} Q} 
		v_{\vec{w}}\right)^{1 / p}}} \prod_{i=1}^{m}\left\|f_{i} 
\chi_{3^{k+1}Q}\right\|_{L^{p_i} (w_i)}.
\end{align}
Hence, we obtain
\begin{align*}
|T_{b_j}^{j} (\vec f)(x)|
&\le |T_{b_j}^{j} (\vec f\chi _{Q^\ast })(x)|+ 
C|b_j(x)-b_{j,Q}|\sum_{k=1}^{\infty} 
\frac{\varphi\left(3^{k+1} Q\right)^{\frac{m\theta }{q'} 
-N}}{{\left(\int_{3^{k+1} Q} v_{\vec{w}}\right)^{1 / p}}} 
\prod_{i=1}^{m}\left\|f_{i} \chi_{3^{k+1}Q}\right\|_{L^{p_i} (w_i)}\\
& \quad +C\sum_{k=1}^{\infty} \frac{\varphi\left(3^{k+1} 
Q\right)^{\frac{m\theta }{q'} -N}\cdot 
|b_{j,3^{k+1}Q}-b_{j,Q}|}{{\left(\int_{3^{k+1} 
Q} v_{\vec{w}}\right)^{1 / p}}} 
\prod_{i=1}^{m}\left\|f_{i} \chi_{3^{k+1}Q}\right\|_{L^{p_i} (w_i)}\\
& \quad +C ||b_j||_{BMO_{\tilde{\theta}} (\varphi )}\sum_{k=1}^{\infty} 
\varphi\left(3^{k+1} 
Q\right)^{\frac{m\theta }{q'} -\frac{N}{q'} }\frac{\varphi (3^{k+1} 
Q)^{\frac{\tilde{s}(p_j-q')}{p_jq'} }}{{\left(\int_{3^{k+1} Q} 
v_{\vec{w}}\right)^{1 / p}}} \prod_{i=1}^{m}\left\|f_{i} 
\chi_{3^{k+1}Q}\right\|_{L^{p_i} (w_i)}\\
&:=G_1+G_2+G_3+G_4.
\end{align*}
Taking $ L^{p}(v_{\vec w}) $ norm on the cube $ Q(z,r) $ of $ G_1 $, using 
Theorem \ref{Theorem2.1}, it follows that
\begin{align*}
||G_{1}\chi _{Q(z,r)} ||_{L^{p}(v_{\vec w}) }\le C||b_j||_{BMO_{\tilde{\theta}} 
(\varphi )}\prod_{i=1}^{m} ||f_i\chi 
_{3Q}||_{L^{p_i}({w_i})}.
\end{align*}

Let us estimate the term $ G_2 $, applying Lemma \ref{Lemma3.9}, we have
\begin{align*}
\Big( \int_{Q}  | b_j(x)-b_{j,Q} |^pv_{\vec w}(x) dx  \Big) ^{\frac{1}{p} } 
&=v_{\vec w}(Q)^{\frac{1}{p} } \Big( \frac{1}{v_{\vec w}(Q)}\int_{Q}  | 
b_j(x)-b_{j,Q} |^{p} v_{\vec w}(x) dx \Big ) ^{\frac{1}{p} } \\
&\le Cv_{\vec w}(Q)^{\frac{1}{p} }\varphi (Q)^{\frac{\tilde{s}}{p} 
}||b_j||_{BMO_{\tilde{\theta}}(\varphi )} \\
& \le C v_{\vec w}(Q)^{\frac{1}{p} }\varphi (3^{k+1} Q)^{\frac{\tilde{s}}{p} 
}|| b_j||_{BMO_{\tilde{\theta}}(\varphi ) }.
\end{align*}
Taking $ L^{p}(v_{\vec w}) $ norm on the cube 
$ Q(z,r) $ of $ G_2 $,  we get
\begin{align*}
&||G_{2}\chi _{Q(z,r)} ||_{L^{p}(v_{\vec w}) }\\
&\le C ||b_j||_{BMO_{\tilde{\theta}} (\varphi )}\sum_{k=1}^{\infty} 
\varphi\left(3^{k+1} Q\right)^{\frac{m\theta }{q'} 
	-\frac{N}{q'}+\frac{\tilde{s}}{p}}\frac{(\int_{Q} 
	v_{\vec{w}})^{1/p}}{{\left(\int_{3^{k+1} Q} 
		v_{\vec{w}}\right)^{1 /p}}} \prod_{i=1}^{m}\left\|f_{i} 
\chi_{3^{k+1}Q}\right\|_{L^{p_i} (w_i)}.
\end{align*}

For the term $ G_3 $, using Lemma \ref{Lemma3.1}, we obtain
\begin{align*}
\left | b_{j,3^{k+1}Q}-b_{j,Q} \right |
%&=\left |\frac{1}{|3^{k+1}Q|} \int_{3^{k+1}Q}  b_{j}(x)dx-\frac{1}{|3^{k+1}Q|} 
%\int_{3^{k+1}Q}  b_{j,Q} dx \right | \\
\le \frac{1}{|3^{k+1}Q|} \int_{3^{k+1}Q}  |b_{j}(x)-b_{j,Q}|dx 
\le Ck  \varphi (3^{k+1} Q)^{\tilde{\theta}}\| 
b_j\|_{BMO_{\tilde{\theta}}(\varphi ) }.
\end{align*}
Taking $ L^{p}(v_{\vec w}) $ norm on the cube 
$ Q(z,r) $ of $ G_3 $, we conclude that
\begin{align*}
&||G_{3}\chi _{Q(z,r)} ||_{L^{p}(v_{\vec w}) }\\
&\le C ||b_j||_{BMO_{\tilde{\theta}} (\varphi )}\sum_{k=1}^{\infty}k 
\varphi\left(3^{k+1} Q\right)^{\frac{m\theta }{q'} 
	-\frac{N}{q'}+\tilde{\theta} }\frac{(\int_{Q} 
	v_{\vec{w}})^{1/p}}{{\left(\int_{3^{k+1} Q} 
		v_{\vec{w}}\right)^{1 /p}}} \prod_{i=1}^{m}\left\|f_{i} 
\chi_{3^{k+1}Q}\right\|_{L^{p_i} (w_i)}.
\end{align*}
Taking $ L^{p}(v_{\vec w}) $ norm on the cube 
$ Q(z,r) $ of $ G_4 $,  we deduce that
\begin{align*}
&||G_{4}\chi _{Q(z,r)} ||_{L^{p}(v_{\vec w}) }\\
&\le C 
||b_j||_{BMO_{\tilde{\theta}} (\varphi ) }\sum_{k=1}^{\infty} 
\varphi\left(3^{k+1} Q\right)^{\frac{m\theta }{q'} 
	-\frac{N}{q'}+\frac{\tilde{s}(p_j-q')}{p_jq'}}\frac{(\int_{Q} 
	v_{\vec{w}})^{1/p}}{{\left(\int_{3^{k+1} Q} 
		v_{\vec{w}}\right)^{1 /p}}} \prod_{i=1}^{m}\left\|f_{i} 
\chi_{3^{k+1}Q}\right\|_{L^{p_i} (w_i)}.
\end{align*}
Combining the above estimates, we obtain
\begin{align*}
&||T_{b_j}^{j} (\vec f)\chi _{Q(z,r)}||_{L^{p}(v_{\vec w}) }\\
&\le C||b_j||_{BMO_{\tilde{\theta}} (\varphi )}\prod_{i=1}^{m} ||f_i\chi 
_{3Q}||_{L^{p_i}({w_i})}\\
& \quad +C ||b_j||_{BMO_{\tilde{\theta}} (\varphi )}\sum_{k=1}^{\infty} 
\varphi\left(3^{k+1} Q\right)^{\frac{m\theta }{q'} 
-\frac{N}{q'}+\frac{\tilde{s}}{p}}\frac{(\int_{Q} 
v_{\vec{w}})^{1/p}}{{\left(\int_{3^{k+1} Q} 
	v_{\vec{w}}\right)^{1 /p}}} \prod_{i=1}^{m}\left\|f_{i} 
\chi_{3^{k+1}Q}\right\|_{L^{p_i} (w_i)} \\
& \quad +C ||b_j||_{BMO_{\tilde{\theta}} (\varphi )}\sum_{k=1}^{\infty}k 
\varphi\left(3^{k+1} Q\right)^{\frac{m\theta }{q'} 
-\frac{N}{q'}+\tilde{\theta} }\frac{(\int_{Q} 
v_{\vec{w}})^{1/p}}{{\left(\int_{3^{k+1} Q} 
v_{\vec{w}}\right)^{1 /p}}} \prod_{i=1}^{m}\left\|f_{i} 
\chi_{3^{k+1}Q}\right\|_{L^{p_i} (w_i)} \\
& \quad + C ||b_j||_{BMO_{\tilde{\theta}} (\varphi ) }\sum_{k=1}^{\infty} 
\varphi\left(3^{k+1} Q\right)^{\frac{m\theta }{q'} 
-\frac{N}{q'}+\frac{\tilde{s}(p_j-q')}{p_jq'}}\frac{(\int_{Q} 
v_{\vec{w}})^{1/p}}{{\left(\int_{3^{k+1} Q} 
v_{\vec{w}}\right)^{1 /p}}} \prod_{i=1}^{m}\left\|f_{i} 
\chi_{3^{k+1}Q}\right\|_{L^{p_i} (w_i)}\\
&:=G_{1}^{'} +G_{2}^{'} +G_{3}^{'} +G_{4}^{'}.
\end{align*}
Using Lemma \ref{Lemma3.3} and Lemma \ref{Lemma3.7}, we get
\begin{align*}
\frac{(\int_{Q} v_{\vec{w}})^{\frac{1}{p} }}{{\left(\int_{3^{k+1} Q} 
v_{\vec{w}}\right)^{\frac{1}{p}}}} 
=\frac{\big(v_{\vec{w}}(Q)\big) ^{\frac{1}{p} 
}}{{\big(v_{\vec{w}}(3^{k+1}Q)\big)^{\frac{1}{p}}}} 
\le C\frac{1}{3^{nk\delta /p}} \varphi (3^{k+1} Q)^{\frac{\eta }{p} }.
\end{align*}
Using Lemma \ref{Lemma3.3} and Lemma \ref{Lemma3.8}, then $ v_{\vec w}\in 
A_{mp/q'}^{\theta } (\varphi ) $, and it follows that
\begin{align*}
v_{\vec w}(Q)^{-\lambda }\le C\varphi (3^{k+1}Q)^{\frac{mp\theta \lambda }{q'}  
}v_{\vec w}(3^{k+1}Q)^{-\lambda }.
\end{align*}
Multiplying both sides of $ G_{1}^{'} $ by  $ \varphi(Q)^{\alpha} 
v_{\vec{w}}(Q)^{-\lambda} $, using Lemma \ref{Lemma3.3} and Lemma 
\ref{Lemma3.8}, then
\begin{align*}
G_{1}^{'}\varphi(Q)^{\alpha} v_{\vec{w}}(Q)^{-\lambda} 
&\le C\varphi (Q)^{\alpha }\varphi (3Q)^{\frac{mp\theta \lambda }{q'}  
}v_{\vec{w}}\left(3 Q\right)^{-\lambda}||b_j||_{BMO_{\tilde{\theta}} (\varphi 
	)} \prod_{i=1}^{m}\left\|f_{i} 
\chi_{3Q}\right\|_{L^{p_i} (w_i)} \\
&\le  C \varphi (3Q)^{\alpha +\frac{mp \theta\lambda}{q'}  }
v_{\vec{w}}\left(3 Q\right)^{-\lambda} ||b_j||_{BMO_{\tilde{\theta}} (\varphi 
	)} \prod_{i=1}^{m}\left\|f_{i} 
\chi_{3Q}\right\|_{L^{p_i} (w_i)}.
\end{align*}
Note that $ \alpha \in (-\infty  ,\infty )  $, then
$$\varphi (Q)^{\alpha }\le \varphi (3^{k+1}Q)^{|\alpha| },$$
and thus
\begin{align*}
\frac{\varphi (Q)^{\alpha }}{ \varphi (3^{k+1}Q)^{\frac{N}{2q'}  }} \le 
\frac{\varphi (3^{k+1}Q)^{|\alpha| }}{\varphi (3^{k+1}Q)^{\frac{N}{2q'}  }}
=\frac{\varphi (3^{k+1}Q)^{\alpha }}{\varphi (3^{k+1}Q)^{\frac{N}{2q'} 
-|\alpha|+\alpha  }}.
\end{align*}
Taking $ N\ge N_{11}:=2q'(|\alpha |-\alpha ) $, we have
$$\frac{\varphi (Q)^{\alpha }}{ \varphi (3^{k+1}Q)^{\frac{N}{2q'}  }} \le C 
\varphi (3^{k+1}Q) ^{\alpha}.$$

Multiplying both sides of $ G_{2}^{'} $ by  $ \varphi(Q)^{\alpha} 
v_{\vec{w}}(Q)^{-\lambda} $, using Lemma \ref{Lemma3.3}, Lemma \ref{Lemma3.7}  
and Lemma \ref{Lemma3.8}, taking $N\ge \mathop{\max}\left \{ N_{11}, 
N_{12}:=2m\theta +2(\eta+ \tilde{s})q'/p \right \} $, we conclude
\begin{align*}
&G_{2}^{'} \varphi(Q)^{\alpha} v_{\vec{w}}(Q)^{-\lambda}\\
%&\le C \sum_{k=1}^{\infty }||b_j||_{BMO_{\tilde{\theta}} (\varphi 
%)}\varphi(3^{k+1}Q)
%^{\frac{m\theta }{q'} -\frac{N}{q'}+\frac{\tilde{s}}{p}  }\frac{v_{\vec 
%w}(Q)^{\frac{1}{p} }}{v_{\vec w}(3^{k+1}Q)^{\frac{1}{p} }}\varphi(Q)^{\alpha} 
%v_{\vec{w}}(Q)^{-\lambda} \prod_{i=1}^{m}\left\|f_{i} 
%\chi_{3^{k+1}Q}\right\|_{L^{p_i} (w_i)} \\
%&\le C\sum_{k=1}^{\infty }\frac{||b_j||_{BMO_{\tilde{\theta}} (\varphi 
%)}}{3^{nk\delta /p}}\varphi(3^{k+1}Q)
%^{\frac{m\theta }{q'} -\frac{N}{q'}+\frac{\tilde{s}}{p} +\frac{\eta }{p}  
%}\varphi(Q)^{\alpha} v_{\vec{w}}(Q)^{-\lambda} \prod_{i=1}^{m}\left\|f_{i} 
%\chi_{3^{k+1}Q}\right\|_{L^{p_i} (w_i)} \\
&\le C||b_j||_{BMO_{\tilde{\theta}} (\varphi 
	)}\sum_{k=1}^{\infty }\frac{1}{3^{nk\delta /p}}\varphi(3^{k+1}Q)
^{\frac{m\theta }{q'} -\frac{N}{q'}+\frac{\tilde{s}}{p} +\frac{\eta }{p}  
}\varphi (Q)^{\alpha }\varphi (3^{k+1}Q)^{\frac{mp \theta\lambda}{q'} } 
v_{\vec{w}}\left(3^{k+1} Q\right)^{-\lambda} \\
& \quad \times \prod_{i=1}^{m}\left\|f_{i} 
\chi_{3^{k+1}Q}\right\|_{L^{p_i} (w_i)} \\
&= C||b_j||_{BMO_{\tilde{\theta}} (\varphi 
	)}\sum_{k=1}^{\infty }\frac{1}{3^{nk\delta /p}}\varphi(3^{k+1}Q)
^{\frac{m\theta }{q'} -\frac{N}{2q'}+\frac{\tilde{s}}{p} +\frac{\eta }{p}  
}\frac{\varphi (Q)^{\alpha }}{\varphi (3^{k+1}Q)^{\frac{N}{2q'}  }} \varphi 
(3^{k+1}Q)^{\frac{mp \theta\lambda}{q'} } \\
& \quad \times v_{\vec{w}}\left(3^{k+1} 
Q\right)^{-\lambda}\prod_{i=1}^{m}\left\|f_{i} 
\chi_{3^{k+1}Q}\right\|_{L^{p_i} (w_i)} \\
&\le C||b_j||_{BMO_{\tilde{\theta}} (\varphi 
	)}\sum_{k=1}^{\infty }\frac{1}{3^{nk\delta /p}}\varphi (3^{k+1}Q)^{\alpha 
	+\frac{mp \theta\lambda}{q'} }
v_{\vec{w}}\left(3^{k+1} Q\right)^{-\lambda}\prod_{i=1}^{m}\left\|f_{i} 
\chi_{3^{k+1}Q}\right\|_{L^{p_i} (w_i)}.
\end{align*}

Using a method similar to the one described above, multiplying both sides of $ 
G_{3}^{'} $ by  $ \varphi(Q)^{\alpha} v_{\vec{w}}(Q)^{-\lambda} $, using Lemma 
\ref{Lemma3.3}, Lemma \ref{Lemma3.7},  and Lemma \ref{Lemma3.8}, taking $$ N\ge 
\mathop{\max}\left \{ N_{11}, N_{13}:=2m\theta +2( \tilde{\theta }+\eta /p )q' 
\right \},$$ we get
\begin{align*}
&G_{3}^{'} \varphi(Q)^{\alpha} v_{\vec{w}}(Q)^{-\lambda}\\
&\le C||b_j||_{BMO_{\tilde{\theta}} (\varphi 
	)}\sum_{k=1}^{\infty } \frac{k}{3^{nk\delta /p}} \varphi (3^{k+1}Q)^{\alpha 
	+\frac{mp 
		\theta\lambda}{q'} }
v_{\vec{w}}\left(3^{k+1} Q\right)^{-\lambda}\prod_{i=1}^{m}\left\|f_{i} 
\chi_{3^{k+1}Q}\right\|_{L^{p_i} (w_i)}.
\end{align*}

Multiplying both sides of $ G_{4}^{'} $ by  $ \varphi(Q)^{\alpha} 
v_{\vec{w}}(Q)^{-\lambda} $, using Lemma \ref{Lemma3.3}, Lemma \ref{Lemma3.7},  
and Lemma \ref{Lemma3.8}, taking $ N\ge \mathop{\max}\left \{ N_{11}, 
N_{14}:=2m\theta +2\tilde{s}+2\eta q' /p \right \} $, 
we obtain
\begin{align*}
&G_{4}^{'} \varphi(Q)^{\alpha} v_{\vec{w}}(Q)^{-\lambda}\\
&\le C||b_j||_{BMO_{\tilde{\theta}} (\varphi 
	)}\sum_{k=1}^{\infty } \frac{1}{3^{nk\delta /p}} \varphi (3^{k+1}Q)^{\alpha 
	+\frac{mp 
		\theta\lambda}{q'}  }
v_{\vec{w}}\left(3^{k+1} Q\right)^{-\lambda}\prod_{i=1}^{m}\left\|f_{i} 
\chi_{3^{k+1}Q}\right\|_{L^{p_i} (w_i)}.
\end{align*}
Taking the above estimate together, let $N\ge \mathop{\max} \left \{ N_{11}, 
N_{12},N_{13},N_{14}\right \}$, we obtain
\begin{align}\label{4.6} 
&\varphi(Q)^{\alpha} v_{\vec{w}}(Q)^{-\lambda} ||T_{b_j}^{j} (\vec f)\chi 
_{Q(z,r)}||_{L^{p}(v_{\vec w}) } \nonumber\\ 
& \le C||b_j||_{BMO_{\tilde{\theta}} 
	(\varphi )}\sum_{k=0}^{\infty } \frac{k+1}{3^{nk\delta /p}} \varphi 
	(3^{k+1}Q)^{\alpha +\frac{mp 
		\theta\lambda}{q'}  }
v_{\vec{w}}\left(3^{k+1} Q\right)^{-\lambda}\prod_{i=1}^{m}\left\|f_{i} 
\chi_{3^{k+1}Q}\right\|_{L^{p_i} (w_i)}.
\end{align}
Taking $L^{l}(\mathbb{R}^{n})$ norm of both sides of \eqref{4.6}, we 
deduce that
\begin{align}\label{4.7} 
&\left\| \varphi(Q)^{\alpha} v_{\vec{w}}(Q)^{-\lambda} ||T_{b_j}^{j} (\vec 
f)\chi _{Q(z,r)}||_{L^{p}(v_{\vec w}) }   
\right\|_{L^{l}\left(\mathbb{R}^{n}\right)} \nonumber\\
& \leq C ||b_j||_{BMO_{\tilde{\theta}} (\varphi 
	) }\sum_{k=0}^{\infty} \frac{k+1}{3^{n k \delta/ p}} \nonumber\\
& \quad \times\left\|\varphi\left(3^{k+1} 
Q\right)^{\alpha+\frac{mp \theta\lambda}{q'}} 
v_{\vec{w}}\left(3^{k+1} Q\right)^{-\lambda  }\prod_{i=1}^{m}|| f_{i} \chi_{3^{k+1} 
	Q}||_{L^{p_{i}}\left(w_{i}\right)}\right\|_{L^{l }\left(\mathbb{R}^{n}\right)}.
\end{align}
Since $ \beta =\alpha +mp\theta \lambda /q'$, the proof of Theorem \ref{Theorem2.3} is finished.
\end{proof}
\subsection{Proof of Theorem 2.4}
\begin{proof} We only present the case $ m=2 $ and $ \theta _1=\theta 
_2=\tilde{\theta}$ for simplicity. In general, the calculation is the same, 
just more complicated. Let $ Q:=Q(z,r) $ and $ Q^\ast :=3Q $. We split each 
$f_j$ as $f_j=f_j^0+f_j^\infty$, where $f_j^0=f_j \chi_{Q^*}$ and $ 
f_j^\infty=f_j-f_j^0 $. We deduce that for $x\in Q(z,r)$,
\begin{align*}
|T_{\prod \vec b}(f_1,f_2)(x)|
\le |T_{\prod \vec b}(f_{1}^{0},f_{2}^{0})(x)|+\sum_{(\alpha_1,\alpha_2)\in 
	\mathscr{L}}|T_{\prod \vec b}(f_{1}^{\alpha_1},f_{2}^{\alpha_2})(x) |
:= \mathrm{I}_1+\sum_{(\alpha_1,\alpha_2)\in 
	\mathscr{L}}\mathrm{I}_2^{\alpha _1,\alpha _2},
\end{align*}
where  $\mathscr{L}:=\{(\alpha_1,\alpha_2)$: there is at least one 
$\alpha_j=\infty\}$, $ j=1,2$.

Taking $ L^{p} (\upsilon _{\vec \omega }) $ norm on the cube $ Q(z,r) $ of $ 
\mathrm{I}_1$, by Theorem \ref{Theorem2.2}, it follows that
\begin{align*}
\left \| T_{\prod \vec b}(f_{1}^{0},f_{2}^{0})\chi _{Q(z,r)} \right \| 
_{L^{p}(\nu _{\vec w })}\le C\prod_{j=1}^{2}\left \| b_j \right 
\|_{BMO_{\tilde{\theta}} (\varphi ) } 
\prod_{i=1}^{2}\left\|f_i\chi_{Q(z,3r)}\right\|_{L^{p_{i}}\left(w_{i}\right)}.
\end{align*}
Multiplying both sides of the above inequality by  $ \varphi(Q)^{\alpha} 
v_{\vec{w}}(Q)^{-\lambda} $, using Lemma \ref{Lemma3.3} and Lemma 
\ref{Lemma3.8}, we obtain
\begin{align*}
&\varphi(Q)^{\alpha} v_{\vec{w}}(Q)^{-\lambda}\left \| T_{\prod \vec 
b}(f_{1}^{0},f_{2}^{0})\chi _{Q(z,r)} \right \| _{L^{p}(v _{\vec w })}\\
& \quad \le C\varphi(Q)^{\alpha } 
v_{\vec{w}}(3Q)^{-\lambda}\varphi(3Q)^{\frac{2p\theta \lambda }{q'} }
\prod_{j=1}^{2}\left \| b_j \right 
\|_{BMO_{\tilde{\theta}} (\varphi ) }  
\prod_{i=1}^{2}\left\| f_i\chi_{Q(z,3r)}\right\|_{L^{p_{i}}\left(w_{i}\right)}\\
& \quad \le C\varphi(3Q)^{\alpha +\frac{2p\theta \lambda }{q'} } 
v_{\vec{w}}(3Q)^{-\lambda}
\prod_{j=1}^{2}\left \| b_j \right 
\|_{BMO_{\tilde{\theta}} (\varphi ) }  
\prod_{i=1}^{2}\left\| f_i\chi_{Q(z,3r)}\right\|_{L^{p_{i}}\left(w_{i}\right)}.
\end{align*}
Applying the same method of the formula \eqref{4.7}, we can estimate the $ 
L^{l}(\mathbb{R}^{n}) $ norm of the above formula.

Let us estimate the term $ \mathrm{I}_2^{\alpha _1,\alpha _2} $. Taking $ \lambda _{j}:=b_{j,Q}, 
j=1,2$, we split $ T_{\prod \vec b}(f_{1}^{\alpha_1},f_{2}^{\alpha_2}) $ 
into four terms.
\begin{align*}
&|T_{\prod \vec b}(f_{1}^{\alpha_1},f_{2}^{\alpha_2})(x)|\\
&\le|b_1(x)-\lambda _1||b_2(x)-\lambda _2||T
(f_{1}^{\alpha_1},f_{2}^{\alpha_2})(x)|+|b_1(x)-\lambda 
_1||T(f_{1}^{\alpha_1},(\lambda _2-b_2)f_{2}^{\alpha_2})(x)|\\
& \quad +|b_2(x)-\lambda _2||T((\lambda 
_1-b_1)f_{1}^{\alpha_1},f_{2}^{\alpha_2})(x)|+|T((b_1-\lambda 
_1)f_{1}^{\alpha_1},(b_2-\lambda _2)f_{2}^{\alpha_2})(x)|\\
&:=\mathrm{I}_{21}^{\alpha _1,\alpha _2}+\mathrm{I}_{22}^{\alpha _1,\alpha _2}+\mathrm{I}_{23}^{\alpha _1,\alpha _2}+\mathrm{I}_{24}^{\alpha _1,\alpha _2}.
\end{align*}
To estimate $\mathrm{I}_{21}^{\alpha _1,\alpha _2}$, the following estimate can be obtained similar 
to \eqref{4.4}.
\begin{align*}
\left |  T\left(f_{1}^{\alpha_{1}}, f_{2}^{\alpha_{2}}\right)(x) \right |
\leq C\sum_{k=1}^{\infty} \frac{\varphi\left(3^{k+1} Q\right)^{\frac{2\theta 
		}{q'} -\frac{N}{q'} }}{{\left(\int_{3^{k+1} Q} v_{\vec{w}}\right)^{1 / 
		p}}} 
\prod_{i=1}^{2}\left\|f_{i} \chi_{3^{k+1} Q}\right\|_{L^{p_i} (w_i)}.
\end{align*}
Using H\"{o}lder's inequality and Lemma \ref{Lemma3.9}, we can deduce
\begin{align*}
&\left ( \int_{Q}|(b_1(x)-\lambda _1)(b_2(x)-\lambda _2)|^p v_{\vec w}(x) 
dx \right ) ^\frac{1}{p}\\
&\le \left ( \int_{Q}|b_1(x)-\lambda _1|^{2p} v_{\vec w}(x)
dx \right ) ^\frac{1}{2p}
\left ( \int_{Q}|b_2(x)-\lambda _2|^{2p} v_{\vec w}(x) dx \right 
) ^\frac{1}{2p} \\
%&=v_{\vec w}(Q)^{\frac{1}{2p} }\left ( \frac{1}{v_{\vec w}(Q)} 
%\int_{Q}|b_1(x)-\lambda _1|^{2p} v_{\vec w}(x)
%dx \right ) ^\frac{1}{2p} \\
%& \quad \times v_{\vec w}(Q)^{\frac{1}{2p} }\left ( \frac{1}{v_{\vec w}(Q)} 
%\int_{Q}|b_2(x)-\lambda _2|^{2p} v_{\vec w}(x)
%dx \right ) ^\frac{1}{2p}\\
&\le C v_{\vec w}(Q)^{\frac{1}{p} }\varphi (Q)^{\frac{\tilde{s}}{p} 
}\prod_{j=1}^{2}\left \| b_j \right \|_{BMO_{\tilde{\theta}} (\varphi ) }.
\end{align*}
Taking the above estimate together, we get
\begin{align*}
&\left ( \int_{Q}|\mathrm{I}_{21}^{\alpha _1,\alpha _2}|^p v_{\vec w}(x) dx \right ) ^\frac{1}{p} \\
&\le \left ( \int_{Q}|(b_1(x)-\lambda _1)(b_2(x)-\lambda _2)|^p v_{\vec 
w}(x) 
dx \right ) ^\frac{1}{p}
\mathop{\sup}\limits_{ x\in {\rm Q}}\left | 
T(f_{1}^{\alpha_1},f_{2}^{\alpha_2})(x) \right |\\
&\le Cv_{\vec w}(Q)^{\frac{1}{p} }\varphi (Q)^{\frac{\tilde{s}}{p} }
\prod_{j=1}^{2}\left \| b_j \right \|_{BMO_{\tilde{\theta}} 
	(\varphi ) }\sum_{k=1}^{\infty} \frac{\varphi\left(3^{k+1} 
	Q\right)^{\frac{2\theta 
		}{q'} -\frac{N}{q'} }}{{\left(\int_{3^{k+1} Q} v_{\vec{w}}\right)^{1 / 
		p}}}
 \prod_{i=1}^{2}\left\|f_{i} 
\chi_{3^{k+1} Q}\right\|_{L^{p_i} (w_i)}.
\end{align*}

Multiplying both sides of the above formula by $ \varphi(Q)^{\alpha} 
v_{\vec{w}}(Q)^{-\lambda}$, using Lemma \ref{Lemma3.3}, Lemma \ref{Lemma3.7} 
and  Lemma \ref{Lemma3.8}, taking $ N\ge \mathop{\max} \left \{ 2q'(|\alpha 
|-\alpha ),2q'(\tilde{s}+\eta)/p+4\theta\right \}$, we obtain
\begin{align}\label{4.8} 
&\varphi(Q)^{\alpha} v_{\vec{w}}(Q)^{-\lambda}||\mathrm{I}_{21}^{\alpha _1,\alpha _2} \chi 
_{Q(z,r)}||_{L^{p}(v_{\vec w})} \nonumber\\
& \quad \le C\prod_{j=1}^{2}\left \| b_j \right 
\|_{BMO_{\tilde{\theta}} (\varphi )}\sum_{k=1}^{\infty 
}\frac{\varphi(3^{k+1}Q)^{\alpha+\frac{2p\theta 
			\lambda}{q'}  }}{3^{nk\delta /p}}  
v_{\vec{w}}(3^{k+1}Q)^{-\lambda} \prod_{i=1}^{2}\left\| 
f_i\chi_{3^{k+1}Q}\right\|_{L^{p_{i}}\left(w_{i}\right)}.
\end{align}
 
Since $ \mathrm{I}_{22}^{\alpha _1,\alpha _2} $ and $\mathrm{I}_{23}^{\alpha _1,\alpha _2}$ are symmetrical, we only estimate $ \mathrm{I}_{22}^{\alpha _1,\alpha _2} $. The following estimates can be obtained similar 
to the formula \eqref{4.5}. We deduce that
\begin{align*}
&\left |  T\left(f_{1}^{\alpha_{1}}, (\lambda 
_2-b_2)f_{2}^{\alpha_{2}}\right)(x) \right | \\
&\leq C \sum_{k=1}^{\infty} \frac{\varphi\left(3^{k+1} 
	Q\right)^{\frac{2\theta }{q'} -\frac{N}{q'} 
	}|b_{2,3^{k+1}Q}-b_{2,Q}|}{{{\left(\int_{3^{k+1} Q} 
			v_{\vec{w}}\right)^{1 / p}}}}\prod_{i=1}^{2}\left\|f_{i} 
			\chi_{3^{k+1}Q}\right\|_{L^{p_i} 
	(w_i)} \\
&\quad + C ||b_2||_{BMO_{\tilde{\theta}} (\varphi )}\sum_{k=1}^{\infty} 
\varphi\left(3^{k+1} 
Q\right)^{\frac{2\theta }{q'} -\frac{N}{q'} }\frac{\varphi \left ( 3^{k+1} 
Q \right ) ^{\frac{\tilde{s}(p_2-q')}{p_2q'} }}{{\left(\int_{3^{k+1} Q} 
v_{\vec{w}}\right)^{1 / p}}} 
\prod_{i=1}^{2}\left\|f_{i} \chi_{3^{k+1}Q}\right\|_{L^{p_i} 
(w_i)}. 
\end{align*}
Applying Lemma \ref{Lemma3.9}, we have
\begin{align*}
\left ( \int_{Q}|b_1(x)-\lambda _1|^p v_{\vec w}(x) dx \right ) ^\frac{1}{p}
\le Cv_{\vec w}(Q)^{\frac{1}{p} }\varphi (Q)^{\frac{\tilde{s}}{p} }\left \| 
b_1 \right \|_{BMO_{\tilde{\theta}} (\varphi )}.
\end{align*}
Taking the above estimate together, we get
\begin{align*}
&\left ( \int_{Q}|\mathrm{I}_{22}^{\alpha _1,\alpha _2}|^p v_{\vec w}(x) dx \right ) ^\frac{1}{p} \\
&\le \left ( \int_{Q}|b_1(x)-\lambda _1|^p v_{\vec w}(x) 
dx \right ) ^\frac{1}{p}\mathop{\sup}\limits_{ x\in {\rm Q}}\left | 
T(f_{1}^{\alpha_1},(\lambda _2-b_2)f_{2}^{\alpha_2})(x) \right |\\
&\le C \left \| b_1 \right \|_{BMO_{\tilde{\theta}} 
	(\varphi )}\sum_{k=1}^{\infty}\varphi (Q)^{\frac{\tilde{s}}{p} 
} \varphi\left(3^{k+1} 
Q\right)^{\frac{2\theta }{q'} -\frac{N}{q'} 
}|b_{2,3^{k+1}Q}-b_{2,Q}| \\
& \quad \times \frac{v_{\vec w}(Q)^{\frac{1}{p} } }{v_{\vec 
		w}(3^{k+1}Q)^{\frac{1}{p} }}\prod_{i=1}^{2}\left\|f_{i} 
\chi_{3^{k+1}Q}\right\|_{L^{p_i} 
	(w_i)} \\
& \quad+C\prod_{j=1}^{2}\left \| b_j \right \|_{BMO_{\tilde{\theta}} 
	(\varphi )}\sum_{k=1}^{\infty}\varphi (Q)^{\frac{\tilde{s}}{p} 
	}\varphi\left(3^{k+1} 
Q\right)^{\frac{2\theta }{q'}-\frac{N}{q'}+\frac{\tilde{s}(p_{2}-q')}{p_{2}q'}  
}\\
& \quad \times \frac{v_{\vec w}(Q)^{\frac{1}{p} } }{v_{\vec 
w}(3^{k+1}Q)^{\frac{1}{p} }}\prod_{i=1}^{2}\left\|f_{i} 
\chi_{3^{k+1} Q}\right\|_{L^{p_i} (w_i)}.
\end{align*}

Multiplying both sides of the above formula by $ \varphi(Q)^{\alpha} 
v_{\vec{w}}(Q)^{-\lambda}$, using Lemma \ref{Lemma3.3}, Lemma \ref{Lemma3.7} 
and  Lemma \ref{Lemma3.8}, taking $$N\ge \mathop{\max} \left \{ 2q'(|\alpha 
|-\alpha ),2(\tilde{s}+\eta )q'/p+4\theta +2\tilde{s},
2(\tilde{s}+\eta )q'/p+4\theta +2\tilde{\theta }q'\right \},$$ we obtain
\begin{align*}
&\varphi(Q)^{\alpha} v_{\vec{w}}(Q)^{-\lambda}||\mathrm{I}_{22}^{\alpha _1,\alpha _2} \chi 
_{Q(z,r)}||_{L^{p}(v_{\vec w})} \\
& \quad \le C\prod_{j=1}^{2}\left \| b_j \right 
\|_{BMO_{\tilde{\theta}} (\varphi )} \sum_{k=1}^{\infty }\frac{k}{3^{nk\delta 
/p}}  \varphi(3^{k+1}Q)^{\alpha+\frac{2p\theta \lambda}{q'}  } 
v_{\vec{w}}(3^{k+1}Q)^{-\lambda}\prod_{i=1}^{2}\left\| 
f_i\chi_{3^{k+1}Q}\right\|_{L^{p_{i}}\left(w_{i}\right)}.
\end{align*}

For the term $ \mathrm{I}_{24}^{\alpha _1,\alpha _2} $, $|T((b_1-\lambda _1)f_{1}^{\alpha_{1}}, 
(b_2-\lambda _2)f_{2}^{\alpha_{2}})(x) |$ is estimated similarly to the formula 
\eqref{4.4} and \eqref{4.5}. Since 
\begin{align*}
&(b_1(y_1)-\lambda _1)(b_2(y_2)-\lambda _2)\\
&=(b_1(y_1)-b_{1,3^{k+1}Q})(b_2(y_2)-b_{2,3^{k+1}Q})
+(b_1(y_1)-b_{1,3^{k+1}Q})(b_{2,3^{k+1}Q}-b_{2,Q})\\
&\quad +(b_{1,3^{k+1}Q}-b_{1,Q})(b_2(y_2)-b_{2,3^{k+1}Q})
+(b_{1,3^{k+1}Q}-b_{1,Q})(b_{2,3^{k+1}Q}-b_{2,Q}),
\end{align*}
 we deduce that
\begin{align*}
&|  T((b_1-\lambda _1)f_{1}^{\alpha_{1}}, (b_2-\lambda _2)f_{2}^
{\alpha_{2}})(x) |\\
%& \leq |  T((b_1-b_{1, 3^{k+1}Q})f_{1}^{\alpha_{1}}, (b_2-b_{2, 
%3^{k+1}Q})f_{2}^{\alpha_{2}})(x) | \\
%& \quad +|  T((b_1-b_{1, 3^{k+1}Q})f_{1}^{\alpha_{1}}, (b_{2, 3^{k+1}Q}-b_{2, 
%Q})f_{2}^{\alpha_{2}})(x) | \\
%&\quad+|  T((b_{1, 3^{k+1}Q}-b_{1, Q})f_{1}^{\alpha_{1}}, (b_2-b_{2, 
%3^{k+1}Q})f_{2}^{\alpha_{2}})(x) | \\
%&\quad+|  T((b_{1, 3^{k+1}Q}-b_{1, Q})f_{1}^{\alpha_{1}}, (b_{2, 
%3^{k+1}Q}-b_{2, Q})f_{2}^{\alpha_{2}})(x) | \\
& \leq C \sum_{k=1}^{\infty} \frac{1}{\left | 3^{k+1}Q \right |^{2}  
	\varphi\left(3^{k+1} Q\right)^{N}}
\int_{\left(3^{k+1} Q\right)^{2}} \prod_{j=1}^{2} 
\left|(b_j(y_j)-b_{j,3^{k+1} Q})f_{j}\left(y_{j}\right)\right|d \vec{y} \\
& \quad+ C \sum_{k=1}^{\infty} \frac{|b_{2,3^{k+1}Q}-b_{2,Q}|}{\left | 3^{k+1}Q 
\right |^{2}  
	\varphi\left(3^{k+1} Q\right)^{N}}\int_{\left(3^{k+1} Q\right)^{2}} 
\left|(b_1(y_1)-b_{1,3^{k+1} Q})f_{1}\left(y_{1}\right)\right|\left | 
f_{2}\left(y_{2}\right) \right | d \vec{y} \\
& \quad+C \sum_{k=1}^{\infty} \frac{|b_{1,3^{k+1}Q}-b_{1,Q}|}{\left | 3^{k+1}Q 
\right |^{2}  
	\varphi\left(3^{k+1} Q\right)^{N}}\int_{\left(3^{k+1} Q\right)^{2}} 
\left|(b_2(y_2)-b_{2,3^{k+1} Q})f_{2}\left(y_{2}\right)\right|\left | 
f_{1}\left(y_{1}\right) \right | d \vec{y} \\
& \quad+ C \sum_{k=1}^{\infty} \frac{\varphi\left(3^{k+1} 
	Q\right)^{\frac{2\theta }{q'} -\frac{N}{q'} }\cdot 
	|b_{1,3^{k+1}Q}-b_{1,Q}||b_{2,3^{k+1}Q}-b_{2,Q}|}{{\left(\int_{3^{k+1} Q} 
		v_{\vec{w}}\right)^{1 / p}}}\prod_{i=1}^{2}\left\|f_{i} \chi_{3^{k+1} 
	Q}\right\|_{L^{p_i} (w_i)}\\
&\le C\prod_{j=1}^{2}\left \| b_j \right \|_{BMO_{\tilde{\theta}} 
	(\varphi )}\sum_{k=1}^{\infty } 
\frac{\varphi\left(3^{k+1} 
	Q\right)^{\frac{2\theta }{q'} 
		-\frac{N}{q'}+\frac{\tilde{s}(p_1-q')}{p_1q'}+\frac{\tilde{s}(p_2-q')}{p_2q'}}
		 }{v_{\vec 
		w}(3^{k+1}Q)^{\frac{1}{p} }}\prod_{i=1}^{2}\left\|f_{i} \chi_{3^{k+1} 
	Q}\right\|_{L^{p_i} (w_i)}\\
&\quad + C \prod_{j=1}^{2}\left \| b_j \right \|_{BMO_{\tilde{\theta}} 
	(\varphi )}\sum_{k=1}^{\infty }  \frac{k\varphi\left(3^{k+1} 
	Q\right)^{\tilde{\theta }+\frac{2\theta }{q'} -\frac{N}{q'} 
		+\frac{\tilde{s}(p_{1}-q')}{p_{1}q'}}}{v_{\vec 
		w}(3^{k+1}Q)^{\frac{1}{p} }}\prod_{i=1}^{2}\left\|f_{i} \chi_{3^{k+1} 
	Q}\right\|_{L^{p_i} (w_i)} \\
&\quad + C\prod_{j=1}^{2}\left \| b_j \right \|_{BMO_{\tilde{\theta}} 
	(\varphi )}\sum_{k=1}^{\infty } \frac{k\varphi\left(3^{k+1} 
	Q\right)^{\tilde{\theta }+\frac{2\theta }{q'} -\frac{N}{q'} 
		+\frac{\tilde{s}(p_2-q')}{p_2q'}}}{v_{\vec 
		w}(3^{k+1}Q)^{\frac{1}{p} }}\prod_{i=1}^{2}\left\|f_{i} \chi_{3^{k+1} 
	Q}\right\|_{L^{p_i} (w_i)} \\
&\quad + C\prod_{j=1}^{2}\left \| b_j \right \|_{BMO_{\tilde{\theta}} 
	(\varphi )}\sum_{k=1}^{\infty } 
\frac{k^{2}\varphi\left(3^{k+1} 
	Q\right)^{\frac{2\theta }{q'} -\frac{N}{q'} +2  \tilde{\theta} }}{v_{\vec 
		w}(3^{k+1}Q)^{\frac{1}{p} }}\prod_{i=1}^{2}\left\|f_{i} \chi_{3^{k+1} 
	Q}\right\|_{L^{p_i} (w_i)}.
%C \sum_{k=1}^{\infty} \frac{1}{\left | 3^{k+1}Q \right |^{2}  
%\varphi\left(3^{k+1} Q\right)^{N}}
% \int_{\left(3^{k+1} Q\right)^{2}} \prod_{j=1}^{2} 
%\left|(b_j(y_j)-\lambda_j)f_{j}\left(y_{j}\right)\right|d \vec{y} \\
%&  \le C\prod_{j=1}^{2}\left \| b_j \right \|_{BMO_{\tilde{\theta}} 
%	(\varphi )
%}\sum_{k=1}^{\infty } \frac{\varphi\left(3^{k+1} 
%	Q\right)^{\frac{2\theta }{q'} -\frac{N}{q'}+\frac{\tilde{s}(p_1-q')}{p_1q'} 
%		+ \frac{\tilde{s}(p_2-q')}{p_{2}q'} }}{\left(\int_{3^{k+1} Q} 
%	v_{\vec{w}}\right)^{1 / p}} 
%\prod_{i=1}^{2}\left\|f_{i} \chi_{3^{k+1} Q}\right\|_{L^{p_i} (w_i)}\\
%& =C\prod_{j=1}^{2}\left \| b_j \right \|_{BMO_{\tilde{\theta}} 
%	(\varphi )}\sum_{k=1}^{\infty } \varphi\left(3^{k+1} 
%Q\right)^{\frac{2\theta }{q'} 
%-\frac{N}{q'}+\frac{2\tilde{s}}{q'}-\frac{\tilde{s}}{p}   
%} \left(\int_{3^{k+1} Q} 
%v_{\vec{w}}(y)dy\right)^{-\frac{1}{p} }
%\prod_{i=1}^{2}\left\|f_{i} \chi_{3^{k+1} Q}\right\|_{L^{p_i} (w_i)}.
\end{align*}
Combining the above estimates, we conclude that
\begin{align*}
&\left ( \int_{Q}|\mathrm{I}_{24}^{\alpha _1,\alpha _2}|^p v_{\vec w}(x) dx \right ) ^\frac{1}{p} \\
&\le \left ( \int_{Q} v_{\vec w}(x) 
dx \right ) ^\frac{1}{p}\mathop{\sup}\limits_{ x\in {\rm Q}}\left | 
T((b_1-\lambda _1)f_{1}^{\alpha_1},(b_2-\lambda _2)f_{2}^{\alpha_2})(x) \right 
| \\
&\le C\prod_{j=1}^{2}\left \| b_j \right \|_{BMO_{\tilde{\theta}} 
	(\varphi )}\sum_{k=1}^{\infty } \varphi\left(3^{k+1} 
Q\right)^{\frac{2\theta }{q'} 
	-\frac{N}{q'}+\frac{\tilde{s}(p_1-q')}{p_1q'}+\frac{\tilde{s}(p_2-q')}{p_2q'}
}
 \frac{v_{\vec w}(Q)^{\frac{1}{p} } }{v_{\vec 
		w}(3^{k+1}Q)^{\frac{1}{p} }}\prod_{i=1}^{2}\left\|f_{i} \chi_{3^{k+1} 
		Q}\right\|_{L^{p_i} (w_i)} \\
&\quad + C \prod_{j=1}^{2}\left \| b_j \right \|_{BMO_{\tilde{\theta}} 
	(\varphi )}\sum_{k=1}^{\infty } k\varphi\left(3^{k+1} 
Q\right)^{\tilde{\theta }+\frac{2\theta }{q'} -\frac{N}{q'} 
+\frac{\tilde{s}(p_1-q')}{p_1q'}   
}
 \frac{v_{\vec w}(Q)^{\frac{1}{p} } }{v_{\vec 
		w}(3^{k+1}Q)^{\frac{1}{p} }}\prod_{i=1}^{2}\left\|f_{i} \chi_{3^{k+1} 
		Q}\right\|_{L^{p_i} (w_i)}\\
&\quad + C\prod_{j=1}^{2}\left \| b_j \right \|_{BMO_{\tilde{\theta}} 
	(\varphi )}\sum_{k=1}^{\infty } k\varphi\left(3^{k+1} 
Q\right)^{\tilde{\theta }+\frac{2\theta }{q'} -\frac{N}{q'} 
+\frac{\tilde{s}(p_2-q')}{p_2q'}   
}
\frac{v_{\vec w}(Q)^{\frac{1}{p} } }{v_{\vec 
		w}(3^{k+1}Q)^{\frac{1}{p} }}\prod_{i=1}^{2}\left\|f_{i} \chi_{3^{k+1} 
		Q}\right\|_{L^{p_i} (w_i)}\\
&\quad + C\prod_{j=1}^{2}\left \| b_j \right \|_{BMO_{\tilde{\theta}} 
	(\varphi )}\sum_{k=1}^{\infty } k^{2}\varphi\left(3^{k+1} 
Q\right)^{\frac{2\theta }{q'} -\frac{N}{q'} +2  \tilde{\theta} 
}
 \frac{v_{\vec w}(Q)^{\frac{1}{p} } }{v_{\vec 
w}(3^{k+1}Q)^{\frac{1}{p} }}\prod_{i=1}^{2}\left\|f_{i} \chi_{3^{k+1} 
Q}\right\|_{L^{p_i} (w_i)}.
\end{align*}

Multiplying both sides of the above formula by $ \varphi(Q)^{\alpha} 
v_{\vec{w}}(Q)^{-\lambda}$, using Lemma \ref{Lemma3.3} and  Lemma 
\ref{Lemma3.8}, taking $$N\ge \mathop{\max} \left \{ 2q'(|\alpha |-\alpha 
),4\theta +4\tilde{s}+\frac{2(\eta -\tilde{s})q'}{p} ,4\theta +2\tilde{\theta 
}q'+2\tilde{s}+\frac{2\eta q'}{p} , 4\theta +4\tilde{\theta }q'+\frac{2\eta 
q'}{p} \right \},$$ we have
\begin{align*}
&\varphi(Q)^{\alpha} v_{\vec{w}}(Q)^{-\lambda}||\mathrm{I}_{24}^{\alpha _1,\alpha _2} \chi 
_{Q(z,r)}||_{L^{p}(v_{\vec w})} \\
& \quad \le C\prod_{j=1}^{2}\left \| b_j \right 
\|_{BMO_{\tilde{\theta}} (\varphi )} \sum_{k=1}^{\infty }\frac{k^2}{3^{nk\delta 
		/p}}  \varphi(3^{k+1}Q)^{\alpha+\frac{2p\theta \lambda}{q'}  } 
v_{\vec{w}}(3^{k+1}Q)^{-\lambda}\prod_{i=1}^{2}\left\| 
f_i\chi_{3^{k+1}Q}\right\|_{L^{p_{i}}\left(w_{i}\right)}.
\end{align*}
Summing up the above estimates, applying the same method of the formula \eqref{4.7}, we have
\begin{align*}
	&\left\| \varphi(Q)^{\alpha} v_{\vec{w}}(Q)^{-\lambda} ||T_{\prod \vec b}(\vec f)\chi _{Q(z,r)}||_{L^{p}(v_{\vec w}) }   
	\right\|_{L^{l}\left(\mathbb{R}^{n}\right)} \\
	& \leq C\prod_{j=1}^{2}\left \| b_j \right 
	\|_{BMO_{\tilde{\theta}} (\varphi )} \sum_{k=0}^{\infty }\frac{(k+1)^2}{3^{nk\delta 
			/p}}  \left \| \varphi(3^{k+1}Q)^{\alpha+\frac{2p\theta \lambda}{q'}  } 
	v_{\vec{w}}(3^{k+1}Q)^{-\lambda}\prod_{i=1}^{2}\left\| 
	f_i\chi_{3^{k+1}Q}\right\|_{L^{p_{i}}\left(w_{i}\right)} \right \| _{L^{l}\left(\mathbb{R}^{n}\right)}.
\end{align*}
Since $ \beta =\alpha +2p\theta \lambda /q'$, the proof of Theorem \ref{Theorem2.4} is finished.
\end{proof}

\bigskip

\noindent Chunliang Li, Shuhui Yang and Yan Lin    

\smallskip

\noindent School of Science, China University of Mining and
Technology, Beijing 100083,  People's Republic of China

\smallskip

\noindent{\it E-mails:} \texttt{lichunliang@student.cumtb.edu.cn} (C. 
Li)

\noindent\phantom{{\it E-mails:} }\texttt{yangshuhui@student.cumtb.edu.cn} (S. Yang)

\noindent\phantom{{\it E-mails:} }\texttt{linyan@cumtb.edu.cn} (Y. Lin)

\end{document}